\documentclass[reqno,a4paper]{article}
\usepackage{etex}
\usepackage[T1]{fontenc}
\usepackage[utf8]{inputenc}
\usepackage{lmodern}

\usepackage[style=alphabetic,firstinits=true,backref,backend=biber,isbn=false,url=false,maxbibnames=99]{biblatex}

\renewbibmacro{in:}{}
\newbibmacro{string+doi}[1]{\iffieldundef{doi}{\iffieldundef{url}{#1}{\href{\thefield{url}}{#1}}}{\href{http://dx.doi.org/\thefield{doi}}{#1}}}
\DeclareFieldFormat*{title}{\usebibmacro{string+doi}{\emph{#1}}}
\DefineBibliographyStrings{english}{%
  backrefpage = {$\uparrow$},
  backrefpages = {$\uparrow$},
}

\usepackage{amssymb,amsmath,amsthm}
\usepackage{units}
\usepackage{graphicx}
\usepackage[all]{xy}
\usepackage{tikz}
\usetikzlibrary{arrows,calc,positioning,decorations.pathreplacing}
\usepackage{relsize}
\usepackage{mhsetup}
\usepackage{mathtools}
\usepackage{stmaryrd}
\usepackage{tocloft}
\usepackage{paralist} % For compactitem environment
\usepackage[left=3cm,right=3cm,top=3cm,bottom=3cm]{geometry} % Decrease the margins %
\usepackage[bottom]{footmisc}
\usepackage[colorlinks,citecolor=blue,linkcolor=blue,urlcolor=blue,filecolor=blue,bookmarksdepth=2,breaklinks]{hyperref}
\usepackage{footnotebackref}
\usepackage[format=plain,indention=1em,labelsep=quad,labelfont={up},textfont={footnotesize},margin=2em]{caption}
\usepackage[mathscr]{euscript}
\usepackage{accents}
\usepackage[textsize=scriptsize]{todonotes}

\definecolor{lightblue}{rgb}{0.8,0.8,1}

\numberwithin{equation}{section}
\numberwithin{figure}{section}

\newtheoremstyle{italicised}
        {\topsep}{\topsep}  % space above % space below
        {\itshape}  % body font
        {}  % indent (empty = no indent)
        {\bfseries}  % header font
        {}  % punctuation after header
        {1ex}  % space after header
        {}  % manually specify the header
\theoremstyle{italicised}
\newtheorem{thm}{Theorem}[section]
\newtheorem{lem}[thm]{Lemma}
\newtheorem{sublem}[thm]{Sublemma}
\newtheorem{prop}[thm]{Proposition}
\newtheorem{coro}[thm]{Corollary}

\newtheorem{fact}[thm]{Fact}
\newtheorem{athm}{Theorem}

\newtheoremstyle{upright}
        {\topsep}{\topsep}  % space above % space below
        {\upshape}  % body font
        {}  % indent (empty = no indent)
        {\bfseries}  % header font
        {}  % punctuation after header
        {1ex}  % space after header
        {}  % manually specify the header
\theoremstyle{upright}
\newtheorem{defn}[thm]{Definition}
\newtheorem{rmk}[thm]{Remark}
\newtheorem{recap}[thm]{Recap}
\newtheorem{eg}[thm]{Example}

\newtheorem{notation}[thm]{Notation}
\newtheorem{assumption}[thm]{Assumption}
\newtheorem{convention}[thm]{Convention}

\newtheoremstyle{italicised-restate}
        {\topsep}{\topsep}  % space above % space below
        {\itshape}  % body font
        {}  % indent (empty = no indent)
        {\bfseries}  % header font
        {}  % punctuation after header
        {1ex}  % space after header
        {\thmname{#1}\thmnote{ \bfseries #3}}  % manually specify the header
\theoremstyle{italicised-restate}

\makeatletter
\renewcommand*{\@seccntformat}[1]{\upshape\csname the#1\endcsname.\hspace{1ex}}
\renewcommand*{\section}{\@startsection{section}{1}{\z@}%
	{2.5ex \@plus 1ex \@minus 0.2ex}%
	{1.5ex \@plus 0.2ex}%
	{\normalfont\large\bfseries}}
\renewcommand*{\subsection}{\@startsection{subsection}{2}{\z@}%
	{2.5ex \@plus 1ex \@minus 0.2ex}%
	{-1.5ex \@plus -0.2ex}%
	{\normalfont\normalsize\bfseries}}
\renewcommand*{\subsubsection}{\@startsection{subsubsection}{3}{\z@}%
	{2.5ex \@plus 1ex \@minus 0.2ex}%
	{-1.5ex \@plus -0.2ex}%
	{\normalfont\normalsize\bfseries}}
\renewcommand*{\paragraph}{\@startsection{paragraph}{4}{\z@}%
	{2.5ex \@plus 1ex \@minus 0.2ex}%
	{-1.5ex \@plus -0.2ex}%
	{\normalfont\normalsize\bfseries}}
\renewcommand*{\subparagraph}{\@startsection{subparagraph}{5}{\z@}%
	{2.5ex \@plus 1ex \@minus 0.2ex}%
	{-1.5ex \@plus -0.2ex}%
	{\normalfont\normalsize\slshape}}
\makeatother

\renewcommand{\qedsymbol}{\qedsquareb}

\newcommand{\lhto}{\lhook\joinrel\longrightarrow}
\newcommand{\xbar}{\ensuremath{{\,\overline{\! X}}}}
\newcommand{\ybar}{\ensuremath{{\overline{Y}}}}
\newcommand{\xbarp}{\ensuremath{\smash{\,\overline{\! X}}\vphantom{X}^{\prime}}}
\newcommand{\ybarp}{\ensuremath{\smash{\overline{Y}}\vphantom{Y}^{\prime}}}
\newcommand{\abar}{\ensuremath{{\,\overline{\! A}}}}
\newcommand{\bbar}{\ensuremath{{\,\overline{\! B}}}}
\newcommand{\iotab}{\ensuremath{\bar{\iota}}}
\newcommand{\iotabb}{\ensuremath{\hat{\iota}}}

\newcommand{\Emb}{\mathrm{Emb}}

\newcommand{\hconn}{\ensuremath{h\mathrm{conn}}}

\newcommand{\hofib}{\mathrm{hofib}}
\newcommand{\fib}{\mathrm{fib}}
\newcommand{\Aut}{\mathrm{Aut}}

\newcommand{\Diff}{\mathrm{Diff}}

\newcommand{\cref}{\ensuremath{\mathit{cr}}}
\newcommand{\Map}{\ensuremath{\mathrm{Map}}}

\newcommand{\geomr}[1]{\lVert #1 \rVert}
\newcommand{\mbar}{\ensuremath{{\,\,\overline{\!\! M\!}\,}}}

\newcommand{\incl}[3][right]%
{%
	\draw[<-,>=#1 hook] #2 to ($ #2!0.5!#3 $);
	\draw[->] ($ #2!0.5!#3 $) to #3;%
}
\newcommand{\inclusion}[5][right]%
{%
	\draw[<-,>=#1 hook] #4 to ($ #4!0.5!#5 $) node[#2,font=\small]{#3};
	\draw[->,>=stealth'] ($ #4!0.5!#5 $) to #5;%
}

\newenvironment{itemizeb}%
{\begin{compactitem}

}%
{\end{compactitem}}

\newcommand{\cC}{\mathcal{C}}

\newcommand{\cE}{\mathcal{E}}

\newcommand{\cH}{\mathcal{H}}
\newcommand{\cI}{\mathcal{I}}

\newcommand{\cU}{\mathcal{U}}
\newcommand{\cV}{\mathcal{V}}
\newcommand{\cW}{\mathcal{W}}

\newcommand{\cY}{\mathcal{Y}}

\newcommand{\bN}{\mathbb{N}}

\newcommand{\bR}{\mathbb{R}}

\newcommand{\bZ}{\mathbb{Z}}

\renewcommand{\geq}{\geqslant}
\renewcommand{\leq}{\leqslant}
\renewcommand{\footnoterule}{%
  \kern -3pt
  \hrule width \textwidth height 0.4pt
  \kern 2.6pt
}
\newcommand{\colim}{\ensuremath{\mathrm{colim}}}
\newcommand{\X}{\ensuremath{\mathscr{X}}}
\newcommand{\doubleDelta}{\ensuremath{\Delta\!\!\!\!\Delta}}
\newcommand{\anglenumber}[1]{{$\langle\,\text{#1}\,\rangle$}}
\newcommand{\purebraid}{\ensuremath{P\!\beta}}
\newcommand{\cf}{\textit{cf}.\ }

\usepackage{fancyhdr}
\usepackage{lastpage}
\fancypagestyle{plain}{%
\lhead{}
\chead{}
\rhead{}
\lfoot{}
\cfoot{p.\ \thepage\ $/\!\!/$ \pageref*{LastPage}}
\rfoot{}

}

\begin{document}
\title{\Large\bfseries Homological stability for moduli spaces of disconnected submanifolds, I\vspace{-1ex}}
\author{\small Martin Palmer\quad $/\!\!/$\quad 29\textsuperscript{th} April 2020\vspace{-1ex}}
\date{}
\maketitle
{
\makeatletter
\renewcommand*{\BHFN@OldMakefntext}{}
\makeatother
\footnotetext{2010 \textit{Mathematics Subject Classification}: Primary 55R80, 57S05; Secondary 57N20, 58B05.}
\footnotetext{\textit{Key words}: Moduli spaces of submanifolds, homological stability, embedding spaces, configuration spaces.}
\footnotetext{[---Also available at \href{https://mdp.ac/papers/hsfmsods1}{mdp.ac/papers/hsfmsods1}, where any addenda or informal related notes will also be posted.---]}
}
\begin{abstract}
A well-known property of unordered configuration spaces of points (in an open, connected manifold) is that their homology \emph{stabilises} as the number of points increases. We generalise this result to moduli spaces of submanifolds of higher dimension, where stability is with respect to the number of components having a fixed diffeomorphism type and isotopy class. As well as for unparametrised submanifolds, we prove this also for partially-parametrised submanifolds -- where a \emph{partial parametrisation} may be thought of as a superposition of parametrisations related by a fixed subgroup of the mapping class group.

In a companion paper (\cite{Palmer2018-hs-msods-II}) this is further generalised to submanifolds equipped with labels in a bundle over the embedding space, from which we deduce corollaries for the stability of diffeomorphism groups of manifolds with respect to parametrised connected sum and addition of singularities.
\end{abstract}
\tableofcontents

%%%%%%%%%%%%%%%%%%%%%%%%%%%%%%%%%%%%%%%%%%%%%%%%%%%%%%%%%
%%%%%%%%%%%%%%%%%%%%%%%%%%%%%%%%%%%%%%%%%%%%%%%%%%%%%%%%%
\section{Introduction}

Let $M$ be an open connected manifold of dimension at least $2$. The configuration space of $n$ points in $M$ is defined to be $C_n(M) = (M^n \smallsetminus \doubleDelta)/\Sigma_n$, where $\doubleDelta$ is the ``collision set'' of $M^n$
\[
\doubleDelta = \{ (p_1,\dotsc,p_n)\in M^n \;|\; p_i = p_j \text{ for some } i\neq j \},
\]
and $\Sigma_n$ denotes the symmetric group on $n$ letters. Alternatively, it can be written as $C_n(M) = \Emb(n,M)/\Sigma_n$, where $\Emb(n,M)$ denotes the space of embeddings of the zero-dimensional manifold $n$ into $M$. More generally, one can define a labelled configuration space, for a given label space $X$, by $C_n(M,X) = \Emb(n,M) \times_{\Sigma_n} X^n$. An important property of these spaces is that they satisfy homological stability as $n$ varies:

\begin{thm}[{\cite{Segal1973Configurationspacesand, McDuff1975Configurationspacesof, Segal1979topologyofspaces, Randal-Williams2013Homologicalstabilityunordered}}]\label{tClassical}
There is a natural map $C_n(M,X)\to C_{n+1}(M,X)$ which is split-injective on homology in all degrees, and if $X$ is path-connected it induces isomorphisms on homology up to degree $\frac{n}{2}$.
\end{thm}

A natural question to consider is what happens for higher-dimensional embedded submanifolds. Certain special cases have been considered before, up to dimension $2$ (embedded circles and surfaces) -- see \hyperref[para:related-results]{\S\anglenumber{v.}} for a brief overview. The main theorem of this paper is a general stability result for moduli spaces of disconnected submanifolds of a fixed diffeomorphism type and isotopy class, where stabilisation occurs by adding new connected components to the submanifold.

\paragraph{\anglenumber{i.} Unparametrised submanifolds.}
\addcontentsline{toc}{subsection}{Unparametrised submanifolds.}

Let $P$ be a closed manifold. A first guess for the natural analogue of $C_n(M)$ in the case of embedded copies of $P$ may be
\[
\mathbf{C}_{nP}(M) = \Emb(nP,M) / (\mathrm{Diff}(P) \wr \Sigma_n).
\]
The wreath product $\mathrm{Diff}(P) \wr \Sigma_n$ is naturally a subgroup of $\mathrm{Diff}(nP)$, where $nP$ is shorthand for $\{1,\ldots,n\} \times P$, and they are equal if $P$ is connected. However, in general (consider for example $M=\bR^3$ and $P=S^1$) this has a diverging number of path-components as $n\to\infty$, so it will not satisfy homological stability even in degree zero. Instead, we study the path-components of this space separately. Assume that $M$ is the interior of a manifold with boundary $\mbar$, and choose:
\begin{itemizeb}
\item[$\circ$] a self-embedding $e \colon \mbar \hookrightarrow \mbar$ that is isotopic to the identity,
\item[$\circ$] an embedding $\iota \colon P \hookrightarrow \partial\mbar$ such that $\iota(P) \cap e(\mbar) = \varnothing$ and $e(\iota(P)) \subseteq M$.
\end{itemizeb}
Then we may define a map
\begin{equation}\label{eStabMapIntro}
\mathbf{C}_{nP}(M) \longrightarrow \mathbf{C}_{(n+1)P}(M)
\end{equation}
by pushing a given collection of submanifolds inwards along $e$ and adjoining a new copy of $P$ near the boundary using $e \circ \iota$, in the region ``vacated'' by $e$.

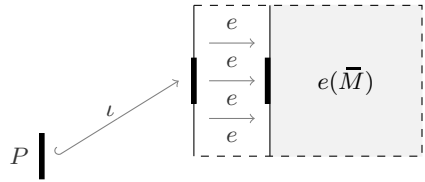
\begin{center}
\begin{tikzpicture}
[x=1mm,y=1mm]
\draw[fill,black!5] (10,-10) rectangle (30,10);
\draw[dashed] (0,10) -- (30,10) -- (30,-10) -- (0,-10);
\draw (0,-10) -- (0,10);
\draw (10,-10) -- (10,10);
\draw[fill] (-0.3,-3) rectangle (0.3,3);
\draw[fill] (9.4,-3) rectangle (10,3);
\draw[fill] (-20.3,-13) rectangle (-19.7,-7);
\draw[->,black!50] (2,0) -- (8,0);
\draw[->,black!50] (2,5) -- (8,5);
\draw[->,black!50] (2,-5) -- (8,-5);
\node at (5,2.5) [font=\small,black!80] {$e$};
\node at (5,-2.5) [font=\small,black!80] {$e$};
\node at (5,-7.5) [font=\small,black!80] {$e$};
\node at (5,7.5) [font=\small,black!80] {$e$};
\draw[<-,>=right hook,black!50] (-18,-10) to ($ (-18,-10)!0.5!(-2,0) $);
\draw[->,black!50] ($ (-18,-10)!0.5!(-2,0) $) to (-2,0);
\node at (-11,-4) [font=\small,black!80] {$\iota$};
\node at (20,0) [font=\small] {$e(\mbar)$};
\node at (-23,-10) [font=\small,black!80] {$P$};
\end{tikzpicture}
\end{center}

There is a natural basepoint $[e \circ \iota]$ of $\mathbf{C}_P(M) = \Emb(P,M)/\Diff(P)$. Via the maps \eqref{eStabMapIntro} this determines basepoints of each of the spaces $\mathbf{C}_{nP}(M)$. We now define $C_{nP}(M)$ to be the path-component of $\mathbf{C}_{nP}(M)$ containing this basepoint. The maps \eqref{eStabMapIntro} restrict to \emph{stabilisation maps}
\begin{equation}\label{eStabMapIntro2}
C_{nP}(M) \longrightarrow C_{(n+1)P}(M).
\end{equation}

We can now state a special case of our main theorem.

\begin{thm}[Special case of Theorem \ref{tmain}]
If $\dim(P)\leq \frac12 (\dim(M)-3)$, the stabilisation maps \eqref{eStabMapIntro2} induce split-injections on homology in all degrees and isomorphisms up to degree $\frac{n}{2}$.
\end{thm}

\paragraph{\anglenumber{ii.} Superpositions of parametrised submanifolds.}
\addcontentsline{toc}{subsection}{Superpositions of parametrised submanifolds.}

Instead of taking the orbit space by the action of $\mathrm{Diff}(P) \wr \Sigma_n$, we also consider the orbit space
\[
\mathbf{C}_{nP}(M;G) = \mathrm{Emb}(nP,M)/(G \wr \Sigma_n)
\]
for an open subgroup $G \leq \mathrm{Diff}(P)$. The diffeomorphism group $\mathrm{Diff}(P)$ is locally path-connected, so its open subgroups are in one-to-one corresponedence with the subgroups of the \emph{mapping class group} $\pi_0(\mathrm{Diff}(P))$. A point in $\mathrm{Emb}(nP,M)/(G \wr \Sigma_n)$ is a collection of $n$ pairwise disjoint submanifolds of $M$ that are each diffeomorphic to $P$ and each equipped with a parametrisation up to the action of $G$, in other words, a $G$-orbit of parametrisations, which may be more poetically described as a ``\emph{superposition}'' of parametrisations.

Proceeding exactly as above (see \S\ref{s:detailed-statements} for precise constructions), we restrict to a particular sequence $C_{nP}(M;G)$ of path-components and obtain stabilisation maps
\begin{equation}\label{eStabMapIntro3}
C_{nP}(M;G) \longrightarrow C_{(n+1)P}(M;G).
\end{equation}

Our main theorem is then the following.

\begin{thm}[Equivalent to Theorem \ref{tmain}]\label{t:G}
If $\dim(P)\leq \frac12 (\dim(M)-3)$, the stabilisation maps \eqref{eStabMapIntro3} induce split-injections on homology in all degrees and isomorphisms up to degree $\frac{n}{2}$.
\end{thm}

\paragraph{\anglenumber{iii.} Corollaries for mixed configurations and different kinds of stability.}
\addcontentsline{toc}{subsection}{Corollaries for mixed configurations and different kinds of stability.}

\begin{rmk}[{\textit{Mixed configurations}.}]
So far we have only discussed one path-component of the space $\mathbf{C}_{nP}(M;G)$, namely the one in which the embedded copies of $P$ are isotopic (modulo $G$) to the standard embedding $\{[e \circ \iota] , [e^2 \circ \iota] , \ldots, [e^n \circ \iota] \}$, in particular the different copies of $P$ are isotopic to each other. However, once we know Theorem \ref{t:G}, we also immediately obtain homological stability for ``mixed configurations'' of submanifolds, in which there are $n$ copies of $P$ that are embedded so as to be isotopic to the standard embedding $\{[e \circ \iota] , [e^2 \circ \iota] , \ldots, [e^n \circ \iota] \}$, as well as finitely many other embedded submanifolds (possibly with different diffeomorphism types or even dimensions), such that these two parts of the configuration may be isotoped to be ``far apart'' (more precisely, so that all copies of $P$ are in a specified collar neighbourhood of $\partial\mbar$ and all other components of the configuration are disjoint from this collar neighbourhood).

A rough sketch of how to deduce homological stability for such moduli spaces of mixed configurations is as follows. We consider the forgetful map that forgets the copies of $P$ and remembers only the other components of the configuration, which turns out to be a fibre bundle. The stabilisation map for the moduli spaces of mixed configurations is then a map of fibre bundles over a fixed base space, and its restriction to each fibre is a stabilisation map of the form \eqref{eStabMapIntro3}, which is homologically stable by Theorem \ref{t:G}. Via the induced map of Serre spectral sequences this implies that the stabilisation map for the moduli spaces of mixed configurations is also homologically stable.
\end{rmk}

\begin{rmk}[{\textit{Twisted homological stability}.}]\label{rmk:ths}
Unordered configuration spaces of points are also homologically stable with respect to so-called \emph{finite-degree} -- or \emph{polynomial} -- \emph{twisted coefficient systems}. For the symmetric groups (corresponding to configurations in $\bR^\infty$) this is due to \cite{Betley2002Twistedhomologyof}, and the result was extended to configuration spaces (in any connected, open manifold) in \cite{Palmer2018Twistedhomologicalstability}. The method of proof in \cite{Palmer2018Twistedhomologicalstability} is to relate the statement of homological stability with polynomial twisted coefficients, via a decomposition of the coefficients, a version of Shapiro's lemma and a spectral sequence, to homological stability with constant coefficients, and then deduce the twisted version from the constant version. This method is compatible with the setting of moduli spaces of disconnected submanifolds in this paper, so a generalisation of that argument, together with Theorem \ref{t:G}, implies that the stabilisation maps \eqref{eStabMapIntro3} are also homologically stable with respect to (appropriately-defined) polynomial twisted coefficient systems.

See Theorem D of \cite{Palmer2018-hs-msods-II} for a precise statement of this result, and \S 9 of \cite{Palmer2018-hs-msods-II} for a detailed explanation of how to adapt the method of proof of \cite{Palmer2018Twistedhomologicalstability} to this situation.
\end{rmk}

\begin{rmk}[{\textit{Representation stability}.}]
Via an argument of S{\o}ren Galatius [personal communication], which uses only elementary representation theory of the symmetric groups, representation stability \cite{ChurchFarb2013Representationtheoryand} for the rational cohomology of \emph{ordered} configuration spaces (which was first proved by other means in \cite{Church2012Homologicalstabilityconfiguration}) may be deduced from twisted homological stability for the unordered configuration spaces, for a certain polynomial twisted coefficient system. The same argument applied to Theorem \ref{t:G} (and Remark \ref{rmk:ths} above) implies that the rational cohomology of the \emph{ordered} moduli spaces of disconnected submanifolds, in which the different copies of $P$ are equipped with a total ordering, is representation stable.
\end{rmk}

\paragraph{\anglenumber{iv.} Labelled submanifolds and corollaries for diffeomorphism groups.}\label{para:labelled-submanifolds}
\addcontentsline{toc}{subsection}{Labelled submanifolds and corollaries for diffeomorphism groups.}

In the companion paper (\cite{Palmer2018-hs-msods-II}) we lift Theorem \ref{t:G} to the setting of moduli spaces of \emph{labelled} submanifolds. Recall that we have chosen an open subgroup $G \leq \mathrm{Diff}(P)$. Now let
\[
\pi \colon Z \longrightarrow \mathrm{Emb}(P,\mbar)
\]
be a $G$-equivariant Serre fibration with path-connected fibres. We may consider its $n$th power $\pi^n$, and let $Z_n$ be the preimage of $\mathrm{Emb}(nP,M) \subseteq \mathrm{Emb}(P,\mbar)^n$ under this map. This is a right $(G \wr \Sigma_n)$-space and the map
\[
\pi_n \colon Z_n \longrightarrow \mathrm{Emb}(nP,M)
\]
is $(G \wr \Sigma_n)$-equivariant. We define $\mathbf{C}_{nP}(M,Z;G) = Z_n / (G \wr \Sigma_n)$. The map $\pi_n$ induces a map
\[
\bar{\pi}_n \colon \mathbf{C}_{nP}(M,Z;G) \longrightarrow \mathbf{C}_{nP}(M;G),
\]
which is again a Serre fibration with path-connected fibres (see Corollary \ref{c:fibration-orbit-spaces2}). The preimage of the path-component $C_{nP}(M;G)$ of $\mathbf{C}_{nP}(M;G)$ is therefore a single path-component of $\mathbf{C}_{nP}(M,Z;G)$, which we denote by $C_{nP}(M,Z;G)$. Using some mild auxiliary data (see \S 6 of \cite{Palmer2018-hs-msods-II} for the details) we may also lift the stabilisation maps \eqref{eStabMapIntro3} to maps
\begin{equation}\label{eStabMapIntro4}
C_{nP}(M,Z;G) \longrightarrow C_{(n+1)P}(M,Z;G).
\end{equation}

In \cite{Palmer2018-hs-msods-II} we prove that Theorem \ref{t:G} lifts to this setting:

\begin{thm}[{\cite[Theorem B]{Palmer2018-hs-msods-II}}]
Let $G \leq \mathrm{Diff}(P)$ be an open subgroup and $\pi \colon Z \to \mathrm{Emb}(P,\mbar)$ a $G$-equivariant Serre fibration with path-connected fibres. Assume that $\mathrm{dim}(P) \leq \tfrac12(\mathrm{dim}(M) - 3)$. Then the stabilisation maps \eqref{eStabMapIntro4} induce isomorphisms on homology with field coefficients up to degree $\frac{n}{2}$ and (thus) isomorphisms on homology with integral coefficients up to degree $\frac{n}{2} - 1$.
\end{thm}

This allows us to deduce corollaries about homological stability of different kinds of diffeomorphism groups, including:
\begin{itemizeb}
\item[$\circ$] \emph{Symmetric diffeomorphism groups}, with respect to the operation of parametrised connected sum along a submanifold (see \cite[Theorem A]{Palmer2018-hs-msods-II}). This extends results of \cite{Tillmann2016Homologystabilitysymmetric}, which are concerned with symmetric diffeomorphism groups and the operation of connected sum (at a point).
\item[] {[}If we have two embeddings $e \colon L \hookrightarrow M$ and $f \colon L \hookrightarrow N$ and an isomorphism of normal bundles $\theta \colon \nu_e \cong \nu_f$, the \emph{parametrised connected sum} is obtained by removing tubular neighbourhoods of $e(L)$ and $f(L)$ and identifying the resulting boundaries using $\theta$.]
\item[$\circ$] Diffeomorphism groups of manifolds with \emph{conical singularities} (a special type of Baas-Sullivan singularity \cite{Sullivan1967Hauptvermutungmanifolds,Baas1973bordismtheorymanifolds}), with respect to the number of singularities of a given type (see \cite[Corollary C]{Palmer2018-hs-msods-II}).
\end{itemizeb}

\paragraph{\anglenumber{v.} Related results.}\label{para:related-results}
\addcontentsline{toc}{subsection}{Related results.}

There are several other homological stability results which are closely related to Theorem \ref{t:G}. We will give a brief overview in order of increasing dimension.

\paragraph{Configurations of finite sets on a surface.}

In \cite{Tran2014Homologicalstabilitysubgroups} it is shown that the \emph{partitioned surface braid groups} are homologically stable. Let $S$ be an open connected surface and write $\beta_r^S$ for the \emph{surface braid group} on $S$ on $r$ strands, namely $\beta_r^S = \pi_1(C_r(S))$. Now fix positive integers $\xi$ and $n$. There is a short exact sequence
\[
1 \to \purebraid_{n\xi}^S \longrightarrow \beta_{n\xi}^S \longrightarrow \Sigma_{n\xi} \to 1,
\]
where the kernel $\purebraid_{n\xi}^S$ is the \emph{pure surface braid group} on $n\xi$ strands. The wreath product $\Sigma_\xi \wr \Sigma_n = (\Sigma_\xi)^n \rtimes \Sigma_n$ is naturally a subgroup of $\Sigma_{n\xi}$, and the $n$th \emph{partitioned braid group} $\beta_n^{\xi,S}$ is defined to be its preimage in $\beta_{n\xi}^S$. Geometrically, it is the subgroup of the $n\xi$th braid group on $S$ consisting of braids that preserve a given partition $n\xi = \xi + \xi + \cdots + \xi$ of the endpoints. This is a subgroup of index
\[
\frac{(n\xi)!}{n!\,(\xi!)^n},
\]
corresponding to a covering space of $C_{n\xi}(S)$ with this number of sheets. Alternatively, it may be thought of as the \emph{moduli space of $n$ pairwise disjoint submanifolds of $M$ of diffeomorphism type $P$}, where $P = \{ 1 ,\ldots, \xi \}$. We note that this result is not included as a special case of Theorem \ref{t:G}, since $\mathrm{dim}(P) = 0 \nleqslant \tfrac12(\mathrm{dim}(S) - 3) = -\tfrac12$.

\paragraph{Unlinks in $3$-manifolds.}

Going up from dimension zero to dimension one, in \cite{Kupers2013Homologicalstabilityunlinked} it is shown that the sequence of moduli spaces
\[
C_{nS^1}(M) \longrightarrow C_{(n+1)S^1}(M)
\]
is homologically stable when $M$ is a connected $3$-manifold and $S^1 \hookrightarrow \partial M$ is an embedding into a coordinate chart. These are the spaces of $n$-component unlinks in $M$, for varying $n$. Note that Theorem \ref{t:G} only applies once the dimension of $M$ is at least $5$, so the two results are disjoint.

The subspace $\cE_n \subset C_{nS^1}(D^3)$ of unknotted, unlinked \emph{Euclidean} circles was studied in \cite{BrendleHatcher2013Configurationspacesrings} (the adjective ``Euclidean'' means that a circle is the image of $\{ (x,y,z)\in \bR^3 \mid z=0, x^2+y^2=1 \}$ under rotation, translation and dilation), who showed that the inclusion is a homotopy equivalence. As a corollary, the spaces $\cE_n$ are also homologically stable.

\paragraph{String motion groups.}

Closely related to this is the sequence of fundamental groups of the spaces $\cE_n \simeq C_{nS^1}(D^3)$, which are called the \emph{circle-braid groups} or the \emph{string motion groups}. They are isomorphic to the symmetric automorphism groups $\Sigma\Aut(F_n)$ of the free groups $F_n$ and also to certain quotients of mapping class groups of $3$-manifolds. In this latter guise they were proved in \cite[Corollary 1.2]{HatcherWahl2010Stabilizationmappingclass} to satisfy homological stability. More generally \cite[Corollary 1.3]{HatcherWahl2010Stabilizationmappingclass} proves homological stability for $\Sigma\Aut(G^{*n})$, where $G^{*n}$ denotes the iterated free product $G*\dotsb *G$, for $G=\pi_1(P)$ for certain $3$-manifolds $P$ (in particular $P=S^1 \times D^2$). By \cite[Th\'{e}or\`{e}me 1.2]{CollinetDjamentGriffin2013Stabilitehomologiquepour} these symmetric automorphism groups actually satisfy homological stability for any group $G$. Moreover, their integral homology may in principle be calculated using the methods of \cite{Griffin2013Diagonalcomplexesand}.\footnote{See Theorem C of \cite{Griffin2013Diagonalcomplexesand}. The statement is for $\Aut(G^{*n})$ and doesn't permit $G$ to be $\bZ$, but this restriction is removed by replacing $\Aut(G^{*n})$ with $\Sigma\Aut(G^{*n})$. In particular the integral homology of $\pi_1(\cE_n)=\Sigma\Aut(F_n)$ is identified with the direct sum of $H_*((\bZ/2)\wr\Sigma_n;\bZ)$ together with $H_{*-e}((\bZ/2)\wr\Aut(f);\bZ)$ where $f$ runs over certain directed, labelled, rooted trees with $n$ vertices and $e$ is the number of edges of $f$. The wreath product is formed using the obvious homomorphism $\Aut(f)\to\Sigma_n$ given by the action of an automorphism of $f$ on its vertices.}

Note that the spaces $\cE_n$ are not aspherical (they are finite-dimensional manifolds whose fundamental groups contain torsion), so homological stability for $\pi_1(\cE_n)$ is independent of homological stability for the spaces $\cE_n$ themselves. 

Rationally, homological stability for $\pi_1(\cE_n)$ is also known via a different route. There is a homomorphism to the hyperoctahedral group $\pi_1(\cE_n) \to W_n = (\bZ/2)\wr\Sigma_n$ which just remembers the permutation of the $n$ circles and their orientations; the kernel of this is called the \emph{pure string motion group}. By \cite{Wilson2012Representationstabilitycohomology} the rational cohomology of the pure string motion groups is \emph{uniformly representation stable} (see \cite[Definition 2.6]{ChurchFarb2013Representationtheoryand}) with respect to the action of $W_n$, from which it follows that the groups $\pi_1(\cE_n)$ are rationally homologically stable. In fact, by Theorem 7.1 of \cite{Wilson2012Representationstabilitycohomology}, the rational homology of $\pi_1(\cE_n)$ is trivial.

\paragraph{Spaces of connected subsurfaces.}

One dimension higher again, \cite{CanteroRandal-Williams2017Homologicalstabilityspaces} proves a homological stability result for spaces of \emph{connected} subsurfaces of a manifold. Denote the connected, orientable surface of genus $g$ with $b$ boundary components by $\Sigma_{g,b}$ and let
\[
\cE(\Sigma_{g,b},M) = \Emb(\Sigma_{g,b},M)/\Diff^+(\Sigma_{g,b}),
\]
where the embeddings are prescribed on a sufficiently small collar neighbourhood of the boundary $\partial \Sigma_{g,b}$ and the diffeomorphisms are required to be the identity on a sufficiently small collar neighbourhood of $\partial M$. Given a surface embedded in $\partial M\times [0,1]$ with the correct boundary there is an induced stabilisation map $\cE(\Sigma_{g,b},M) \to \cE(\Sigma_{g^\prime,b^\prime},M)$. Theorem 1.3 of \cite{CanteroRandal-Williams2017Homologicalstabilityspaces} says that these maps are homologically stable, under the assumption that $M$ is simply-connected and of dimension at least $6$ (or dimension $5$ with an extra restriction on the permissible stabilisation maps). The stable ranges for the various stabilisation maps are the same as the best-known ranges for Harer stability \cite{Harer1985Stabilityofhomology, Ivanov1993homologystabilityTeichmuller}, which were obtained in \cite{Boldsen2012Improvedhomologicalstability} and \cite{Randal-Williams2016Resolutionsmodulispaces}; see also \cite{Wahl2013Homologicalstabilitymapping}. In particular Theorem 1.3 of \cite{CanteroRandal-Williams2017Homologicalstabilityspaces} recovers Harer stability when $M=\bR^\infty$.

Moreover, \cite{CanteroRandal-Williams2017Homologicalstabilityspaces} also identifies the homology in the stable range. There is a well-defined scanning map from each $\cE(\Sigma_{g,b},M)$ to a certain explicit limiting space (a union of path-components of the space of compactly-supported sections of a bundle over $M$) which induces an isomorphism on homology in the stable range, assuming that $M$ is simply-connected and at least $5$-dimensional.

\paragraph{\anglenumber{vi.} Stable homology.}
\addcontentsline{toc}{subsection}{Stable homology.}

The next natural question after Theorem \ref{t:G} is \emph{what is the limiting homology?} As mentioned in the paragraph just above, this is known for moduli spaces of connected subsurfaces. For unordered configuration spaces this question was answered by McDuff in \cite{McDuff1975Configurationspacesof}:
\[
\mathrm{colim}_{n\to\infty} H_*(C_n(M,X)) \;\cong\; H_* \bigl( \Gamma_{c,\circ} \bigl( \dot{T}M \wedge_{\mathrm{fib}} X_+ \to M \bigr)\bigr) .
\]
The colimit is taken along the stabilisation maps, $\dot{T}M\to M$ is the fibrewise one-point compactification of the tangent bundle of $M$, $\dot{T}M \wedge_{\mathrm{fib}} X_+$ is its fibrewise smash product with $X_+$ and $\Gamma_{c,\circ}$ denotes any path-component of the space of compactly-supported sections of this bundle.

For moduli spaces of disconnected submanifolds one natural first guess is that the \emph{scanning map} is a homology equivalence in the limit. This is the map which identifies the limiting homology for unordered configuration spaces and for moduli spaces of connected subsurfaces, and it is also well-defined for moduli spaces of disconnected submanifolds. For simplicity we will consider just \emph{unparametrised} submanifolds, and discuss the scanning map for $C_{nP}(M)$.

For a finite-dimensional real vector space $V$ let $\mathit{Gr}_p(V)$ be the Grassmannian of $p$-planes in $V$, let $\gamma_p^\perp(V)$ be the orthogonal complement of the canonical $p$-plane bundle over $\mathit{Gr}_p(V)$ and denote its total space by $A_p(V)$; this is the affine Grassmannian of $p$-planes in $V$. Since $\mathit{Gr}_p(V)$ is compact the Thom space $\mathrm{Th}(\gamma_p^\perp(V))$ is the one-point compactification of $A_p(V)$, denoted $\dot{A}_p(V)$. This construction is functorial in $V$, so one can apply it fibrewise to real vector bundles, and in particular form the bundle $\dot{A}_p(TM) \to M$ of affine $p$-planes in the tangent bundle of $M$. The \emph{scanning map for $C_{nP}(M)$} may be constructed exactly as for spaces of connected subsurfaces (see page 1388 of \cite{CanteroRandal-Williams2017Homologicalstabilityspaces}),\footnote{Strictly speaking, one actually constructs a zig-zag of two maps where the reversed map is a weak equivalence.} and is of the form
\[
C_{nP}(M) \longrightarrow \Gamma_{c,\circ} \bigl( \dot{A}_p(TM)\to M \bigr) .
\]

In the case $M=\bR^m$ the codomain is $\Omega^m_\circ \dot{A}_p(\bR^m)$, so if we specialise to $m=3$ and $P=S^1$ we have
\[
\cE_n \simeq C_{nS^1}(\bR^3) \longrightarrow \Omega_0^3 \dot{A}_1(\bR^3).
\]
In \cite[\S 2]{Kupers2013Homologicalstabilityunlinked} it is shown that the first rational homology of the right-hand side is one-dimensional, whereas (abelianising the computation of $\pi_1(\cE_n)$ in Proposition 3.7 of \cite{BrendleHatcher2013Configurationspacesrings}) the first integral homology of the left-hand side is $(\bZ/2)^3$ for all $n \geq 2$. So the first guess that the scanning map identifies the stable homology of moduli spaces of disconnected submanifolds cannot be correct.

\paragraph{\anglenumber{vii.} Open questions.}
\addcontentsline{toc}{subsection}{Open questions.}

Following on from the previous section, one open problem is of course to identify the stable homology of moduli spaces of disconnected submanifolds $C_{nP}(M;G)$ in the cases when it is known to stabilise. This is -- to the best of the author's knowledge -- currently unknown unless $P$ is a point.

Another question is whether the ``stability slope'' may be improved beyond $\tfrac12$ (i.e., whether there is a homological stability result for the sequence of spaces $C_{nP}(M;G)$ that holds in degrees $* \leq \lambda n$, for some constant $\lambda > \tfrac12$), perhaps depending on the choice of coefficient ring.

Another open question is whether the dimension condition $p \leq \tfrac12(m-3)$ on the dimensions $(m,p) = (\mathrm{dim}(M),\mathrm{dim}(P))$ can be weakened. The evidence suggests that it ``should'' be possible, since the dimension pairs $(2,0)$ and $(3,1)$ are excluded by the condition $p \leq \tfrac12(m-3)$, but configuration spaces of points (or collections of $\xi$ points) on surfaces are certainly homologically stable, as are moduli spaces of unlinks in $3$-manifolds, by \cite{Kupers2013Homologicalstabilityunlinked}. Almost all of the cases in which homological stability is known for spaces of submanifolds are within the range in which there can be no non-trivial knotting --- the condition $p \leq \tfrac12(m-3)$ ensures this, as does the condition that $M$ must have dimension at least $5$ in the case of connected oriented subsurfaces. One exception is of course $C_{nS^1}(M)$ for a $3$-manifold $M$, but in this case only the path-component of the unlink has been shown to be homologically stable; it is not known what happens if one tries to stabilise by adding new non-trivially-knotted components to a configuration.

\begin{rmk}[\emph{Dependence on the dimension hypothesis.}]
\label{r:dimension-hypothesis}
The dimension hypothesis $p \leq \tfrac12(m-3)$ is used several times in the proof of homological stability in \S\ref{s:proof}. It is used for the proof of Proposition \ref{p:condition-iv} (and hence it is also needed for Corollary \ref{c:condition-iv}), and it is also used for Corollary \ref{c:Xn1}, where it is used to verify the dimension hypotheses of Lemma \ref{l:path-of-embeddings} (see Remark \ref{r:dimension-hypothesis-lemma} for a discussion of how these hypotheses are used in the proof of Lemma \ref{l:path-of-embeddings}). In addition, Remark \ref{r:connectivity-of-M-breve} uses the weaker assumption that $p \leq m-3$. Proposition \ref{p:iotabb} uses the hypothesis that the embedding $\iota \colon P \hookrightarrow \partial M$ admits a non-vanishing section of its normal bundle, which is automatically true if we assume that $p \leq \tfrac12(m-2)$. Proposition \ref{p:two-conditions} and Corollary \ref{c:two-conditions} depend on Proposition \ref{p:iotabb}, so they also use the same hypothesis (either that $p \leq \tfrac12(m-2)$ or the weaker hypothesis that $\iota$ admits a non-vanishing section of its normal bundle).
\end{rmk}

\paragraph{Outline.}
\addcontentsline{toc}{subsection}{Outline.}

In the next section we give a more precise statement of the main result of the paper, Theorem \ref{t:G}, which is reformulated there as Theorem \ref{tmain}. This will be proved in \S\ref{s:proof}. In \S\ref{s:axioms} we first isolate the spectral sequence part of the argument by giving an ``axiomatic'' homological stability theorem, or ``homological stability criterion'', Theorem \ref{tAxiomatic}. In \S\ref{s:fibre-bundles} we prove some preliminary results about fibre bundles and fibrations that we will need, notably using the notion of locally retractile group actions to show that certain maps between smooth mapping spaces are fibre bundles. Then in \S\ref{s:proof} we apply the homological stability criterion Theorem \ref{tAxiomatic} to prove Theorem \ref{tmain}. The appendix \S\ref{s:appendix} contains a proof of a technical lemma about homotopy fibres of augmented semi-simplicial spaces, which is used in a key step of the proof.

\paragraph{Acknowledgements.}
\addcontentsline{toc}{subsection}{Acknowledgements.}

The author would like to thank Federico Cantero Mor{\'a}n, S{\o}ren Galatius, Geoffroy Horel, Alexander Kupers, Oscar Randal-Williams and Ulrike Tillmann for many enlightening discussions during the preparation of this article. Additionally, he would like to thank the anonymous referee for helpful and detailed suggestions for improvements to an earlier draft.

%%%%%%%%%%%%%%%%%%%%%%%%%%%%%%%%%%%%%%%%%%%%%%%%%%%%%%%%%
%%%%%%%%%%%%%%%%%%%%%%%%%%%%%%%%%%%%%%%%%%%%%%%%%%%%%%%%%
\section{Precise formulation of the main result}\label{s:detailed-statements}

Let $M$ be a smooth, connected manifold of dimension $m$, with non-empty boundary. Let $P$ be a smooth, closed manifold of dimension $p$ and choose an embedding
\[
\iota \colon P \lhto \partial M,
\]
as well as a collar neighbourhood for the boundary of $M$, in other words a proper embedding
\[
\lambda \colon \partial M \times [0,2] \lhto M
\]
such that $\lambda(z,0) = z$ for all $z \in \partial M$. We also assume that $\lambda$ has the property that it may be extended to a slightly larger proper embedding $\partial M \times [0,2+\epsilon] \hookrightarrow M$ for some $\epsilon > 0$, but we do not fix this data. Choose an open (therefore closed) subgroup $G \leq \mathrm{Diff}(P)$.

\begin{notation}
For $\epsilon \in [0,2]$ we will write $M_\epsilon = M \smallsetminus \lambda (\partial M \times [0,\epsilon])$. In particular, $M_0$ is the interior of $M$. We will use the embedding space $\Emb(P,M)$ very often, so we abbreviate it to $E = \Emb(P,M)$. For a non-negative integer $k$, we write $kP$ for the disjoint union of $k$ copies of the manifold $P$, which is $\varnothing$ if $k=0$.
\end{notation}

Let $n$ be a non-negative integer. Consider the embedding space $\Emb(nP,M_1)$, which has a left-action of the group $G \wr \Sigma_n \leq \mathrm{Diff}(nP)$. The quotient space $\Emb(nP,M_1) / (G \wr \Sigma_n)$ may be thought of as the subspace of the symmetric product $\mathrm{Sp}^n(E/G)$ consisting of configurations $\{[\varphi_1],\ldots,[\varphi_n]\}$ such that the images $\varphi_i(P)$ are contained in $M_1 \subseteq M$ and are pairwise disjoint. Analogously, we consider the quotient $\Emb((n+1)P,M_0)/(G \wr \Sigma_{n+1})$ as a subspace of $\mathrm{Sp}^{n+1}(E/G)$.

\begin{defn}\label{d:standard}
A \emph{standard configuration} in $\Emb(nP,M_1) / (G \wr \Sigma_n) \subseteq \mathrm{Sp}^n(E/G)$ is one of the form
\[
\{ [\lambda(-,t_1)\circ\iota], \ldots, [\lambda(-,t_n)\circ\iota] \},
\]
where $t_1,\ldots,t_n$ are $n$ distinct numbers in the interval $(1,2)$. Note that all standard configurations are in the same path-component. A \emph{standard configuration} in $\Emb((n+1)P,M_0)/(G \wr \Sigma_{n+1})$ is one of the form
\[
\{ [\lambda(-,t_1)\circ\iota], \ldots, [\lambda(-,t_{n+1})\circ\iota] \},
\]
where $t_1,\ldots,t_{n+1}$ are $n+1$ distinct numbers in the interval $(0,2)$.
\end{defn}

\begin{defn}\label{d:input-data}
We now define a map $f \colon X \to Y$ depending on the data $(M,P,\lambda,\iota,G,n)$. First, $X$ is the path-component of $\Emb(nP,M_1) / (G \wr \Sigma_n)$ containing the standard configurations and $Y$ is the path-component of $\Emb((n+1)P,M_0)/(G \wr \Sigma_{n+1})$ containing the standard configurations. The map $f$ is then defined as follows:
\[
\{ [\varphi_1] , \ldots , [\varphi_n] \}  \;\longmapsto\; \{ [\lambda(-,\tfrac12)\circ\iota] , [\varphi_1] , \ldots , [\varphi_n] \}.
\]
\end{defn}

\begin{rmk}
This is an explicit version of the more intuitively-defined stabilisation maps \eqref{eStabMapIntro3} in the introduction.
\end{rmk}

\begin{lem}\label{l:injectivity}
The induced map $f_* \colon H_*(X) \to H_*(Y)$ is split-injective.
\end{lem}

This is not difficult to prove: see \S\ref{s:split-injectivity}. The main result of this paper is that these maps are also \emph{surjective} on homology in a range of degrees, if we assume that the dimension $p$ of $P$ is sufficiently small compared to the dimension $m$ of $M$.

\begin{athm}\label{tmain}
The induced map $f_* \colon H_*(X) \to H_*(Y)$ is an isomorphism if $* \leq \frac{n}{2}$ and $p \leq \frac12(m-3)$.
\end{athm}

This is equivalent to Theorem \ref{t:G}, and will be proved in \S\ref{s:proof}. Before that, in \S\ref{s:axioms} we establish some sufficient axiomatic criteria for homological stability.

%%%%%%%%%%%%%%%%%%%%%%%%%%%%%%%%%%%%%%%%%%%%%%%%%%%%%%%%%
%%%%%%%%%%%%%%%%%%%%%%%%%%%%%%%%%%%%%%%%%%%%%%%%%%%%%%%%%
\section{Homological stability criteria}\label{s:axioms}

In order to separate the geometric part of the proof from the more technical manipulation of spectral sequences, we give an axiomatic criterion for homological stability in this section, and apply it to moduli spaces of disconnected submanifolds in \S\ref{s:proof}. The idea of this method of proving homological stability -- finding ``resolutions'' of the maps one wishes to prove stability for -- is due to \cite{Randal-Williams2016Resolutionsmodulispaces} and has been used many times subsequently, for example in \cite{CanteroRandal-Williams2017Homologicalstabilityspaces, GalatiusRandal-Williams2018Homologicalstabilitymoduli, KupersMiller2016Homologicalstabilitytopological, Perlmutter2016Linkingformsstabilization, Randal-Williams2013Homologicalstabilityunordered}.

\subsection{Setup}\label{ss:axioms-setup}

Suppose that for each integer $n\geq 0$ we have a collection $\X(n)$ of maps between path-connected spaces. We want to find conditions implying that each $f\in\X(n)$ is $\frac{n}{2}$-homology-connected, meaning that it is an isomorphism on homology up to degree $\frac{n}{2}-1$ and a surjection up to degree $\frac{n}{2}$, equivalently that its mapping cone $Cf$ has trivial reduced homology up to degree $\frac{n}{2}$. As a convention we take $\X(n) = \{ \text{all continuous maps} \}$ for $n<0$.\footnote{This is compatible with our aim that each $f\in\X(n)$ is $\frac{n}{2}$-homology-connected: for $n=-1$ or $-2$ this says that the mapping cone $Cf$ has trivial reduced homology in degree $-1$, in other words $Cf\neq\varnothing$, but mapping cones are always non-empty (the mapping cone of $\varnothing\to X$ is $X_+$). For $n\leq -3$ the condition is vacuous.}

We will show that this holds if each $f\in\X(n)$ admits a \emph{resolution}, each level of which can be \emph{approximated} by a map in $\X(k)$ for $k<n$, plus a certain factorisation condition that these data must satisfy. We begin by making this last sentence precise.

\begin{notation}[Homology-connectivity.]
For a map $f\colon X\to Y$, the number $\hconn(f)$ is the largest integer $n$ such that $f_* \colon H_*(X) \to H_*(Y)$ is an isomorphism for $*\leq n-1$ and surjective for $*=n$. Equivalently it is the largest integer $n$ such that the reduced homology of the mapping cone of $f$ is trivial up to degree $n$.
\end{notation}

\begin{defn}[Resolution of a space.]\label{dResolutionSpace}
Recall that an augmented semi-simplicial space $Y_\bullet$ is a diagram of the form
\begin{center}
\begin{tikzpicture}
[x=0.8mm,y=1mm]
\node at (0,0) {$\cdots$};
\node at (20,0) {$Y_1$};
\node at (40,0) {$Y_0$};
\node at (56,0) [anchor=west] {$Y_{-1}=Y$};
\draw[->] (5,2)--(15,2);
\draw[->] (5,0)--(15,0);
\draw[->] (5,-2)--(15,-2);
\draw[->] (25,1)--(35,1);
\draw[->] (25,-1)--(35,-1);
\draw[->] (45,0)--(55,0);
\end{tikzpicture}
\end{center}

with face maps $d_i\colon Y_k\to Y_{k-1}$ for $1\leq i\leq k+1$ which satisfy the simplicial identities $d_i d_j = d_{j-1}d_i$ when $i<j$. Its geometric realisation $\geomr{Y_\bullet}$ is the quotient of $\bigsqcup_{k\geq 0} Y_k \times \Delta^k$ by the face relations $(d_i(y),z) \sim (y,\delta_i(z))$, where $\delta_i$ is the inclusion of the $i$th face of the standard simplex $\Delta^{k+1}$. This depends only on the unaugmented part $Y_{\geq 0}$ of $Y_\bullet$, and the augmentation map $Y_0\to Y$ induces a well-defined map $\geomr{Y_\bullet}\to Y$. The augmented semi-simplicial space $Y_\bullet$ is a \emph{$c$-resolution} of $Y$ if $\hconn(\geomr{Y_\bullet}\to Y) \geq \lfloor c\rfloor$.
\end{defn}

\begin{defn}[Resolution of a map.]\label{dResolutionMap}
If we have a map of augmented semi-simplicial spaces
\begin{equation}\label{eAxResolution}
\centering
\begin{split}
\begin{tikzpicture}
[x=1mm,y=1mm]
\node (la) at (0,0) {$X$};
\node (ra) at (20,0) {$Y$};
\node (lb) at (0,10) {$X_0$};
\node (rb) at (20,10) {$Y_0$};
\node (lc) at (0,20) {$X_1$};
\node (rc) at (20,20) {$Y_1$};
\node (ld) at (0,32) {$\vdots$};
\node (rd) at (20,32) {$\vdots$};
\draw[->] (la) to node[above,font=\small]{$f$} (ra);
\draw[->] (lb) to node[above,font=\small]{$f_0$} (rb);
\draw[->] (lc) to node[above,font=\small]{$f_1$} (rc);
\draw[->] (lb) to (la);
\draw[->] (rb) to (ra);
\draw[->] ($ (lc.south)+(-1,0) $) -- ($ (lb.north)+(-1,0) $);
\draw[->] ($ (lc.south)+(1,0) $) -- ($ (lb.north)+(1,0) $);
\draw[->] ($ (rc.south)+(-1,0) $) -- ($ (rb.north)+(-1,0) $);
\draw[->] ($ (rc.south)+(1,0) $) -- ($ (rb.north)+(1,0) $);
\draw[->] ($ (ld.south)+(-2,0) $) -- ($ (lc.north)+(-2,0) $);
\draw[->] ($ (ld.south)+(0,0) $) -- ($ (lc.north)+(0,0) $);
\draw[->] ($ (ld.south)+(2,0) $) -- ($ (lc.north)+(2,0) $);
\draw[->] ($ (rd.south)+(-2,0) $) -- ($ (rc.north)+(-2,0) $);
\draw[->] ($ (rd.south)+(0,0) $) -- ($ (rc.north)+(0,0) $);
\draw[->] ($ (rd.south)+(2,0) $) -- ($ (rc.north)+(2,0) $);
\end{tikzpicture}
\end{split}
\end{equation}
where $X_\bullet$ is a $(c-1)$-resolution and $Y_\bullet$ is a $c$-resolution, then we say that the semi-simplicial map $f_\bullet \colon X_\bullet \to Y_\bullet$ is a \emph{$c$-resolution} of $f\colon X\to Y$.
\end{defn}

\begin{defn}[Approximation of a resolution.]\label{dApproximationMap}
We say that a given $c$-resolution $f_\bullet \colon X_\bullet \to Y_\bullet$ is \emph{approximated} by maps $\{ f_{i\alpha}^\prime \colon X_{i\alpha}^\prime \to Y_{i\alpha}^\prime \}$ if the following holds (the index $\alpha$ runs over an indexing set which depends on $i$). For each $i\geq 0$ there is a commutative square
\begin{equation}\label{eAxApproximation}
\centering
\begin{split}
\begin{tikzpicture}
[x=1mm,y=1mm]
\node (tl) at (0,10) {$X_i$};
\node (tr) at (20,10) {$Y_i$};
\node (bl) at (0,0) {$A_i$};
\node (br) at (20,0) {$B_i$};
\draw[->] (tl) to node[above,font=\small]{$f_i$} (tr);
\draw[->] (tl) to node[left,font=\small]{$p_i$} (bl);
\draw[->] (tr) to node[right,font=\small]{$q_i$} (br);
\draw[->] (bl) to node[below,font=\small]{$\phi_i$} (br);
\end{tikzpicture}
\end{split}
\end{equation}
in which $p_i$ and $q_i$ are Serre fibrations and $\phi_i$ is a weak equivalence. Additionally, for at least one point $a_{i\alpha}$ in each path-component $A_{i\alpha}$ of $A_i$, the restriction of $f_i$ to $p_i^{-1}(a_{i\alpha}) \to q_i^{-1}(\phi_i(a_{i\alpha}))$ is equal to $f_{i\alpha}^\prime \colon X_{i\alpha}^\prime \to Y_{i\alpha}^\prime$.
\end{defn}

\begin{defn}[Double resolution and approximation]
A \emph{double resolution and approximation} of a map $f \colon X \to Y$ is a $c_1$-resolution $f_\bullet \colon X_\bullet \to Y_\bullet$ approximated by maps $\{ f_{i\alpha}^\prime \colon X_{i\alpha}^\prime \to Y_{i\alpha}^\prime \}$, a commutative square
\begin{equation}\label{eAxApproximation2}
\centering
\begin{split}
\begin{tikzpicture}
[x=1mm,y=1mm]
\node (tl) at (0,12) {$\xbar_\alpha$};
\node (tr) at (50,12) {$\ybar_\alpha$};
\node (bl) at (0,0) {$p_0^{-1}(a_{0\alpha})$};
\node (br) at (50,0) {$q_0^{-1}(\phi_0(a_{0\alpha}))$};
\node at (-13,0) {$X_{0\alpha}^\prime =$};
\node at (65,0) {$= Y_{0\alpha}^\prime$};
\draw[->] (tl) to node[above,font=\small]{$g_\alpha$} (tr);
\draw[->] (tl) to (bl);
\draw[->] (tr) to (br);
\draw[->] (bl) to node[below,font=\small]{$f_{0\alpha}^\prime$} (br);
\end{tikzpicture}
\end{split}
\end{equation}
for each point $a_{0\alpha} \in A_0$ as above, in which the vertical maps are weak equivalences, followed by (for each $\alpha$) a $c_2$-resolution $g_{\alpha\bullet} \colon \xbar_{\alpha\bullet} \to \ybar_{\alpha\bullet}$ approximated by maps $\{ \xbarp_{\alpha i \beta} \to \ybarp_{\alpha i \beta} \}$. This can of course be continued in the obvious way to define \emph{iterated resolutions and approximations}, but we will only need double resolutions and approximations later.
\end{defn}

\begin{defn}[Three factorisation conditions.]\label{dFactorisationConditions}
Consider a commutative square of continuous maps
\begin{equation}\label{e:commutative-square}
\centering
\begin{split}
\begin{tikzpicture}
[x=1mm,y=1mm]
\node (tl) at (0,10) {$A$};
\node (tr) at (20,10) {$B$};
\node (bl) at (0,0) {$C$};
\node (br) at (20,0) {$D$};
\draw[->] (tl) to (tr);
\draw[->] (bl) to node[below,font=\small]{$f$} (br);
\draw[->] (tl) to (bl);
\draw[->] (tr) to node[right,font=\small]{$h$} (br);
\end{tikzpicture}
\end{split}
\end{equation}
This satisfies the \emph{weak factorisation condition} if $h$ factors up to homotopy through $f$. It satisfies the \emph{strong factorisation condition} if it has a triangle decomposition
\begin{equation}\label{e:triangle-decomposition}
\centering
\begin{split}
\begin{tikzpicture}
[x=1mm,y=1mm]
\node (tl) at (0,10) {$A$};
\node (tr) at (20,10) {$B$};
\node (bl) at (0,0) {$C$};
\node (br) at (20,0) {$D$};
\draw[->] (tl) to (tr);
\draw[->] (bl) to node[below,font=\small]{$f$} (br);
\draw[->] (tl) to (bl);
\draw[->] (tr) to node[right,font=\small]{$h$} (br);
\draw[->] (tr) to (bl);
\node at (6,7) [font=\small] {$H$};
\node at (14,4) [font=\small] {$J$};
\end{tikzpicture}
\end{split}
\end{equation}
(meaning that there exists a diagonal map $B \to C$ and homotopies $H$ and $J$), and the composite homotopy $H \! J \colon S^1 \times A \to D$ is homotopic to the identity homotopy. It satisfies the \emph{moderate factorisation condition} if there is a triangle decomposition \eqref{e:triangle-decomposition} and a map $\ell \colon Z \to A$ such that the composite map $H \! J \circ (\mathrm{id} \times \ell) \colon S^1 \times Z \to D$ factors up to homotopy through $f$, in other words there is a dotted map making the square
\begin{equation}\label{e:triangle-decomposition-2}
\centering
\begin{split}
\begin{tikzpicture}
[x=1mm,y=1mm]
\node (tl) at (0,10) {$S^1 \times Z$};
\node (tr) at (30,10) {$S^1 \times A$};
\node (bl) at (0,0) {$C$};
\node (br) at (30,0) {$D$};
\draw[->] (tl) to node[above,font=\small]{$\mathrm{id} \times \ell$} (tr);
\draw[->] (bl) to node[below,font=\small]{$f$} (br);
\draw[->,dashed] (tl) to (bl);
\draw[->] (tr) to node[right,font=\small]{$H \! J$} (br);
\end{tikzpicture}
\end{split}
\end{equation}
commute up to homotopy.
\end{defn}

\begin{rmk}\label{r:strong-moderate}
The moderate factorisation condition implies the weak factorisation condition: if there is a triangle decomposition \eqref{e:triangle-decomposition}, then $h$ factors (via $J$) up to homotopy through $f$. Also, the strong factorisation condition implies the moderate factorisation condition: if $H \! J \colon S^1 \times A \to D$ is homotopic to the identity homotopy, then we may take $Z=A$, $\ell=\mathrm{id}$ and the dotted map $S^1 \times A \to C$ to be the projection onto the second factor followed by the map $A \to C$ from \eqref{e:commutative-square}.
\end{rmk}

\subsection{Sufficient criteria for stability}

\begin{thm}\label{tAxiomatic}
For each integer $n\geq 0$, let $\X(n)$ be a collection of maps between path-connected spaces. Then any one of the following six conditions implies that $\hconn(f) \geq \lfloor \frac{n}{2} \rfloor$ for each $f \in \X(n)$ and each $n\geq 0$, in other words, the sequence $\X(n)$ is homologically stable.

For any $m\leq n$ define $\X_m^n = \{ f \mid f \text{ is weakly equivalent to a map in } \X(k) \text{ for } m \leq k \leq n \}$.

The six conditions are as follows. For each $n\geq 2$, each $f \colon X \to Y$ of $\X(n)$ has
\begin{itemizeb}
\item[\textup{(1)}] an $\frac{n}{2}$-resolution $f_\bullet \colon X_\bullet \to Y_\bullet$ approximated by maps $f_{i\alpha}^\prime$ with $f_{0\alpha}^\prime \in \X_{n-2}^{n-1}$ and $f_{i\alpha}^\prime \in \X_{n-2i}^{n-1}$ for $i\geq 1$. In addition, for each $\alpha$, the square
\begin{equation}\label{eSquare}
\centering
\begin{split}
\begin{tikzpicture}
[x=1mm,y=1mm]
\node (tl) at (0,10) {$X_{0\alpha}^\prime$};
\node (tr) at (20,10) {$Y_{0\alpha}^\prime$};
\node (bl) at (0,0) {$X$};
\node (br) at (20,0) {$Y$};
\draw[->] (tl) to node[above,font=\small]{$f_{0\alpha}^\prime$} (tr);
\draw[->] (bl) to node[below,font=\small]{$f$} (br);
\draw[->] (tl) to (bl);
\draw[->] (tr) to (br);
\end{tikzpicture}
\end{split}
\end{equation}
  \begin{itemizeb}
  \item[\textup{(1s)}] satisfies the strong factorisation condition.
  \item[\textup{(1m)}] satisfies the moderate factorisation condition with $\ell \in \X_{n-2}^{n-1}$.
  \item[\textup{(1w)}] satisfies the weak factorisation condition, and we assume that every map in $\bigcup_{n\geq 0}\X(n)$ induces injections on homology in all degrees.
  \end{itemizeb}
\item[\textup{(2)}] an $\frac{n}{2}$-resolution $f_\bullet \colon X_\bullet \to Y_\bullet$ approximated by maps $f_{i\alpha}^\prime$ with $f_{0\alpha}^\prime \in \X_{n-2}^{n-1}$ and $f_{i\alpha}^\prime \in \X_{n-2i}^{n-1}$ for $i\geq 1$. Also, for each $\alpha$, the map $g_\alpha \colon \xbar_\alpha \to \ybar_\alpha$ has an $\frac{n}{2}$-resolution $g_{\alpha\bullet} \colon \xbar_{\alpha\bullet} \to \ybar_{\alpha\bullet}$ approximated by maps $g_{\alpha i \beta}^\prime$ with $g_{\alpha 0 \beta}^\prime \in \X_{n-2}^{n-1}$ and $g_{\alpha i \beta}^\prime \in \X_{n-2i}^{n-1}$ for $i\geq 1$. In addition, for each $\alpha$ and $\beta$, the square
\begin{equation}\label{eSquare2}
\centering
\begin{split}
\begin{tikzpicture}
[x=1mm,y=1mm]
\node (tl) at (0,10) {$\xbarp_{\alpha 0 \beta}$};
\node (tr) at (25,10) {$\ybarp_{\alpha 0 \beta}$};
\node (bl) at (0,0) {$X$};
\node (br) at (25,0) {$Y$};
\draw[->] (tl) to node[above,font=\small]{$g_{\alpha 0 \beta}^\prime$} (tr);
\draw[->] (bl) to node[below,font=\small]{$f$} (br);
\draw[->] (tl) to (bl);
\draw[->] (tr) to (br);
\end{tikzpicture}
\end{split}
\end{equation}
  \begin{itemizeb}
  \item[\textup{(2s)}] satisfies the strong factorisation condition.
  \item[\textup{(2m)}] satisfies the moderate factorisation condition with $\ell \in \X_{n-2}^{n-1}$.
  \item[\textup{(2w)}] satisfies the weak factorisation condition, and we assume that every map in $\bigcup_{n\geq 0}\X(n)$ induces injections on homology in all degrees.
  \end{itemizeb}
\end{itemizeb}
\end{thm}

\begin{rmk}
Our convention from the beginning of this section implies that the condition $f\in\X(n)$ is vacuous when $n<0$, so there is no condition on $f_{i\alpha}^\prime$ or $g_{\alpha i \beta}^\prime$ for $i>\frac{n}{2}$.
\end{rmk}

\begin{rmk}
There is an obvious generalisation of this theorem, in which one takes arbitrarily many resolutions and approximations of a given $f \in \X(n)$, before verifying the strong/moderate/weak factorisation condition. The number of iterated resolutions and approximations is permitted to depend on $f$. The proof of this generalisation is an immediate extension of the proof of Theorem \ref{tAxiomatic} given below.
\end{rmk}

\begin{rmk}
Theorem \ref{tAxiomatic} also admits improvements in terms of the range of homological degrees to which it applies, namely its conclusion that $\hconn(f) \geq \lfloor \tfrac{n}{2} \rfloor$ for each $f \in \X(n)$, which represents a \emph{stability slope} of $\tfrac12$. For example, the homological stability results of \cite{Randal-Williams2016Resolutionsmodulispaces} and \cite{CanteroRandal-Williams2017Homologicalstabilityspaces} have a stability slope of $\tfrac23$ rather than $\tfrac12$. In fact, Theorem \ref{tAxiomatic} does not apply directly to the situations of \cite{Randal-Williams2016Resolutionsmodulispaces, CanteroRandal-Williams2017Homologicalstabilityspaces}, since in those cases one has a collection of maps graded by a semigroup which is not $\bN$. Restricting attention to a subcollection of maps graded by a subsemigroup isomorphic to $\bN$ allows one in principle to apply Theorem \ref{tAxiomatic}, but its conclusion would not then be optimal. However, Theorem \ref{tAxiomatic} admits generalisations for collections of maps $\X(-)$ graded by other semigroups, which would recover this improved stability slope.

Another setting to which Theorem \ref{tAxiomatic} does not apply directly is that of \cite{KupersMiller2016Homologicalstabilitytopological}, since in that case the stability slope is strictly \emph{smaller} than $\tfrac12$ in general. However, a slight modification of Theorem \ref{tAxiomatic} would apply to this setting too: modifying the condition that $f_{i\alpha}^{\prime} \in \X_{n-2i}^{n-1}$ to a condition of the form $f_{i\alpha}^{\prime} \in \X_{n-v(i)}^{n-1}$ for some affine function $v$, the argument may be modified slightly to conclude homological stability with a stability slope depending on the function $v$.
\end{rmk}

\begin{proof}[Proof of Theorem \ref{tAxiomatic}]
The proof is by induction on $n$. First note that when $n=0,1$ the claim is just that each $f\in\X(n)$ induces a surjection on path-components, which is true since we have taken all spaces to be path-connected. So let $n\geq 2$ and assume by induction that the conclusion holds for smaller values of $n$. Let $f\in\X(n)$.

\vspace{1ex}
\textbf{\slshape Spectral sequences.}
We have a resolution and approximation of $f$, so we may consider the square \eqref{eAxApproximation} for each $i\geq 0$. First replace it by an objectwise weakly-equivalent one in which all spaces have path-components that are open. (This is possible because, for any space $Z$ with path-components $Z_\alpha$, the natural map $\bigsqcup_\alpha Z_\alpha \to Z$ is a weak equivalence.) It now splits as the topological disjoint union of the squares
\begin{equation}\label{eAxApproximationComponent}
\centering
\begin{split}
\begin{tikzpicture}
[x=1mm,y=1mm]
\node (tl) at (0,10) {$X_{i\alpha}$};
\node (tr) at (20,10) {$Y_{i\alpha}$};
\node (bl) at (0,0) {$A_{i\alpha}$};
\node (br) at (20,0) {$B_{i\alpha}$};
\draw[->] (tl) to node[above,font=\small]{$f_{i\alpha}$} (tr);
\draw[->] (tl) to node[left,font=\small]{$p_{i\alpha}$} (bl);
\draw[->] (tr) to node[right,font=\small]{$q_{i\alpha}$} (br);
\draw[->] (bl) to node[below,font=\small]{$\phi_{i\alpha}$} (br);
\end{tikzpicture}
\end{split}
\end{equation}
where $A_{i\alpha}$ runs through the path-components of $A_i$ as $\alpha$ varies, $B_{i\alpha}$ is the path-component of $B_i$ that contains $\phi_i(A_{i\alpha})$ (hence $B_{i\alpha}$ also runs through the path-components of $B_i$ as $\alpha$ varies, since $\phi_i$ is assumed to be a weak equivalence, in particular a $\pi_0$-bijection), $X_{i\alpha} = p_i^{-1}(A_{i\alpha})$, $Y_{i\alpha} = q_i^{-1}(B_{i\alpha})$ and the maps $-_{i\alpha}$ are the corresponding restrictions of the maps $-_i$. For each $\alpha$, this square can be replaced by an objectwise homotopy-equivalent one in which $A_{i\alpha} = B_{i\alpha}$ and $\phi_{i\alpha}$ is the identity, and $p_{i\alpha}$, $q_{i\alpha}$ are still Serre fibrations. There is then, for each $i\geq 0$ and $\alpha$, a relative Serre spectral sequence (\cf Remark 2 on page 351 of \cite{Switzer1975Algebraictopologyhomotopy} and Exercise 5.6 on page 178 of \cite{McCleary2001usersguideto}) converging to the reduced homology of $Cf_{i\alpha}$, the mapping cone of $f_{i\alpha}$, and whose $E^2$ page can be identified as
\begin{equation}\label{eRSSS}
E^2_{s,t} \cong H_s(A_{i\alpha};\widetilde{H}_t(Cf_{i\alpha}^\prime)).
\end{equation}
This is first quadrant and its $r$th differential has bidegree $(-r,r-1)$. Moreover, the edge homomorphism
\begin{equation}\label{eRSSSedgehom}
\widetilde{H}_t(Cf_{i\alpha}^\prime) \cong E^2_{0,t} \twoheadrightarrow E^{\infty}_{0,t} \hookrightarrow \widetilde{H}_t(Cf_{i\alpha})
\end{equation}
is the map on reduced homology induced the two inclusions $X_{i\alpha}^\prime = p_i^{-1}(a_{i\alpha})\hookrightarrow p_i^{-1}(A_{i\alpha}) = X_{i\alpha}$ and $Y_{i\alpha}^\prime = q_i^{-1}(\phi_i(a_{i\alpha}))\hookrightarrow q_i^{-1}(B_{i\alpha}) = Y_{i\alpha}$.

Secondly, the map of augmented semi-simplicial spaces $f_\bullet \colon X_\bullet \to Y_\bullet$ induces a spectral sequence converging to the shifted reduced homology $\widetilde{H}_{*+1}$ of the total homotopy cofibre (twice iterated mapping cone) of
\begin{equation}\label{eTotalHomotopyCofibre}
\centering
\begin{split}
\begin{tikzpicture}
[x=1mm,y=1mm]
\node (tl) at (0,10) {$\lVert X_\bullet \rVert$};
\node (tr) at (20,10) {$\lVert Y_\bullet \rVert$};
\node (bl) at (0,0) {$X$};
\node (br) at (20,0) {$Y$};
\draw[->] (tl) to node[above,font=\small]{$\lVert f_\bullet \rVert$} (tr);
\draw[->] (bl) to node[below,font=\small]{$f$} (br);
\draw[->] (tl) to (bl);
\draw[->] (tr) to (br);
\end{tikzpicture}
\end{split}
\end{equation}
and whose $E^1$ page can be identified as
\begin{equation}\label{eSSSS}
E^1_{s,t} \cong \widetilde{H}_t(Cf_s).
\end{equation}
This is slightly larger than first quadrant -- it lives in $\{t\geq 0, s\geq -1\}$ -- and its $r$th differential has bidegree $(-r,r-1)$. The first differential $E^1_{0,t}\to E^1_{-1,t}$ can be identified with the map $\widetilde{H}_t(Cf_0) \to \widetilde{H}_t(Cf)$ on homology induced by the augmentation maps. See \cite[\S 2.3]{Randal-Williams2016Resolutionsmodulispaces} (a construction is also given in \cite[Appendix B]{Palmer2013Homologicalstabilityoriented}).

\vspace{1ex}
\textbf{\slshape Strategy.}
Since each $X_{0\alpha} \subseteq X_0$ is a union of path-components (and similarly for $Y_{0\alpha} \subseteq Y_0$), the mapping cone $Cf_0$ decomposes as the wedge $\bigvee_\alpha Cf_{0\alpha}$. For each $\alpha$ we have a map $Cf_{0\alpha}^\prime \to Cf_{0\alpha}$. Assuming one of the variants of condition (1), the strategy is to prove that the composite map $\bigvee_\alpha Cf_{0\alpha}^\prime \to \bigvee_\alpha Cf_{0\alpha} = Cf_0 \to Cf$ is both surjective on homology up to degree $\lfloor \frac{n}{2} \rfloor$ and the zero map on reduced homology in this range. If instead we assume one of the variants of condition (2), then we also have a homology equivalence $Cg_\alpha \to Cf_{0\alpha}^\prime$ for each $\alpha$, as well as a decomposition $Cg_{\alpha 0} = \bigvee_\beta Cg_{\alpha 0 \beta}$ and maps $Cg_{\alpha 0 \beta}^\prime \to Cg_{\alpha 0 \beta}$. The strategy in this case is to prove that the composite map
\[
\textstyle
\bigvee_{\alpha\beta} Cg_{\alpha 0 \beta}^\prime \to
\bigvee_{\alpha\beta} Cg_{\alpha 0 \beta} =
\bigvee_\alpha Cg_{\alpha 0} \to
\bigvee_\alpha Cg_\alpha \to
\bigvee_\alpha Cf_{0\alpha}^\prime \to \bigvee_\alpha Cf_{0\alpha} = Cf_0 \to Cf
\]
is both surjective and zero on reduced homology up to degree $\lfloor \frac{n}{2} \rfloor$.

\vspace{1ex}
\textbf{\slshape Surjectivity on homology.}
We first prove that each map $Cf_{0\alpha}^\prime \to Cf_{0\alpha}$ is surjective on homology in the required range. By condition (1) and the inductive hypothesis we have that $\hconn(f_{0\alpha}^\prime)\geq \lfloor\frac{n}{2}\rfloor -1$, in other words $\widetilde{H}_t(Cf_{0\alpha}^\prime)=0$ for $t\leq \lfloor\frac{n}{2}\rfloor -1$. Hence the $E^2$ page of the relative Serre spectral sequence \eqref{eRSSS} (for $i=0$) vanishes for $t\leq \lfloor\frac{n}{2}\rfloor -1$ and any $s\geq 0$. So in the slightly larger range $t\leq\lfloor\frac{n}{2}\rfloor$ the second map of \eqref{eRSSSedgehom} is the identity (there are no extension problems for this total degree) and therefore $\widetilde{H}_t(Cf_{0\alpha}^\prime) \to \widetilde{H}_t(Cf_{0\alpha})$ is surjective.

Secondly, we show that $Cf_0 \to Cf$ is also surjective on homology in this range. By condition~(1) and the inductive hypothesis we have that $\hconn(f_{i\alpha}^\prime) \geq \lfloor\frac{n}{2}\rfloor -i$ for $i\geq 1$, in other words $\widetilde{H}_t(Cf_{i\alpha}^\prime)=0$ for $t\leq\lfloor\frac{n}{2}\rfloor -i$ and $i\geq 1$. Hence the $E^2$ page of the relative Serre spectral sequence \eqref{eRSSS} (for $i\geq 1$) vanishes for $t\leq\lfloor\frac{n}{2}\rfloor -i$ and any $s\geq 0$. Therefore in the limit we have
\[
\widetilde{H}_*(Cf_{i\alpha})=0 \qquad\text{for } *\leq\lfloor\tfrac{n}{2}\rfloor -i\quad (\text{when } i\geq 1) \qquad\text{and for } *\leq\lfloor\tfrac{n}{2}\rfloor -1\quad (\text{when } i=0).
\]
The $E^1$ page of the spectral sequence \eqref{eSSSS} is the direct sum over $\alpha$ of $\widetilde{H}_t(Cf_{s\alpha})$, and so it has a trapezium of zeros as shown in Figure \ref{fSSSS}. Note that since $f_\bullet \colon X_\bullet \to Y_\bullet$ is an $\frac{n}{2}$-resolution we have that $\widetilde{H}_{*+1}$ of the total homotopy cofibre of \eqref{eTotalHomotopyCofibre} is trivial for $*+1\leq\lfloor\frac{n}{2}\rfloor$, so the spectral sequence \eqref{eSSSS} converges to zero in total degree $*\leq\lfloor\frac{n}{2}\rfloor -1$. In particular for $t\leq\lfloor\frac{n}{2}\rfloor$ we have that $E^{\infty}_{-1,t}=0$. Moreover, the second, third and later differentials that hit $E^{r}_{-1,t}$ all have trivial domain and so cannot kill it. Hence $E^1_{-1,t}$ must already be killed by the first differential -- in other words, the first differential $\widetilde{H}_t(Cf_0) \cong E^1_{0,t} \to E^1_{-1,t} \cong \widetilde{H}_t(Cf)$ must be surjective.

If instead we assume condition (2), then all of the arguments above remain valid, and in addition we repeat these arguments, using the spectral sequences associated to the resolutions $g_{\alpha\bullet} \colon \xbar_{\alpha\bullet} \to \ybar_{\alpha\bullet}$ and the relative Serre spectral sequences associated to the corresponding approximations to prove that the maps $Cg_{\alpha 0 \beta}^\prime \to Cg_{\alpha 0 \beta}$ and $Cg_{\alpha 0} \to Cg_\alpha$ are all surjective on homology up to degree $\lfloor \frac{n}{2} \rfloor$.
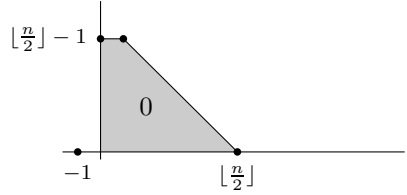
\begin{figure}[ht]
\centering
\begin{tikzpicture}
[x=1mm,y=1mm]
\fill[black!20] (0,0)--(0,15)--(3,15)--(18,0)--cycle;
\draw (0,-1)--(0,20);
\draw (-5,0)--(40,0);
\node (minusone) at (-3,0) [circle,fill,inner sep=1pt] {};
\node at (minusone.south) [anchor=north,font=\footnotesize] {$-1$};
\node (tl) at (0,15) [circle,fill,inner sep=1pt] {};
\node (tr) at (3,15) [circle,fill,inner sep=1pt] {};
\node (br) at (18,0) [circle,fill,inner sep=1pt] {};
\node at (br.south) [anchor=north,font=\footnotesize] {$\lfloor\frac{n}{2}\rfloor$};
\node at (tl.west) [anchor=east,font=\footnotesize] {$\lfloor\frac{n}{2}\rfloor -1$};
\draw (tl.center)--(tr.center)--(br.center);
\node (middle) at (6,6) {$0$};
\end{tikzpicture}
\caption{Zeros in the $E^1$ page of the spectral sequence \eqref{eSSSS}.}\label{fSSSS}
\end{figure}

\vspace{1ex}
\textbf{\slshape Zero on homology.}
If we assume one of the variants of condition (1), we now want to show that, for each $\alpha$, the composite map $Cf_{0\alpha}^\prime \to Cf_{0\alpha} \hookrightarrow Cf_0 \to Cf$ is zero on homology up to degree $\lfloor\frac{n}{2}\rfloor$. This is the map on mapping cones induced by the square \eqref{eSquare}. If we assume one of the variants of condition (2), we instead want to show that, for each $\alpha$ and $\beta$, the composite map
\[
Cg_{\alpha 0 \beta}^\prime \to Cg_{\alpha 0 \beta} \hookrightarrow Cg_{\alpha 0} \to Cg_\alpha \to Cf_{0\alpha}^\prime \to Cf_{0\alpha} \hookrightarrow Cf_0 \to Cf
\]
is zero on homology up to degree $\lfloor\frac{n}{2}\rfloor$. This is the map on mapping cones induced by the square \eqref{eSquare2}. So we just have to show that either \eqref{eSquare} or \eqref{eSquare2} induces the zero map on the homology of the mapping cones (of the horizontal maps) in a range of degrees. In fact, this is what we will do under the moderate and strong factorisation conditions, but under the weak factorisation condition we will do something slightly different. We start with this case.

\vspace{1ex}
\textbf{\slshape The weak factorisation condition.}
We first assume condition (1w), the weak factorisation condition and injectivity on homology of all maps in $\bigcup_{n\geq 0} \X(n)$. The square \eqref{eSquare} induces a map of long exact sequences on homology:
\begin{center}
\begin{tikzpicture}
[x=1.2mm,y=1.5mm]
\node (t2) at (25,10) {$\cdots$};
\node (t3) at (40,10) {$\widetilde{H}_t(Y_{0\alpha}^\prime)$};
\node (t4) at (60,10) {$\widetilde{H}_t(Cf_{0\alpha}^\prime)$};
\node (t5) at (80,10) {$\widetilde{H}_{t-1}(X_{0\alpha}^\prime)$};
\node (t6) at (105,10) {$\widetilde{H}_{t-1}(Y_{0\alpha}^\prime)$};
\node (t7) at (120,10) {$\cdots$};
\node (b1) at (5,0) {$\cdots$};
\node (b2) at (20,0) {$\widetilde{H}_t(X)$};
\node (b3) at (40,0) {$\widetilde{H}_t(Y)$};
\node (b4) at (60,0) {$\widetilde{H}_t(Cf)$};
\node (b5) at (75,0) {$\cdots$};
\draw[->] (t2) to (t3);
\draw[->] (t3) to node[above,font=\small]{$(**)$} (t4);
\draw[->] (t4) to (t5);
\draw[->] (t5) to node[above,font=\small]{$(f_{0\alpha}^\prime)_*$} (t6);
\draw[->] (t6) to (t7);
\draw[->] (b1) to (b2);
\draw[->] (b2) to node[below,font=\small]{$f_*$} (b3);
\draw[->] (b3) to (b4);
\draw[->] (b4) to (b5);
\draw[->] (t3) to (b2);
\draw[->] (t3) to (b3);
\draw[->] (t4) to node[right,font=\small]{$(*)$} (b4);
\end{tikzpicture}
\end{center}
The triangle on the left-hand side comes from the factorisation assumed in the weak factorisation condition. It implies that the composition $(*)\circ (**)$ is zero. Since we have assumed that maps in $\bigcup_{n\geq 0} \X(n)$ all induce injective maps on homology in all degrees, the map $(f_{0\alpha}^\prime)_*$ in the diagram is injective for any $t$, and so by exactness the map $(**)$ is surjective for any $t$.\footnote{It would not be enough to invoke the inductive hypothesis here, since it only tells us that the map $(f_{0\alpha}^\prime)_*$ in the diagram is injective for $t\leq\lfloor\frac{n}{2}\rfloor -1$, and we need the range $t\leq\lfloor\frac{n}{2}\rfloor$.} In the range $t\leq\lfloor\frac{n}{2}\rfloor$, and after taking the direct sum of the top line over $\alpha$, the map $(*)$ is surjective by what we proved above. So we have a map $(*) \circ (**)$ with target $\widetilde{H}_t(Cf)$ which is both surjective and zero in the required range, which finishes the inductive step of the proof.\footnote{Equivalently, one may organise the logic by deducing from the vanishing of $(*) \circ (**)$ and the surjectivity of $(**)$ that $(*)$ also vanishes, and hence, since $(*)$ is also surjective, its target must be zero.}

If we assume condition (2w), then we apply exactly the same argument to the map of long exact sequences induced by the square \eqref{eSquare2} instead.

\vspace{1ex}
\textbf{\slshape The moderate and strong factorisation conditions.}
We first assume either (1m) or (1s) and show that the map $Cf_{0\alpha}^\prime \to Cf$ is zero on reduced homology up to degree $\lfloor\frac{n}{2}\rfloor$. The triangle decomposition \eqref{e:triangle-decomposition} of \eqref{eSquare} induces a decomposition of this map on homology as follows:
\begin{center}
\begin{tikzpicture}
[x=1.2mm,y=1.5mm]
\node (t2) at (25,10) {$\cdots$};
\node (t3) at (40,10) {$\widetilde{H}_t(Cf_{0\alpha}^\prime)$};
\node (t4) at (60,10) {$\widetilde{H}_{t-1}(X_{0\alpha}^\prime)$};
\node (t5) at (75,10) {$\cdots$};
\node (b1) at (5,0) {$\cdots$};
\node (b2) at (20,0) {$\widetilde{H}_t(Y)$};
\node (b3) at (40,0) {$\widetilde{H}_t(Cf)$};
\node (b4) at (55,0) {$\cdots$};
\draw[->] (t2) to (t3);
\draw[->] (t3) to (t4);
\draw[->] (t4) to (t5);
\draw[->] (b1) to (b2);
\draw[->] (b2) to (b3);
\draw[->] (b3) to (b4);
\draw[->] (t3) to (b3);
\draw[->] (t4) to (b2);
\end{tikzpicture}
\end{center}
where the diagonal map is
\[
\widetilde{H}_{t-1}(X_{0\alpha}^\prime) \lhook\joinrel\xrightarrow{\;\text{K{\"u}nneth}\;} \widetilde{H}_t(S^1 \times X_{0\alpha}^\prime) \xrightarrow{\;(H \! J)_*\;} \widetilde{H}_t(Y).
\]
This is proved in \cite[Appendix A]{Palmer2013Homologicalstabilityoriented}, and it also follows from \cite[Lemma 7.5]{CanteroRandal-Williams2017Homologicalstabilityspaces}, which in fact gives a space-level decomposition. Putting this together with \eqref{e:triangle-decomposition-2} we have a commutative diagram
\begin{center}
\begin{tikzpicture}
[x=1.2mm,y=1.5mm]
\node (tl) at (0,20) {$\widetilde{H}_{t-1}(Z)$};
\node (tm) at (30,20) {$\widetilde{H}_{t-1}(X_{0\alpha}^\prime)$};
\node (tr) at (60,20) {$\widetilde{H}_t(Cf_{0\alpha}^\prime)$};
\node (ml) at (0,10) {$\widetilde{H}_t(S^1 \times Z)$};
\node (mm) at (30,10) {$\widetilde{H}_t(S^1 \times X_{0\alpha}^\prime)$};
\node (bl) at (0,0) {$\widetilde{H}_t(X)$};
\node (bm) at (30,0) {$\widetilde{H}_t(Y)$};
\node (br) at (60,0) {$\widetilde{H}_t(Cf)$};
\draw[->>] (tl) to node[above,font=\small]{$\ell_*$} (tm);
\draw[->] (tr) to (tm);
\draw[->] (ml) to (mm);
\draw[->] (bl) to node[above,font=\small]{$f_*$} (bm);
\draw[->] (bm) to (br);
\draw[->] (ml) to (bl);
\draw[->] (mm) to node[right,font=\small]{$(H \! J)_*$} (bm);
\draw[->] (tr) to node[right,font=\small]{$(\dagger)$} (br);
\incl{(tl)}{(ml)}
\incl{(tm)}{(mm)}
\draw[->] (bl.south)--(0,-4)-- node[below,font=\small]{$0$} (60,-4)--(br.south);
\end{tikzpicture}
\end{center}
with either $\ell \in \X_{n-2}^{n-1}$ (if we assume (1m)) or $\ell = \mathrm{id}$ (if we assume (1s), see Remark \ref{r:strong-moderate}). In either case we know that the induced map $\ell_*$ in the diagram is surjective for $t\leq\lfloor\frac{n}{2}\rfloor$ (by the inductive hypothesis if we assume (1m)). The bottom two horizontal maps are consecutive maps in a long exact sequence, so their composition is zero. A diagram chase therefore shows that the map $(\dagger)$ is zero for $t\leq\lfloor\frac{n}{2}\rfloor$, as required.

If we assume (2m) or (2s), then we use exactly the same argument, applied instead to the triangle decomposition \eqref{e:triangle-decomposition} of \eqref{eSquare2}.
\end{proof}

%%%%%%%%%%%%%%%%%%%%%%%%%%%%%%%%%%%%%%%%%%%%%%%%%%%%%%%%%
%%%%%%%%%%%%%%%%%%%%%%%%%%%%%%%%%%%%%%%%%%%%%%%%%%%%%%%%%
\section{Fibre bundle and fibration lemmas}\label{s:fibre-bundles}

\subsection{Smooth mapping spaces.}\label{ss:smooth-mapping-spaces}

In this section we prove some technical lemmas that certain maps between smooth mapping spaces are fibre bundles. First we discuss briefly the topology that we use for our smooth mapping spaces, and collect some basic facts that will be useful later.

For smooth\footnote{By ``smooth'' we always mean $C^\infty$-smooth. Also, manifolds in this section have empty boundary unless explicitly stated otherwise.} manifolds $L$ and $N$, we will always use the \emph{strong $C^\infty$ topology} on the space $C^\infty(L,N)$ of smooth maps $L \to N$, unless explicitly stated otherwise. This topology is also sometimes known as the \emph{Whitney $C^\infty$ topology}.

Whenever $L$ is compact, this coincides with the \emph{weak $C^\infty$ topology}, also known as the \emph{compact-open $C^\infty$ topology}, but when $L$ is non-compact it is strictly finer than the weak $C^\infty$ topology.\footnote{On the other hand, even when $L$ is non-compact, the strong and weak $C^\infty$ topologies do coincide on the subset $C_{\mathrm{pr}}^\infty(L,N)$ of all proper maps $L \to N$. In particular this means that they coincide on the diffeomorphism group $\mathrm{Diff}(N)$ for any manifold $N$. In practice, we will almost exclusively work with smooth mapping spaces with compact domain and diffeomorphism groups, so the distinction between strong and weak $C^\infty$ topologies does not arise.} In fact, the weak $C^\infty$ topology is metrisable and therefore paracompact, whereas the strong $C^\infty$ topology is not even first-countable when $L$ is non-compact and $\mathrm{dim}(N)>0$. On the other hand, the strong $C^\infty$ topology will be useful for us since various properties of smooth maps, notably the property of being an embedding, are open conditions in this topology, and moreover the strong $C^\infty$ topology makes the mapping space into a Baire space, so the Thom transversality theorem applies. We refer to \cite[\S 2]{Mather1969StabilityofC}, \cite[\S II.3]{GolubitskyGuillemin1973Stablemappingsand} and \cite[\S 2.1]{Hirsch1976Differentialtopology} for the definitions and further discussion of these topologies.

In addition to the facts mentioned above, we record three facts more formally, for future reference.

\begin{fact}\label{fact:locally-contractible}
For compact $L$, the space $C^\infty(L,N)$ is locally contractible. Its open subspace $\mathrm{Emb}(L,N)$ of smooth embeddings is therefore also locally contractible, and in particular $\mathrm{Diff}(L) = \mathrm{Emb}(L,L)$ is locally contractible.

For non-compact manifolds $N$, the diffeomorphism group $\mathrm{Diff}(N)$ is not even locally path-connected, but its subgroup $\mathrm{Diff}_c(N)$ of compactly-supported diffeomorphisms is locally contractible.
\end{fact}

\begin{proof}
For the first statement, see for example \cite[Corollary of Proposition $4'$, page 281]{Cerf1961Topologiedecertains}. For the non-local-path-connectedness of $\mathrm{Diff}(N)$ when $N$ is non-compact (but $\sigma$-compact), see \cite{GuranZarichnyui1984Whitneytopologyand} or Theorem 4 of \cite{BanakhMineSakaiYagasaki2011Homeomorphismanddiffeomorphism}, which states that for such $N$ the full diffeomorphism group $\mathrm{Diff}(N)$ is locally homeomorphic to the infinite box power $\square^\omega l_2$, which is not locally path-connected. Theorem 4 of \cite{BanakhMineSakaiYagasaki2011Homeomorphismanddiffeomorphism} also states that the subgroup $\mathrm{Diff}_c(N)$ is locally homeomorphic to $l_2 \times \bR^\infty \cong \boxdot^\omega l_2$, which is locally contractible.
\end{proof}

\begin{fact}\label{fact:composition-continuous}
The function
\[
C^\infty(L,M) \times C^\infty(M,N) \longrightarrow C^\infty(L,N)
\]
given by composition of smooth maps is in general discontinuous --- in fact it always fails to be continuous if $L$ is non-compact --- but it becomes continuous when it is restricted to the subset $C_{\mathrm{pr}}^\infty(L,M) \times C^\infty(M,N)$ of the domain, where $C_{\mathrm{pr}}^\infty$ denotes the subspace of proper smooth maps.
\end{fact}

This is originally due to \cite[Proposition 1 and Remark 2 on page 259]{Mather1969StabilityofC}. See also Proposition II.3.9 and the remark following it in \cite{GolubitskyGuillemin1973Stablemappingsand}. As a consequence, the right action of $\mathrm{Diff}(L)$ on the embedding space $\mathrm{Emb}(L,N)$ is always continuous, whereas we have to assume that $L$ is compact in order for the left action of $\mathrm{Diff}(N)$ on this space to be continuous.

\begin{fact}\label{fact:topological-embedding}
If $M$ is a submanifold of $N$, then the continuous injection
\[
C_{\mathrm{pr}}^\infty(L,M) \longrightarrow C^\infty(L,N)
\]
given by post-composition with the inclusion is a topological embedding, i.e.\ a homeomorphism onto its image. Thus, if $e \colon L \hookrightarrow N$ is a smooth embedding, the map
\[
e \circ - \colon \mathrm{Diff}(L) \longrightarrow \mathrm{Emb}(L,N)
\]
is a topological embedding.
\end{fact}

The first statement follows from Remark 1 on page 259 of \cite{Mather1969StabilityofC}. To deduce the second statement, set $M = e(L)$ and restrict the domain and codomain to embedding spaces, to obtain a topological embedding $\mathrm{Emb}_{\mathrm{pr}}(L,e(L)) \hookrightarrow \mathrm{Emb}(L,N)$. Now note that a proper embedding between manifolds of the same dimension is a diffeomorphism, so $\mathrm{Diff}(L,e(L)) = \mathrm{Emb}_{\mathrm{pr}}(L,e(L))$, and pre-compose with the homeomorphism $e \circ - \colon \mathrm{Diff}(L) \to \mathrm{Diff}(L,e(L))$.

\begin{assumption}\label{notation-convention}
For the rest of this section, unless otherwise specified, $G$ and $H$ denote topological groups, $L$ and $N$ denote manifolds without boundary, with $L$ assumed to be compact, such that $\mathrm{dim}(L) < \mathrm{dim}(N)$.
\end{assumption}

\subsection{A fibre bundle criterion.}

Let $G$ be a topological group with a left-action on a space $X$, i.e.\ a group homomorphism $G\to \mathrm{Homeo}(X)$ such that the adjoint $a\colon G\times X\to X$ is continuous.\footnote{Note that we do not need to specify a topology on the group of self-homeomorphisms of $X$, since the continuity condition is on the action map $G \times X \to X$ and not on the group homomorphism $G \to \mathrm{Homeo}(X)$.}

\begin{defn}\label{d:G-locally-retractile}
We say that $X$ is \emph{$G$-locally retractile} with respect to this action (equivalently, that the action $G\curvearrowright X$ \emph{admits local sections}) if every $x\in X$ has an open neighbourhood $V_x \subseteq X$ and a continuous map $\gamma_x \colon V_x \to G$ sending $x$ to the identity such that $\gamma_x(y)\cdot x = y$ for all $y\in V_x$, in other words the composition
\[
V_x \xrightarrow{\gamma_x} G \xrightarrow{-\cdot x} X
\]
is the inclusion. Equivalently, each orbit map $G \times \{ x \} \hookrightarrow G\times X \xrightarrow{a} X$ admits a (pointed) section on some open neighbourhood of $x\in X$.
\end{defn}

The notion of $G$-locally retractile spaces comes from \cite[\S 0.4.4]{Cerf1961Topologiedecertains}, where the definition is given in a more general categorical setting; the definition above corresponds to taking $\cC$ to be the category of topological spaces and $T'$ to be the identity functor in \S 0.4.4 of \cite{Cerf1961Topologiedecertains}.\footnote{A minor difference is that the definition of \S 0.4.4 of \cite{Cerf1961Topologiedecertains} requires that $\gamma_x(y) \cdot y = x$ for all $y \in V_x$, instead of $\gamma_x(y) \cdot x = y$. But these conditions are interchangeable, by composing $\gamma_x$ with the inverse map $(-)^{-1} \colon G \to G$.} In the terminology of \cite{Palais1960Localtrivialityof}, one says that $X$ \emph{admits local cross-sections} (with respect to the action of $G$). The notion of $G$-locally retractile spaces has also recently been used in \cite{CanteroRandal-Williams2017Homologicalstabilityspaces}.

Immediately from the definitions, we note:

\begin{lem}\label{l:retractile-basic}
Let $H \leq G \leq K$ be topological groups and let $X$ be a locally retractile $G$-space.
\begin{itemizeb}
\item[\textup{(i)}] If $H$ is open in $G$, then the action of $H$ on $X$ is locally retractile.
\item[\textup{(ii)}] If the $G$-action on $X$ extends to a $K$-action, then this action is also locally retractile.
\item[\textup{(iii)}] If $A \subseteq X$ is an open, $G$-invariant subspace, then the action of $G$ on $A$ is locally retractile.
\end{itemizeb}
In particular, if $G$ is locally path-connected, then the action of $G_0$ on $X$ is locally retractile, where $G_0$ is the path-component of $G$ containing the identity.
\end{lem}

The following lemma is very useful for checking that a given map is a fibre bundle.

\begin{prop}[{\cite[Theorem A]{Palais1960Localtrivialityof}}]\label{p:G-locally-retractile}
Let $f \colon X \to Y$ be a $G$-equivariant continuous map and assume that the space $Y$ is $G$-locally retractile. Then $f$ is a fibre bundle.
\end{prop}

\begin{proof}
Let $y\in Y$ and take $\gamma_y\colon V_y \to G$ as in Definition \ref{d:G-locally-retractile}. Then the map
\[
(x,v) \longmapsto \gamma_y(v)\cdot x \;\colon\; f^{-1}(y)\times V_y \longrightarrow f^{-1}(V_y)
\]
is a local trivialisation of $f$ over $V_y$, with inverse given by
\[
x \longmapsto (\gamma_y(f(x))^{-1}\cdot x,f(x)) \;\colon\; f^{-1}(V_y) \longrightarrow f^{-1}(y) \times V_y.\qedhere
\]
\end{proof}

In the special case where $f$ is of the form $X \to X/H$ for a right action of another group $H$ on $X$, we have a stronger conclusion, as remarked in \cite{Palais1960Localtrivialityof} on page 307.

\begin{prop}\label{p:G-locally-retractile-principal}
Let $X$ be a space with a left $G$-action and a right $H$-action, which commute. Assume that the $H$-action is free, and moreover that for each $x \in X$ the map $x \cdot - \colon H \to X$ is a topological embedding. Assume that the induced left $G$-action on the quotient space $X/H$ is locally retractile. Then the quotient map $q \colon X \to X/H$ is a principal $H$-bundle.
\end{prop}

\begin{proof}
For a point $y = xH \in X/H$, the local trivialisation constructed in the proof above is
\[
(x \cdot h,v) \longmapsto \gamma_y(v) \cdot x \cdot h \;\colon\; xH \times V_y \longrightarrow q^{-1}(V_y)
\]
and it is clearly $H$-equivariant, where we let $H$ act on $xH$ and on $q^{-1}(V_y)$ by the restriction of its action on $X$, and act trivially on $V_y$. By hypothesis, the map $h \mapsto x \cdot h \colon H \to xH \subseteq X$ is a homeomorphism, so we may compose the trivialisation above with the homeomorphism $(h,v) \mapsto (x \cdot h,v) \colon H \times V_y \to xH \times V_y$ to obtain an $H$-equivariant local trivialisation
\[
(h,v) \longmapsto \gamma_y(v) \cdot x \cdot h \;\colon\; H \times V_y \longrightarrow q^{-1}(V_y).
\]
This is now a local trivialisation of $q$ over $V_y$ as a principal $H$-bundle.
\end{proof}

For example, we note that for any subgroup $H$ of a topological group $G$, we may take $X=G$ with $H$ acting by right-multiplication. In this case the map $H \to G$ given by $h \mapsto gh$ for fixed $g \in G$ is always a topological embedding (its inverse is given by $k \mapsto g^{-1}k$), so the proposition above says in this case that if $G/H$ is $G$-locally retractile, then $G \to G/H$ is a principal $H$-bundle. Subgroups $H$ having the property that $G \to G/H$ is a principal $H$-bundle are sometimes called \emph{admissible}, so one may rephrase this as saying that $H$ is admissible if $G/H$ is $G$-locally retractile.

\subsection{Manifolds without boundary.}

A very useful setting where this criterion applies is for embedding spaces between smooth manifolds. Let $\mathrm{Diff}_c(N)$ denote the topological group of compactly-supported self-diffeomorphisms of $N$, i.e.\ diffeomorphisms $\theta \colon N \to N$ such that $\{ x \in N \mid \theta(x) \neq x \}$ is relatively compact in $N$.

By Fact \ref{fact:composition-continuous}, the action of $\mathrm{Diff}_c(N)$ on $\mathrm{Emb}(L,N)$ is continuous.

\begin{prop}[{\cite[Theorem B]{Palais1960Localtrivialityof}\footnote{In fact, Theorem B of \cite{Palais1960Localtrivialityof} says that the action of the path-component of the identity $\mathrm{Diff}_c(N)_0$ on $\mathrm{Emb}(L,N)$ is locally retractile, but this is equivalent to the statement of Proposition \ref{p:G-locally-retractile-embeddings} by Lemma \ref{l:retractile-basic}.}}]\label{p:G-locally-retractile-embeddings}
The space $\mathrm{Emb}(L,N)$ is $\mathrm{Diff}_c(N)$-locally retractile.
\end{prop}

There is also a 14-line proof in \cite{Lima1963localtrivialityof}. It is also proved in \cite{Cerf1961Topologiedecertains}, in a much more general setting for manifolds with corners, which we will discuss in \S\ref{ss:manifolds-with-boundary} below. We will also give a self-contained proof in \S\ref{ss:quotients-of-embedding-spaces} below, in the course of proving Proposition \ref{p:G-locally-retractile-orbitspace}, which is a variant of this result for quotients of embedding spaces. (Our proof of Proposition \ref{p:G-locally-retractile-orbitspace} will use the ideas of \cite{Lima1963localtrivialityof}.)

We note that a direct corollary of Proposition \ref{p:G-locally-retractile-embeddings} is the isotopy extension theorem.

\begin{thm}
Suppose we are given a smooth isotopy of embeddings from $L$ into $N$, i.e.\ a path $e \colon [0,1] \to \mathrm{Emb}(L,N)$. Then there is a path $\phi \colon [0,1] \to \mathrm{Diff}_c(N)$ of compactly-supported diffeomorphisms such that $\phi(0)$ is the identity and $e(t) = \phi(t) \circ e(0)$ for all $t \in [0,1]$.
\end{thm}
\begin{proof}
The map $- \circ e(0) \colon \Diff_c(N) \to \mathrm{Emb}(L,N)$ is $\Diff_c(N)$-equivariant, and by Proposition~\ref{p:G-locally-retractile-embeddings} the embedding space $\mathrm{Emb}(L,N)$ is $\Diff_c(N)$-locally retractile, so Proposition \ref{p:G-locally-retractile} implies that the map is a fibre bundle and hence a Serre fibration.
\end{proof}

\begin{rmk}
The fact that the map $- \circ e(0) \colon \Diff_c(N) \to \mathrm{Emb}(L,N)$ is a Serre fibration also implies parametrised versions of the isotopy extension theorem, for smooth isotopies of embeddings parametrised by CW complexes. In fact, since $L$ is compact, the space $\mathrm{Emb}(L,N)$ is paracompact, as noted at the beginning of \S\ref{ss:smooth-mapping-spaces}, so this implies that $- \circ e(0)$ is in fact a Hurewicz fibration. Thus we also obtain versions of the isotopy extension theorem parametrised by an arbitrary space.
\end{rmk}

In the proof of Theorem \ref{tmain} in \S\ref{s:proof} below we will not use Proposition \ref{p:G-locally-retractile-embeddings} itself, but rather two different extensions of it, one for manifolds with non-empty boundary (which is contained in the work of Cerf) and one for quotients of embedding spaces by open subgroups of $\mathrm{Diff}(L)$ (which we prove directly). These are discussed in the next two subsections.

\subsection{Manifolds with boundary.}\label{ss:manifolds-with-boundary}

For the proof of Theorem \ref{tmain} in \S\ref{s:proof} we will need a version of Proposition \ref{p:G-locally-retractile-embeddings} for manifolds with boundary. In fact, this result extends to a very general setting for manifolds with corners, as shown in the work of Cerf~\cite{Cerf1961Topologiedecertains}. We first state the special case that we will actually use, and then discuss how to deduce it from \cite{Cerf1961Topologiedecertains}.

Let $L$ and $N$ be manifolds with boundary decomposed as
\[
\partial L = \partial_1 L \sqcup \partial_2 L \qquad \partial N = \partial_1 N \sqcup \partial_2 N
\]
and assume that $L$ is compact. Let $\mathrm{Emb}_{12}(L,N)$ be the space, equipped with the (weak $=$ strong) $C^\infty$ topology, of all smooth embeddings $L \hookrightarrow N$ taking $\partial_1 L$ to $\partial_1 N$, $\partial_2 L$ to $\partial_2 N$ and $\mathrm{int}(L)$ to $\mathrm{int}(N)$. Any diffeomorphism of $N$ that is isotopic to the identity must preserve the decomposition of $\partial N$ into $\partial_1 N$ and $\partial_2 N$, so there is a well-defined action of $\mathrm{Diff}_c(N)_0$ on $\mathrm{Emb}_{12}(L,N)$. We will see in the discussion of Cerf's results below that this action is continuous.

\begin{prop}\label{p:G-locally-retractile-boundary}
With respect to this action, $\mathrm{Emb}_{12}(L,N)$ is $\mathrm{Diff}_c(N)_0$-locally retractile.
\end{prop}

Now let $L$ and $N$ be two smooth manifolds with corners, where $L$ is assumed to be compact, and define $\Emb(L,N)$ to be the space of smooth embeddings of $L$ into $N$ with fixed ``incidence relations'' (see the next paragraph or \cite[\nopp II.1.1.1]{Cerf1961Topologiedecertains}). Everything will be given the strong $C^\infty$ topology. Choose an embedding $f\in\Emb(L,N)$ and an open neighbourhood $U$ of $f(L)$ in $N$. Let $\Diff_U(N)$ be the group of diffeomorphisms of $N$ which restrict to the identity on $N\smallsetminus U$ and let
\[
\mathrm{Iso}_U(N) \;\subseteq\; C^\infty(N\times [0,1],N)
\]
be the subspace of maps $f$ such that $f(-,t)\in\Diff_U(N)$ for all $t$ and $f(-,0)=\mathrm{id}$. It turns out that this is a topological group under pointwise composition and acts continuously on $\Emb(L,N)$ by taking the diffeomorphism at the endpoint of the path and composing with the embedding.

Before stating Cerf's analogue of Proposition \ref{p:G-locally-retractile-embeddings}, we justify and explain the claim at the end of the last paragraph. First, the space $\mathrm{Iso}_U(N)$ is a topological group with respect to pointwise composition by \cite[\nopp II.1.5.2]{Cerf1961Topologiedecertains}. Second, we need to know that composing an embedding $L\hookrightarrow N$ with a diffeomorphism in $\Diff_U(N)$ does not change its incidence relations. To explain this we briefly recall some notions from \cite[\nopp II.1.1.1]{Cerf1961Topologiedecertains}. Each point in $N$ has an \emph{index}, which is the unique integer $k$ such that $N$ is locally homeomorphic to $\bR^k \times [0,\infty)^{n-k}$ at that point. Writing $N^k$ for the subspace of points of index at most $k$, a \emph{face} of $N$ of index $k$ is the closure in $N$ of a path-component of $N^k \smallsetminus N^{k-1}$. An \emph{incidence relation} is then a choice of a face of $N$ for each point in $L$. A diffeomorphism of $N$ induces a permutation of the set of faces of $N$, and it preserves the incidence relations of any embedding into $N$ if and only if this permutation is the identity. It is easy to see that the induced permutation is a locally-constant invariant of diffeomorphisms of $N$, and so any diffeomorphism in the path-component of the identity sends every face to itself. Now, an element of the group $\mathrm{Iso}_U(N)$ is a path in $\Diff(N)$ starting at the identity, so the diffeomorphism at its endpoint sends every face of $N$ to itself, and therefore preserves the incidence relations of any embedding into $N$. Thus the group action map $\mathrm{Iso}_U(N) \times \Emb(L,N) \to \Emb(L,N)$ is well-defined. Finally, we need to know that this map is continuous. It is given by $(a,b)\mapsto a\circ i\circ b$, where $i\colon N\hookrightarrow N\times [0,1]$ takes $x$ to $(x,1)$. This is continuous by \cite[Proposition 5, page 281]{Cerf1961Topologiedecertains} since $i$ is proper and $L$ is compact.

\begin{thm}[{\cite[Th{\'e}or{\`e}me 5, page 293]{Cerf1961Topologiedecertains}}]\label{t:Cerf}
The action of $\mathrm{Iso}_U(N)$ on $\Emb(L,N)$ admits a local section at $f\in\Emb(L,N)$.
\end{thm}

This implies the analogue of Proposition \ref{p:G-locally-retractile-embeddings} for manifolds with corners.

\begin{coro}\label{c:Cerf}
The action of $\Diff_c(N)_0$ on $\Emb(L,N)$ is locally retractile.
\end{coro}

\begin{proof}[Proof of the corollary]
To see this, given a point $f\in\Emb(L,N)$, choose a relatively compact open neighbourhood $U$ of $f(L)$ in $N$. The action of $\mathrm{Iso}_U(N)$ on $\Emb(L,N)$ factors through the endpoint map $\mathrm{Iso}_U(N) \to \Diff_c(N)_0$ so a local section at $f$ for the action of $\mathrm{Iso}_U(N)$ induces a local section at $f$ for the action of $\Diff_c(N)_0$.\footnote{An important point is that although the adjoint $[0,1]\to C^\infty(N,N)$ of a smooth map $g\colon N\times [0,1]\to N$ need not be continuous in general for non-compact $N$, it \emph{is} continuous if $g$ is a compactly-supported diffeotopy (\cf \cite[Lemma 1.12]{Haller1995Groupsofdiffeomorphisms}). This is needed to see that the endpoint map $\mathrm{Iso}_U(N) \to \Diff_c(N)$ lands in $\Diff_c(N)_0$.}
\end{proof}

\begin{proof}[Proof of Proposition \ref{p:G-locally-retractile-boundary} from the corollary]
The setting of Proposition \ref{p:G-locally-retractile-boundary} is a special case of the setting of Cerf. Note that the faces of a manifold with boundary (but no higher-codimension corners) are precisely its boundary-components, together with its interior. The condition imposed in the definition of $\mathrm{Emb}_{12}(L,N)$ is therefore a union of incidence relations in the sense of Cerf, and so $\mathrm{Emb}_{12}(L,N)$ splits as a topological disjoint union of $\mathrm{Emb}_{\cI}(L,N)$ for various different incidence relations $\cI$.\footnote{If $\partial_1 N$ and $\partial_2 N$ are both connected, then there is just one incidence relation in the disjoint union. In general, the number of incidence relations in the disjoint union is $\leq n_1^{\ell_1} n_2^{\ell_2}$, where $n_i$ is the number of components of $\partial_i N$ and $\ell_i$ is the number of components of $\partial_i L$.} The action of $\Diff_c(N)_0$ on $\mathrm{Emb}_{12}(L,N)$ preserves this decomposition, and its action on each piece is locally retractile by Corollary \ref{c:Cerf}, so its action on the whole space $\mathrm{Emb}_{12}(L,N)$ is also locally retractile.
\end{proof}

\subsection{Quotients of embedding spaces.}\label{ss:quotients-of-embedding-spaces}

Another version of Proposition \ref{p:G-locally-retractile-embeddings} that we will need is for quotients of embedding spaces by groups of diffeomorphisms, in the following sense. We now return to the setting of Assumption \ref{notation-convention}, in particular all manifolds have empty boundary and $L$ is compact with $\mathrm{dim}(L) < \mathrm{dim}(N)$. Let $G$ be any open subgroup of $\mathrm{Diff}(L)$.

\begin{prop}\label{p:G-locally-retractile-orbitspace}
The orbit space $\mathrm{Emb}(L,N)/G$ is $\mathrm{Diff}_c(N)$-locally retractile.
\end{prop}

\begin{coro}\label{c:principal-H-bundle}
The quotient map $\mathrm{Emb}(L,N) \to \mathrm{Emb}(L,N)/G$ is a principal $G$-bundle.
\end{coro}

\begin{proof}[Proof of the corollary]
This follows directly from Fact \ref{fact:topological-embedding} and Propositions \ref{p:G-locally-retractile-principal} and \ref{p:G-locally-retractile-orbitspace}.
\end{proof}

\begin{rmk}
In the special case when $G = \mathrm{Diff}(L)$, Corollary \ref{c:principal-H-bundle} was first proved in \cite{BinzFischer1981manifoldofembeddings}, and then generalised, by removing the assumption that $L$ is compact, in \cite[Theorem 10.14]{Michor1980Manifoldsofsmooth}. See also \cite[Theorem 13.14]{Michor1980Manifoldsofdifferentiable} and \cite[Theorem 44.1]{KrieglMichor1997convenientsettingof} for proofs of this result in the non-compact case. It seems likely that Proposition \ref{p:G-locally-retractile-orbitspace}, and therefore also Corollary \ref{c:principal-H-bundle} (for arbitrary open subgroups $G \leq \mathrm{Diff}(L)$), is also true for non-compact $L$, but we have not pursued this greater generality since we will only need the result in the case when $L$ is compact.
\end{rmk}

\begin{proof}[Proof of Propositions \ref{p:G-locally-retractile-embeddings} and \ref{p:G-locally-retractile-orbitspace}]
Denote the projection map by
\[
q \colon \mathrm{Emb}(L,N) \longrightarrow \mathrm{Emb}(L,N) / G.
\]
Fix $e \in \mathrm{Emb}(L,N)$. We will construct a local section for the action of $\mathrm{Diff}_c(N)$ on $\mathrm{Emb}(L,N)/G$ at the point $q(e)$ in two steps, as a composition of two maps $b$ and $\theta$.\footnote{At the end we mention how to modify this to obtain instead a local section for the action of  $\mathrm{Diff}_c(N)$ on $\mathrm{Emb}(L,N)$ at the point $e$. Essentially, one discards the first step and uses $\theta$ directly.}

\textbf{Step 1.}
First choose a tubular neighbourhood for $e$, namely an embedding $t_e \colon \nu_e \hookrightarrow N$ such that $t_e \circ o_e = e$, where $\pi_e \colon \nu_e \to L$ denotes the normal bundle of $e$ and $o_e \colon L \to \nu_e$ denotes its zero section. Note that $V = t_e(\nu_e)$ is an open neighbourhood of $e(L)$ in $N$, so $\mathrm{Emb}(L,V)$ is an open neighbourhood of $e$ in $\mathrm{Emb}(L,N)$ (\cf Fact \ref{fact:topological-embedding}). The map
\[
\Phi_e = \pi_e \circ t_e^{-1} \circ - \colon \mathrm{Emb}(L,V) \longrightarrow C^\infty(L,L)
\]
is continuous by Fact \ref{fact:composition-continuous}. Define $\cW_e = \Phi_e^{-1}(\mathrm{id})$ and $\cU_e = \Phi_e^{-1}(G)$. Note that $e \in \cW_e \subseteq \cU_e$ and that $\cU_e$ is open in $\mathrm{Emb}(L,N)$ since we assumed that $G$ is open in $\mathrm{Diff}(L)$. Alternative descriptions of these two subsets are:
\begin{align*}
\cW_e  &= \{ t_e \circ s \mid s \text{ is a smooth section of } \pi_e \} \\
\cU_e  &= \{ t_e \circ s \circ g \mid s \text{ is a smooth section of } \pi_e \text{ and } g \in G \} .
\end{align*}
Note that $q^{-1}(q(\cW_e))$ is the $G$-orbit of $\cW_e$, which is $\cU_e$. Hence $q(\cW_e)$ is an open neighbourhood of $q(e)$ in the quotient space $\mathrm{Emb}(L,N)/G$. Also note that $q|_{\cW_e}$ is \emph{injective}, since sections are equal if and only if they have the same image.

\begin{sublem}
The continuous injection $q|_{\cW_e}$ is a homeomorphism onto its image.
\end{sublem}
\begin{proof}[Proof of the sublemma]
\let\qedsymboloriginal\qedsymbol
\renewcommand{\qedsymbol}{{\small (sublemma)} \qedsymboloriginal}
Let $U$ be an open subset of $\mathrm{Emb}(L,N)$. We need to show that $q(U \cap \cW_e)$ is open in $q(\cW_e)$. Since $q(\cW_e)$ is open in $\mathrm{Emb}(L,N)/G$, this is equivalent to showing that $q(U \cap \cW_e)$ is open in $\mathrm{Emb}(L,N)/G$, i.e.\ that $q^{-1}(q(U \cap \cW_e))$ is open in $\mathrm{Emb}(L,N)$. Define
\[
\Psi_e \colon \cU_e \longrightarrow \cW_e
\]
by $f \mapsto f \circ \Phi_e(f)^{-1}$. This is continuous by Fact \ref{fact:composition-continuous} and the fact that inversion is continuous for diffeomorphism groups. Thus $\Psi_e^{-1}(U \cap \cW_e)$ is open in $\cU_e$, and hence in $\mathrm{Emb}(L,N)$. Now we note that $\Psi_e^{-1}(U \cap \cW_e)$ is precisely the $G$-orbit of $U \cap \cW_e$, which is $q^{-1}(q(U \cap \cW_e))$.
\end{proof}

Denote the inverse of $q|_{\cW_e}$ by
\[
b \colon q(\cW_e) \longrightarrow \cW_e.
\]
By the sublemma, it is continuous. Also, note that $b(q(e)) = e$.

\textbf{Step 2.}
Now (by \cite{Whitney1936Differentiablemanifolds}) we may choose a proper embedding
\[
\iota \colon N \lhto \bR^k
\]
for some $k$. We also choose a tubular neighbourhood for $\iota$, namely an embedding $t \colon \nu_\iota \hookrightarrow \bR^k$ such that $t \circ o_\iota = \iota$, where $\pi_\iota \colon \nu_\iota \to N$ is the normal bundle of $\iota$ and $o_\iota \colon N \to \nu_\iota$ is its zero section.

We will choose a tubular neighbourhood for the composition $\iota \circ e = \iota e \colon L \hookrightarrow \bR^k$ more carefully. There is an embedding of bundles
\begin{center}
\begin{tikzpicture}
[x=1mm,y=1mm]
\node (tl) at (0,12) {$\nu_{\iota e}$};
\node (tr) at (30,12) {$T\bR^k$};
\node (bl) at (0,0) {$L$};
\node (br) at (30,0) {$\bR^k$};
\inclusion{above}{$u$}{(tl)}{(tr)}
\inclusion{below}{$\iota e$}{(bl)}{(br)}
\draw[->] (tl) to node[left,font=\small]{$\pi_{\iota e}$} (bl);
\draw[->] (tr) to (br);
\end{tikzpicture}
\end{center}
so that $u(\nu_{\iota e})$ is the orthogonal complement of $T(\iota e(L))$ in $T\bR^k|_{\iota e(L)}$ with respect to the standard Riemannian metric on $\bR^k$. Write $D_r = D_r(\nu_{\iota e})$ for the subbundle of $\nu_{\iota e}$ consisting of vectors of norm $\leq r$, using the norm inherited via this embedding. Since $L$ is compact, we may choose $r>0$ sufficiently small so that the exponential map for $T\bR^k$ restricts to an embedding
\[
v \colon D_r \lhto \bR^k
\]
such that $vo_{\iota e} = \iota e$, where $o_{\iota e} \colon L \to D_r$ is the zero section of
\[
\hat{\pi} = \pi_{\iota e}|_{D_r} \colon D_r \longrightarrow L.
\]
We note that:
\begin{itemizeb}
\item[(i)] The restriction of $v$ to each fibre of $\hat{\pi} \colon D_r \to L$ is an isometry.
\item[(ii)] Moreover, we have $v(D_r) = \{ x \in \bR^k \mid \exists\, y \in L \text{ such that } \lvert \iota e(y) - x \rvert \leq r \}$.
\end{itemizeb}
Now we decrease $r > 0$ if necessary so that we also have
\begin{itemizeb}
\item[(iii)] $v(D_r) \subseteq t(\nu_\iota)$.
\end{itemizeb}
(Again, this is possible since $L$ is compact.) Define
\[
\cV_e = \{ f \in \mathrm{Emb}(L,N) \mid \forall\, x \in L \text{ we have } \lvert \iota f(x) - \iota e(x) \rvert < \tfrac{r}{2} \}
\]
and note that this is a $\mathrm{Diff}(L)$-equivariant open subset of $\mathrm{Emb}(L,N)$. Also choose a smooth map $\lambda \colon [0,\infty) \to [0,1]$ such that $\lambda(t) = 1$ for $t \leq \tfrac{r}{8}$ and $\lambda(t) = 0$ for $t \geq \tfrac{r}{4}$.

For any $f \in \cV_e$ we now attempt to define a map $\theta(f) \colon N \to N$ as follows.
\begin{itemizeb}
\item[(1)] If $x \in N \smallsetminus \iota^{-1}v(D_{r/4})$, we set $\theta(f)(x) = x$.
\item[(2)] If $x \in \iota^{-1}v(D_{r/2})$, we set $\theta(f)(x) = \pi_\iota t^{-1} \iota_f(x)$, where
\[
\iota_f(x) \;=\; \iota(x) + \lambda(\lvert \bar{x} \rvert) (\iota f \hat{\pi}(\bar{x}) - \iota e \hat{\pi}(\bar{x})),
\]
and $\bar{x} = v^{-1}\iota(x) \in D_{r/2}$.
\end{itemizeb}
To check that this is well-defined and smooth, it is enough to check that
\begin{itemizeb}
\item[(a)] $\iota_f(x)$ is in the image of $t$ in case (2),
\item[(b)] the two cases agree for $x \in \iota^{-1}v(D_{r/2} \smallsetminus D_{r/4})$.
\end{itemizeb}
The verification of (b) follows from the fact that $\lambda(t) = 0$ for $t \geq \tfrac{r}{4}$. For condition (a), suppose that $x \in \iota^{-1}v(D_{r/2})$ and let $\bar{x} = v^{-1}\iota(x) \in D_{r/2}$. Then
\[
\lvert \iota_f(x) - \iota e \hat{\pi}(\bar{x}) \rvert \;=\; \lvert \iota_f(x) - \iota(x) \rvert \; + \; \lvert \iota(x) - \iota e \hat{\pi}(\bar{x}) \rvert .
\]
For the first summand, we have
\[
\lvert \iota_f(x) - \iota(x) \rvert \leq \lvert \iota f \hat{\pi}(\bar{x}) - \iota e \hat{\pi}(\bar{x}) \rvert < \tfrac{r}{2}
\]
since $f \in \cV_e$. For the second summand, we have
\begin{align*}
\lvert \iota(x) - \iota e \hat{\pi}(\bar{x}) \rvert &= \lvert v(\bar{x}) - vo_{\iota e} \hat{\pi}(\bar{x}) \rvert \\
&= \lvert \bar{x} - o_{\iota e} \hat{\pi}(\bar{x}) \rvert \\
&\leq \tfrac{r}{2}
\end{align*}
by property (i) and since $\bar{x} \in D_{r/2}$. Hence $\iota_f(x) \in v(D_r) \subseteq t(\nu_\iota)$ by properties (ii) and (iii). This completes the verification of (a), so $\theta(f) \colon N \to N$ is a well-defined, smooth map. Note also that it is supported in $\iota^{-1}v(D_{r/2}) \subseteq N$, which is compact since $L$ (and therefore $D_{r/2}$) is compact and $\iota$ is a \emph{proper} embedding. So we have a function
\[
\theta \colon \cV_e \longrightarrow C_c^\infty(N,N),
\]
and it is not hard to see from Fact \ref{fact:composition-continuous}, plus the fact that pointwise addition and multiplication of smooth maps into $\bR^k$ are continuous, that $\theta$ is continuous.

\textbf{Step 3.} (Putting together $b$ and $\theta$.)
Recall that $\cV_e$ is a $G$-equivariant open subset of $\mathrm{Emb}(L,N)$, so $q(\cV_e)$ is an open subset of $\mathrm{Emb}(L,N)/G$. Therefore, the intersection $q(\cV_e) \cap q(\cW_e)$ is an open neighbourhood of $q(e)$ in $\mathrm{Emb}(L,N)/G$. Recall the continuous inverse $b \colon q(\cW_e) \to \cW_e$ of $q|_{\cW_e}$ from Step 1. Note that
\[
b(q(\cV_e) \cap q(\cW_e)) \subseteq \cW_e \cap q^{-1}(q(\cV_e)) = \cW_e \cap \cV_e \subseteq \cV_e .
\]
We may therefore define
\[
\bar{\theta} = \theta \circ b \colon q(\cV_e) \cap q(\cW_e) \longrightarrow C_c^\infty(N,N)
\]
and observe that $\bar{\theta}(q(e)) = \theta(e) = \mathrm{id}$. Since $\mathrm{Diff}_c(N)$ is an open neighbourhood of the identity in $C_c^\infty(N,N)$, its preimage
\[
\cY_e = \bar{\theta}^{-1}(\mathrm{Diff}_c(N))
\]
is an open neighbourhood of $q(e)$ in $\mathrm{Emb}(L,N)/G$ and we have a continuous map
\[
\bar{\theta} \colon \cY_e \longrightarrow \mathrm{Diff}_c(N).
\]
It now remains to check that $q(\bar{\theta}(q(f)) \circ e) = q(f)$ for all $q(f) \in \cY_e$. We may as well assume that the representative $f \in \mathrm{Emb}(L,N)$ of $q(f)$ has been chosen so that $b(q(f)) = f$, since this is always possible, so this is equivalent to checking that $q(\theta(f) \circ e) = q(f)$. We will show that $\theta(f) \circ e = f$ as embeddings $L \hookrightarrow N$.

Let $z \in L$. Then $\iota e(z) = vo_{\iota e}(z)$, so $e(z) \in \iota^{-1}v(D_{r/2})$ and hence
\begin{align*}
\theta(f)(e(z)) &= \pi_\iota t^{-1} ( \iota e(z) + \lambda(0)(\iota f(z) - \iota e(z)) ) \\
&= \pi_\iota t^{-1} (\iota f(z)) \\
&= \pi_\iota t^{-1} t o_\iota f(z) \\
&= f(z).
\end{align*}

Hence $(\cY_e,\bar{\theta})$ is a local section at $q(e)$ for the action of $\mathrm{Diff}_c(N)$ on $\mathrm{Emb}(L,N)/G$.

\textbf{Note.}
The proof also shows that $\mathrm{Emb}(L,N)$ -- without taking a quotient -- is $\mathrm{Diff}_c(N)$-locally retractile. To see this, discard Step 1, which constructs the map $b$, and use the map $\theta$ directly, restricted to $\theta^{-1}(\mathrm{Diff}_c(N))$.
\end{proof}

\subsection{Fibration lemmas.}

Directly from the definition, one may see that products, compositions and pullbacks of Serre/Hurewicz fibrations are again Serre/Hurewicz fibrations. In the case of Serre fibrations there is a partial converse to the statement about pullbacks: if the pullback of a map along a surjective Serre fibration is a Serre fibration, then the original map must already be a Serre fibration. More explicitly:

\begin{lem}\label{l:Serre-fibration-pullback}
Suppose we have a pullback diagram in topological spaces
\begin{center}
\begin{tikzpicture}
[x=1mm,y=1mm]

\node (a) at (0,15) {$A$};
\node (c) at (15,15) {$C$};
\node (b) at (0,0) {$B$};
\node (d) at (15,0) {$D$};
\draw[->] (c) to (a);
\draw[->] (c) to node[right,font=\small]{$r$} (d);
\draw[->] (a) to node[left,font=\small]{$l$} (b);
\draw[->] (d) to node[below,font=\small]{$b$} (b);
\node at (12,12) {$\llcorner$};

\end{tikzpicture}
\end{center}
in which $r$ and $b$ are Serre fibrations and $b$ is surjective. Then $l$ is also a Serre fibration.
\end{lem}

\begin{proof}
Consider a lifting problem
\begin{center}
\begin{tikzpicture}
[x=1mm,y=1mm]

\node (x) at (-25,15) {$[0,1]^{k-1}$};
\node (y) at (-25,0) {$[0,1]^k$};
\node (a) at (0,15) {$A$};
\node (c) at (15,15) {$C$};
\node (b) at (0,0) {$B$};
\node (d) at (15,0) {$D$};
\draw[->] (c) to (a);
\draw[->] (c) to node[right,font=\small]{$r$} (d);
\draw[->] (a) to node[left,font=\small]{$l$} (b);
\draw[->] (d) to node[below,font=\small]{$b$} (b);
\node at (12,12) {$\llcorner$};
\inclusion{left}{$i$}{(x)}{(y)}
\draw[->] (x) to node[above,font=\small]{$f$} (a);
\draw[->] (y) to node[below,font=\small]{$g$} (b);

\end{tikzpicture}
\end{center}
for $k\geq 0$. Since $b$ is surjective we may lift one corner of the map $g$ to $D$. Then, using the lifting property for $b$ ($k$ times) we may lift the whole map $g$ to a map $\bar{g}\colon [0,1]^k \to D$. The maps $\bar{g}\circ i$ and $f$ induce a map $\bar{f}\colon [0,1]^{k-1} \to C$ by the universal property of the pullback. We may now find a map $\bar{h}\colon [0,1]^k \to C$ solving the lifting problem $(\bar{f},\bar{g})$ and compose with the map $C\to A$ to solve the original lifting problem $(f,g)$.
\end{proof}

\begin{coro}\label{c:fibration-orbit-spaces}
Let $f\colon X\to Y$ be an equivariant map between $G$-spaces, and assume that it is additionally a Serre fibration. Suppose that $Y \to Y/G$ is a principal $G$-bundle. Then the map of orbit spaces $\bar{f}\colon X/G \to Y/G$ is also a Serre fibration.
\end{coro}
\begin{proof}
The facts that $Y \to Y/G$ is a principal $G$-bundle and $f \colon X \to Y$ is $G$-equivariant imply that $X \to X/G$ is also a principal $G$-bundle. Hence the square
\begin{center}
\begin{tikzpicture}
[x=1mm,y=1mm]
\node (tl) at (0,15) {$X/G$};
\node (tr) at (20,15) {$X$};
\node (bl) at (0,0) {$Y/G$};
\node (br) at (20,0) {$Y$};
\draw[->] (tr) to node[right,font=\small]{$f$} (br);
\draw[->] (tl) to node[left,font=\small]{$\bar{f}$} (bl);
\draw[->] (tr) to (tl);
\draw[->] (br) to (bl);
\end{tikzpicture}
\end{center}
is a morphism of principal $G$-bundles, and therefore a pullback square. Any principal $G$-bundle is a surjective Serre fibration, so this square satisfies the hypotheses of Lemma \ref{l:Serre-fibration-pullback}.
\end{proof}

\begin{coro}\label{c:fibration-orbit-spaces2}
The map $\bar{\pi}_n$ in \textup{\hyperref[para:labelled-submanifolds]{\S\anglenumber{iv.}}} is a Serre fibration with path-connected fibres.
\end{coro}
\begin{proof}
By assumption, the map $\pi \colon Z \to \mathrm{Emb}(P,\mbar)$ is a $G$-equivariant Serre fibration with path-connected fibres. Hence its $n$th power $\pi^n \colon Z^n \to \mathrm{Emb}(P,\mbar)^n$ is a $(G \wr \Sigma_n)$-equivariant Serre fibration with path-connected fibres. The pullback $\pi_n$ of $\pi^n$ along the inclusion $\mathrm{Emb}(nP,M) \hookrightarrow \mathrm{Emb}(P,\mbar)^n$ is therefore also a $(G \wr \Sigma_n)$-equivariant Serre fibration with path-connected fibres.

Since $G \wr \Sigma_n$ is an open subgroup of $\mathrm{Diff}(nP)$, Corollary \ref{c:principal-H-bundle} implies that
\[
\mathrm{Emb}(nP,M) \longrightarrow \mathrm{Emb}(nP,M)/(G \wr \Sigma_n) = \mathbf{C}_{nP}(M;G)
\]
is a principal $(G \wr \Sigma_n)$-bundle. Then Corollary \ref{c:fibration-orbit-spaces} implies that $\bar{\pi}_n$ is a Serre fibration. Moreover, the proof of Corollary \ref{c:fibration-orbit-spaces} shows that
\begin{center}
\begin{tikzpicture}
[x=1mm,y=1mm]
\node (tl) at (-35,15) {$\mathbf{C}_{nP}(M,Z;G)$};
\node (tm) at (0,15) {$Z_n/(G \wr \Sigma_n)$};
\node (tr) at (40,15) {$Z_n$};
\node (bl) at (-35,0) {$\mathbf{C}_{nP}(M;G)$};
\node (bm) at (0,0) {$\mathrm{Emb}(nP,M)/(G \wr \Sigma_n)$};
\node (br) at (40,0) {$\mathrm{Emb}(nP,M)$};
\draw[->] (tr) to node[right,font=\small]{$\pi_n$} (br);
\draw[->] (tm) to (bm);
\draw[->] (tl) to node[left,font=\small]{$\bar{\pi}_n$} (bl);
\draw[->] (tr) to (tm);
\draw[->] (br) to (bm);
\draw (tm) edge[double equal sign distance] (tl);
\draw (bm) edge[double equal sign distance] (bl);
\end{tikzpicture}
\end{center}
is a pullback square, so the fibres of $\bar{\pi}_n$ are homeomorphic to fibres of $\pi_n$ -- in particular, they are path-connected.
\end{proof}

%%%%%%%%%%%%%%%%%%%%%%%%%%%%%%%%%%%%%%%%%%%%%%%%%%%%%%%%%
%%%%%%%%%%%%%%%%%%%%%%%%%%%%%%%%%%%%%%%%%%%%%%%%%%%%%%%%%
\section{Proof of stability}\label{s:proof}

In this section we apply Theorem \ref{tAxiomatic} to prove our main theorem. Fix integers $m$ and $p$ satisfying $0 \leq p \leq \frac12(m-3)$ and let $n$ be a non-negative integer. Let $\X(n)$ be the collection of all maps $f \colon X \to Y$ as constructed in Definition \ref{d:input-data}, for all choices of $(M,P,\lambda,\iota,G)$, where $M$ is a smooth, connected $m$-manifold, $P$ is a smooth, closed $p$-manifold, $\iota$ is an embedding $P \hookrightarrow \partial M$, $\lambda$ is a proper embedding $\partial M \times [0,2] \hookrightarrow M$ with $\lambda(-,0) = \text{inclusion}$ and $G$ is an open subgroup of $\mathrm{Diff}(P)$.

To prove Lemma \ref{l:injectivity} and Theorem \ref{tmain} we will show that every map in $\bigcup_{n\geq 0}\X(n)$ is split-injective on homology and the sequence $\X(n)$ satisfies condition (2w) of Theorem \ref{tAxiomatic}. In the next section we spell this out precisely, as a guide to the proof.

As a side remark, although we have fixed the dimensions $m$ and $p$ such that they satisfy the hypothesis $p \leq \frac12(m-3)$ throughout this section, whenever this dimension hypothesis (or something weaker than it) is used in the proof of a lemma, proposition or corollary, it will be stated explicitly, in order to keep track of where it is needed in the proof.

\subsection{Overview of the proof.}\label{s:overview}

We first show in \S\ref{s:split-injectivity} that every map $f \in \bigcup_{n\geq 0}\X(n)$ induces split injections on homology in all degrees.

Now fix $n\geq 2$ and $f \colon X \to Y$ in $\X(n)$. We construct in \S\ref{s:firstr} a map $f_\bullet \colon X_\bullet \to Y_\bullet$ of augmented semi-simplicial spaces whose $(-1)$st level is $f \colon X \to Y$, and show that $\hconn(\lVert X_\bullet \rVert \to X) \geq n-1$ and $\hconn(\lVert Y_\bullet \rVert \to Y) \geq n$.

Fix $i\geq 0$. In \S\ref{s:firsta} we construct a commutative square
\begin{equation}\label{e:firsta}
\centering
\begin{split}
\begin{tikzpicture}
[x=1mm,y=1mm]
\node (tl) at (0,10) {$X_i$};
\node (tr) at (20,10) {$Y_i$};
\node (bl) at (0,0) {$A_i$};
\node (br) at (20,0) {$B_i$};
\draw[->] (tl) to node[above,font=\small]{$f_i$} (tr);
\draw[->] (tl) to node[left,font=\small]{$p_i$} (bl);
\draw[->] (tr) to node[right,font=\small]{$q_i$} (br);
\incl{(bl)}{(br)}
\end{tikzpicture}
\end{split}
\end{equation}
where $A_i$ is path-connected, $p_i$ and $q_i$ are Serre fibrations and the inclusion $A_i \hookrightarrow B_i$ is a homotopy equivalence, and we show, for a certain point $a_i \in A_i$, that the restriction of $f_i$ to $p_i^{-1}(a_i) \to q_i^{-1}(a_i)$ is in $\X(n-i-1)$. In the case $i=0$ we also construct a commutative square
\begin{equation}
\centering
\begin{split}
\begin{tikzpicture}
[x=1mm,y=1mm]
\node (tl) at (0,10) {$\xbar$};
\node (tr) at (40,10) {$\ybar$};
\node (bl) at (0,0) {$p_0^{-1}(a_0)$};
\node (br) at (40,0) {$q_0^{-1}(a_0)$};
\draw[->] (tl) to node[above,font=\small]{$g$} (tr);
\draw[->] (tl) to (bl);
\draw[->] (tr) to (br);
\draw[->] (bl) to node[below,font=\small]{restriction of $f_0$} (br);
\end{tikzpicture}
\end{split}
\end{equation}
in which the vertical maps are homeomorphisms. (This last step is of course not very significant: the map $g$ is just slightly more convenient to work with than the restriction of $f_0$ for the next constructions.)

In \S\ref{s:secondr} we construct a map $g_\bullet \colon \xbar_\bullet \to \ybar_\bullet$ of augmented semi-simplicial spaces whose $(-1)$st level is $g \colon \xbar \to \ybar$, and we show that $\lVert \xbar_\bullet \rVert \to \xbar$ and $\lVert \ybar_\bullet \rVert \to \ybar$ are weak equivalences.

Fix $j\geq 0$. In \S\ref{s:seconda} we construct another commutative square
\begin{equation}\label{e:seconda}
\centering
\begin{split}
\begin{tikzpicture}
[x=1mm,y=1mm]
\node (tl) at (0,12) {$\xbar_j$};
\node (tr) at (20,12) {$\ybar_j$};
\node (bl) at (0,0) {$\abar_j$};
\node (br) at (20,0) {$\bbar_j$};
\draw[->] (tl) to node[above,font=\small]{$g_j$} (tr);
\draw[->] (tl) to node[left,font=\small]{$\bar{p}_j$} (bl);
\draw[->] (tr) to node[right,font=\small]{$\bar{q}_j$} (br);
\incl{(bl)}{(br)}
\end{tikzpicture}
\end{split}
\end{equation}
where $\bar{p}_j$ and $\bar{q}_j$ are Serre fibrations and the inclusion $\abar_j \hookrightarrow \bbar_j$ is a homotopy equivalence. We also show that, for every point $\bar{a} \in \abar_j$, the restriction of $g_j$ to $\bar{p}_j^{-1}(\bar{a}) \to \bar{q}_j^{-1}(\bar{a})$ is in $\X(n-1)$.

Finally, fix $a_0 \in A_0$ as above and any point $\bar{a} \in \abar_0$. In \S\ref{s:weak-factorisation} we show that the composite map
\[
\bar{q}_0^{-1}(\bar{a}) \lhto \ybar_0 \longrightarrow \ybar \xrightarrow{\;\;\cong\;\;} q_0^{-1}(a_0) \lhto Y_0 \longrightarrow Y
\]
factors up to homotopy through $f \colon X \to Y$. (See Figure \ref{fhomotopy} for a picture.)

Combining all of these constructions from \S\S\ref{s:firstr}--\ref{s:weak-factorisation} with split-injectivity from \S\ref{s:split-injectivity} verifies condition (2w) for the sequence $\X(n)$, and therefore Theorem \ref{tAxiomatic} implies Theorem \ref{tmain}.

\subsection{Split-injectivity.}\label{s:split-injectivity}

Fix the input data $M,P,\lambda,\iota,G$ and temporarily write the map $f \colon X \to Y$ of Definition \ref{d:input-data} as $f_n \colon X_n \to Y_n$ to make the dependence on $n$ explicit (but we continue to hide the dependence on the other choices from the notation). By construction, $X_{n+1}$ is a subspace of $Y_n$. The only difference is that, in $Y_n$, the submanifolds are required to have image contained in $M_0$, whereas in $X_{n+1}$ they are required to have image contained in the slightly smaller manifold $M_1 \subset M_0$. Using the collar neighbourhood $\lambda$, it is easy to construct a deformation retraction for the inclusion $M_1 \hookrightarrow M_0$, which induces a deformation retraction for the inclusion $X_{n+1} \hookrightarrow Y_n$.

Now choosing a homotopy inverse for each of these inclusions, the maps $f_n$ give us a sequence
\[
\cdots \longrightarrow X_{n-1} \longrightarrow X_n \longrightarrow X_{n+1} \longrightarrow \cdots ,
\]
and taking reduced homology in any fixed degree $*$ we get a sequence
\[
0 = \widetilde{H}_*(X_0) \longrightarrow \widetilde{H}_*(X_1) \longrightarrow \widetilde{H}_*(X_2) \longrightarrow \widetilde{H}_*(X_3) \longrightarrow \cdots
\]
of abelian groups and homomorphisms. (Note that $X_0$ is the one-point space.)

\begin{lem}[Lemma 2 of \cite{Dold1962DecompositiontheoremsSn}]\label{l:Dold}
Suppose we have a sequence of abelian group homomorphisms $s_n \colon A_n \to A_{n+1}$ for $n\geq 0$ with $A_0 = 0$, together with homomorphisms $\tau_{k,n} \colon A_n \to A_k$ for $1 \leq k \leq n$ such that $\tau_{n,n} = \mathrm{id}$ and $\tau_{k,n} = \tau_{k,n+1} \circ s_n \;\mathrm{mod}\; \mathrm{im}(s_{k-1})$, i.e.\ $\mathrm{im}(\tau_{k,n} - \tau_{k,n+1} \circ s_n) \subseteq \mathrm{im}(s_{k-1})$. Then every $s_n$ is split-injective.
\end{lem}

To prove Lemma \ref{l:injectivity} it would therefore suffice to define maps $X_n \to X_k$ for each $1 \leq k \leq n$ such that the induced maps on reduced homology satisfy the hypotheses of this lemma. In fact we will define maps $X_n \to \mathrm{Sp}^b (X_k)$ where $b = \binom{n}{k}$. These induce homomorphisms $\widetilde{H}_*(X_n) \to \widetilde{H}_*(X_k)$ by sending a cycle in $X_n$ first to a cycle in $\mathrm{Sp}^b(X_k)$ and then viewing it as a formal sum of cycles in $X_k$. Alternatively, we may apply the Dold-Thom theorem \cite{DoldThom1958Quasifaserungenundunendliche} that $H_* = \pi_* \circ \mathrm{Sp}^\infty$ for $* \geq 1$ and note that $\mathrm{Sp}^\infty \circ \mathrm{Sp}^b = \mathrm{Sp}^\infty$.

\begin{defn}
Let $\tau_{k,n} \colon X_n \to \mathrm{Sp}^{\binom{n}{k}}(X_k)$ be the map taking $\{[\varphi_1],\ldots,[\varphi_n]\}$ to the formal sum of $\{[\varphi_i] \mid i \in S \}$ over all subsets $S \subseteq \{1,\ldots,n\}$ with $\lvert S \rvert = k$.
\end{defn}

Clearly $\tau_{n,n} = \mathrm{id}$, and there are homotopies
\[
\tau_{k,n+1} \circ f_n \;\simeq\; \tau_{k,n} + f_{k-1} \circ \tau_{k-1,n}
\]
of maps $X_n \longrightarrow \mathrm{Sp}^{\binom{n+1}{k}}(X_k)$. The hypotheses of Lemma \ref{l:Dold} are therefore satisfied after taking reduced homology, so Lemma \ref{l:Dold} implies Lemma \ref{l:injectivity}.

\subsection{Resolution by subconfigurations.}\label{s:firstr}

Let $n\geq 2$ and fix the input data $(M,P,\lambda,\iota,G)$ determining a map $f \colon X \to Y$ in $\X(n)$. Recall that we denote $\Emb(P,M)$ by $E$ and that $G$ is an open subgroup of $\mathrm{Diff}(P)$.

\begin{defn}
For $i\geq -1$ let $X_i$ be the subspace of $X \times (E/G)^{i+1}$ consisting of tuples
\[
(\{ [\varphi_1],\ldots,[\varphi_n] \}, ([\psi_0],\ldots,[\psi_i]) )
\]
such that the $[\psi_0],\ldots,[\psi_i]$ are pairwise distinct and $\{ [\psi_0],\ldots,[\psi_i] \} \,\subseteq\, \{ [\varphi_1],\ldots,[\varphi_n] \}$. There are obvious face maps $d_j \colon X_i \to X_{i-1}$, given by forgetting $[\psi_j]$, that turn this into an augmented semi-simplicial space $X_\bullet$ with $X_{-1} = X$. We define an augmented semi-simplicial space $Y_\bullet$ with $Y_{-1} = Y$ in exactly the same way, and there is a semi-simplicial map $f_\bullet \colon X_\bullet \to Y_\bullet$, defined by adjoining $[\lambda(-,\frac12)\circ\iota]$ to a configuration, such that $f_{-1} = f$.
\end{defn}

We now show that the given embedding $\iota \colon P \hookrightarrow \partial M$ may be extended to an embedding
\begin{equation}
\label{eq:iotabb}
\iotabb \colon P \times (-1,1) \times [0,2] \lhto M,
\end{equation}
which will be used several times for the constructions in the following sections. To make explicit how the relative dimension hypothesis on $M$ and $P$ is used, we split this into two statements in the proposition below. Recall that $m = \mathrm{dim}(M)$ and $p = \mathrm{dim}(P)$. In this section we are assuming that $p \leq \tfrac12(m-3)$, so in particular $p \leq \tfrac12(m-2)$.

\begin{prop}
\label{p:iotabb}
The assumption that $p \leq \tfrac12(m-2)$ implies that the normal bundle of the embedding $\iota \colon P \hookrightarrow \partial M$ admits a non-vanishing section. Whenever $\iota$ satisfies this property, it may be extended to an embedding of the form \eqref{eq:iotabb}.
\end{prop}
\begin{proof}
Write $\nu_\iota \colon N \to P$ for the normal bundle of $\iota$ and $o_\iota \colon P \to N$ for its zero section. Choose a metric on $\nu_\iota$ and denote by $S(\nu_\iota) \colon S(N) \to P$ its unit sphere subbundle with respect to this metric. This has fibres homeomorphic to $S^{m-p-2}$, so the obstructions to the existence of a section of $S(\nu_\iota)$ live in cohomology groups of the form $H^i(P;\pi_{i-1}(S^{m-p-2}))$. Since $P$ has dimension $p \leq m-p-2$, these groups all vanish, and so there exists a section of $S(\nu_\iota)$, thus a non-vanishing section of $\nu_\iota$.

For the second statement, choose a tubular neighbourhood $\iotab$ for $\iota$, i.e.\ an embedding $\iotab \colon N \hookrightarrow \partial M$ such that $\iotab \circ o_\iota = \iota$. A choice of non-vanishing section of $\nu_\iota$ determines a trivial one-dimensional subbundle $P \times \bR \to P$ of $\nu_\iota$, in particular an embedding $P \times \bR \hookrightarrow N$. Composing this with $\iotab$ and restricting to $P \times (-1,1)$, we obtain an embedding $\check{\iota} \colon P \times (-1,1) \hookrightarrow \partial M$. Note that $\check{\iota}(-,0) = \iota$. Now we may define the desired embedding $\iotabb$ by $\iotabb(z,s,t) = \lambda(\check{\iota}(z,s),t)$, using the given collar neighbourhood $\lambda$ of $\partial M$. Intuitively, the coordinate $s \in (-1,1)$ specifies a ``shift'' of the standard embedding $\iota \colon P \hookrightarrow \partial M$ parallel to the boundary, and the coordinate $t \in [0,2]$ specifies a shift orthogonal to the boundary, into the interior of $M$ if $t>0$.
\end{proof}

For a space $Z$, denote the configuration space of $n$ unordered, distinct points in $Z$ by $C_n(Z)$. There is then a map
\[
\mathit{st}_n \colon C_n((-1,1) \times (1,2)) \longrightarrow X
\]
given by $\mathit{st}_n(\{ (s_1,t_1),\ldots,(s_n,t_n) \}) = \{ [\iotabb(-,s_1,t_1)],\ldots,[\iotabb(-,s_n,t_n)] \}$. Note that the image of $\mathit{st}_n$ contains the \emph{standard configurations} of Definition \ref{d:standard}, corresponding to the configurations with $s_1 = s_2 = \cdots = s_n = 0$. Similarly, there is a map
\[
\mathit{st}_{n+1}^\prime \colon C_{n+1}((-1,1) \times (0,2)) \longrightarrow Y
\]
whose image contains the standard configurations in $Y$.

Now we show that $\hconn(\lVert X_\bullet \rVert \to X) \geq n-1$ and $\hconn(\lVert Y_\bullet \rVert \to Y) \geq n$. We will do this just in the first case, since the second case is almost identical. We first note that all face maps (and compositions of face maps) of $X_\bullet$ are covering maps. In particular, the augmentation $X_0 \to X$ is a covering map. Now, $X$ is path-connected by definition, and it is not hard to see that $X_0$ is also path-connected: we just need to show that two points in the same fibre of $X_0 \to X$ over $x \in X$ can be connected by a path in $X_0$, which we can do by first moving $x$ into the image of $\mathit{st}_n$ and then using the image of a braid in $C_n((-1,1) \times (1,2))$ to permute the components of the configuration of submanifolds $x$. Since the space of vertices $X_0$ is path-connected, so is the geometric realisation $\lVert X_\bullet \rVert$.\footnote{An alternative argument for path-connectivity of $\lVert X_\bullet \rVert$, avoiding the use of the map $\mathit{st}_n$ and hence the hypothesis of Proposition \ref{p:iotabb}, is as follows. We show below that the homotopy fibre of $\lVert X_\bullet \rVert \to X$ is $(n-2)$-connected, where $n\geq 2$, so it is at least path-connected. The codomain $X$ is path-connected by definition. Hence the domain $\lVert X_\bullet \rVert$ must also be path-connected.} The relative Hurewicz theorem therefore tells us that $\hconn(\lVert X_\bullet \rVert \to X) \geq n-1$ as long as the map $\lVert X_\bullet \rVert \to X$ is $(n-1)$-connected, in other words its homotopy fibre is $(n-2)$-connected.

We will use the following lemma, which appeared in an earlier, preprint version of \cite{Randal-Williams2016Resolutionsmodulispaces} (but not in the final, published version), and which is also very similar to Lemma 2.14 of \cite{EbertRandal-Williams2019Semi-simplicialspaces}. We also give a self-contained proof in the appendix \S\ref{s:appendix}.

\begin{lem}\label{l:homotopy-fibres}
For any augmented semi-simplicial space $X_\bullet$ and point $x \in X = X_{-1}$, the canonical map $\lVert \mathrm{hofib}_x(k_\bullet) \rVert \to \mathrm{hofib}_x(\lVert X_\bullet \rVert \to X)$ is a weak equivalence, where $k_n \colon X_n \to X$ denotes the unique composition of face maps $X_n \to X_{n-1} \to \cdots \to X_0 \to X$.
\end{lem}

We note that we are implicitly using a specific point-set model for the homotopy fibre of a map, so that $\mathrm{hofib}_x(k_\bullet)$ is a semi-simplicial \emph{space}, and not just a semi-simplicial object in the homotopy category. This is explained in more detail in \S\ref{ss:homotopy-fibres}.

In our case, the maps $k_n \colon X_n \to X$ are covering maps, therefore Serre fibrations, so there is a levelwise weak equivalence $k_{\bullet}^{-1}(x) \to \mathrm{hofib}_x(f_\bullet)$, which then induces a weak equivalence on geometric realisations (\cf Theorem 2.2 of \cite{EbertRandal-Williams2019Semi-simplicialspaces}). Thus the homotopy fibre of $\lVert X_\bullet \rVert \to X$ is weakly equivalent to $\lVert k_{\bullet}^{-1}(x) \rVert$.

The semi-simplicial set $k_{\bullet}^{-1}(x)$ has as its set of $i$-simplices all ordered $(i+1)$-tuples of pairwise distinct elements of the set $k_0^{-1}(x) = x = \{ [\varphi_1],\ldots,[\varphi_n] \} \in X$. This is often called the \emph{complex of injective words} on $n$ letters, and its geometric realisation is known to be homotopy equivalent to a wedge of $(n-1)$-spheres, see for example \cite[Proposition 3.3]{Randal-Williams2013Homologicalstabilityunordered}. In particular, this means that $\lVert k_{\bullet}^{-1}(x) \rVert$, and therefore the homotopy fibre of $\lVert X_\bullet \rVert \to X$, is $(n-2)$-connected, as claimed.

From the above discussion (and an identical argument in the case of $Y_\bullet$) we conclude:

\begin{lem}
The map $f_\bullet \colon X_\bullet \to Y_\bullet$ of augmented semi-simplicial spaces is an $n$-resolution of $f \colon X \to Y$, in other words $\hconn(\lVert X_\bullet \rVert \to X) \geq n-1$ and $\hconn(\lVert Y_\bullet \rVert \to Y) \geq n$.
\end{lem}

\subsection{The first approximation.}\label{s:firsta}

Fix $i\geq 0$. For a space $Z$, write $\widetilde{C}_k(Z)$ for the configuration space of $k$ ordered points in $Z$ and define a map
\[
\widetilde{\mathit{st}}_k \colon \widetilde{C}_k((-1,1) \times (1,2)) \longrightarrow (E/G)^k
\]
by $\widetilde{\mathit{st}}_k((s_1,t_1),\ldots,(s_k,t_k)) = ([\iotabb(-,s_1,t_1)],\ldots,[\iotabb(-,s_k,t_k)])$. The domain of $\widetilde{\mathit{st}}_k$ (the $k$th ordered configuration space of the plane) is path-connected, and therefore so is its image in $(E/G)^k$.

\begin{defn}
Let $A_i$ be the path-component of $\Emb((i+1)P,M_1)/G^{i+1} \subseteq (E/G)^{i+1}$ containing the standard configurations, i.e.\ the image of $\widetilde{\mathit{st}}_{i+1}$. Equivalently, $A_i$ is the subspace of $(E/G)^{i+1}$ consisting of ordered tuples $([\psi_0],\ldots,[\psi_i])$ such that the images $\psi_\alpha(P)$ are pairwise disjoint and contained in $M_1$, and there exists a path of such configurations starting at $([\psi_0],\ldots,[\psi_i])$ and ending at a standard configuration.

We define $B_i$ in exactly the same way, except that we replace $M_1$ with $M_0$.
\end{defn}

\begin{lem}
The inclusion $A_i \hookrightarrow B_i$ is a homotopy equivalence.
\end{lem}
\begin{proof}
Using the collar neighbourhood $\lambda$, we may construct a deformation retraction for the inclusion $M_1 \hookrightarrow M_0$, which induces a deformation retraction for the inclusion $A_i \hookrightarrow B_i$.
\end{proof}

\begin{defn}
Define $p_i \colon X_i \to A_i$ by sending $(\{ [\varphi_1],\ldots,[\varphi_n] \}, ([\psi_0],\ldots,[\psi_i]) )$ to $([\psi_0],\ldots,[\psi_i])$. In words, $p_i$ takes a configuration of submanifolds in which $i+1$ components have been marked (and ordered), and forgets all of the non-marked components. The map $q_i \colon Y_i \to B_i$ is defined in exactly the same way.
\end{defn}

These are clearly well-defined, and fit into a commutative square \eqref{e:firsta}.

\begin{lem}\label{l:Serre-fibrations-1}
The maps $p_i \colon X_i \to A_i$ and $q_i \colon Y_i \to B_i$ are Serre fibrations.
\end{lem}
\begin{proof}
By Proposition \ref{p:G-locally-retractile-orbitspace}, the action of $\mathrm{Diff}_c(M_1)$ on the quotient space $\Emb((i+1)P,M_1)/G^{i+1}$ is locally retractile. Since $\mathrm{Diff}_c(M_1)$ is locally path-connected by Fact \ref{fact:locally-contractible}, Lemma \ref{l:retractile-basic} implies that the restriction of this action to the identity path-component $\mathrm{Diff}_c(M_1)_0$ is also locally retractile. Embedding spaces are locally contractible as long as the domain manifold is compact (by Fact \ref{fact:locally-contractible} again), so in particular they are locally path-connected, and so their path-components are open. The property of having open path-components passes to quotient spaces, so $A_i$ is an open subspace of $\Emb((i+1)P,M_1)/G^{i+1}$. It is also $\mathrm{Diff}_c(M_1)_0$-invariant, since it is a path-component and the group $\mathrm{Diff}_c(M_1)_0$ is path-connected. Therefore Lemma \ref{l:retractile-basic} tells us that the action of $\mathrm{Diff}_c(M_1)_0$ on $A_i$ is locally retractile.

There is also a well-defined action of $\mathrm{Diff}_c(M_1)_0$ on $X_i$ given by post-composition, and $p_i \colon X_i \to A_i$ is equivariant with respect to these actions. Thus Proposition \ref{p:G-locally-retractile} implies that it is a fibre bundle, in particular a Serre fibration.

An almost identical argument, replacing $M_1$ by $M_0$ everywhere, shows that $q_i \colon Y_i \to B_i$ is also a fibre bundle, and therefore a Serre fibration.
\end{proof}

Now choose a point $a_i = ([\psi_0],\ldots,[\psi_i]) \in A_i$ such that each $\psi_\alpha(P)$ is contained in $M_2$ and define
\[
\mathrm{im}(a_i) = \psi_0(P) \cup \psi_1(P) \cup \ldots \cup \psi_i(P).
\]

\begin{prop}\label{p:two-conditions}
Assume that the normal bundle of the embedding $\iota \colon P \hookrightarrow \partial M$ admits a non-vanishing section. Then the following two conditions are equivalent, for a configuration
\[
\{ [\varphi_1],\ldots,[\varphi_{n-i-1}] \} \in \Emb((n-i-1)P,M_1 \smallsetminus \mathrm{im}(a_i)) / (G \wr \Sigma_{n-i-1}).
\]
\begin{itemizeb}
\item[\textup{(1)}] There is a path in this space from $\{ [\varphi_1],\ldots,[\varphi_{n-i-1}] \}$ to a standard configuration, i.e.\ one of the form $\{ [\lambda(\iota(-),t_1)] ,\ldots, [\lambda(\iota(-),t_{n-i-1})] \}$ for distinct $t_1,\ldots,t_{n-i-1}$ in the interval $(1,2)$, \cf Definition \ref{d:standard}.
\item[\textup{(2)}] There is a path in the space $\Emb(nP,M_1)/(G \wr \Sigma_n)$ from $\{ [\varphi_1] ,\ldots, [\varphi_{n-i-1}] , [\psi_0] ,\ldots, [\psi_i] \}$ to a standard configuration, i.e.\ one of the form $\{ [\lambda(\iota(-),t_1)] ,\ldots, [\lambda(\iota(-),t_n)] \}$ for distinct $t_1,\ldots,t_n$ in the interval $(1,2)$.
\end{itemizeb}
\end{prop}

We defer the proof for a few paragraphs. This immediately implies:

\begin{coro}\label{c:two-conditions}
Under the same assumption as in Proposition \ref{p:two-conditions}, there is a canonical homeomorphism $p_i^{-1}(a_i) \cong X_{n-i-1}(M \smallsetminus \mathrm{im}(a_i))$ given by
\[
(\{ [\varphi_1],\ldots,[\varphi_n] \}, ([\psi_0],\ldots,[\psi_i]) ) \;\longmapsto\; \{ [\varphi_1],\ldots,[\varphi_{n-i-1}] \} ,
\]
where we assume that the indexing of the $[\varphi_1],\ldots,[\varphi_n]$ is chosen so that $[\psi_\alpha] = [\varphi_{n-\alpha}]$.
\end{coro}

\begin{rmk}
The notation $X_{n-i-1}(M \smallsetminus \mathrm{im}(a_i))$ means the space $X$ from Definition \ref{d:input-data}, with the same $P,\lambda,\iota,G$ as before, but with $n$ replaced by $n-i-1$ and with $M$ replaced by its open submanifold $M \smallsetminus \mathrm{im}(a_i)$.
\end{rmk}

\begin{proof}[Proof of Corollary \ref{c:two-conditions}]
The map given above defines a homeomorphism from $p_i^{-1}(a_i)$ onto the subspace of
\[
\Emb((n-i-1)P,M_1 \smallsetminus \mathrm{im}(a_i)) / (G \wr \Sigma_{n-i-1})
\]
consisting of those elements that satisfy condition (2) of Proposition \ref{p:two-conditions}. By definition, the space $X_{n-i-1}(M \smallsetminus \mathrm{im}(a_i))$ is the subspace of $\Emb((n-i-1)P,M_1 \smallsetminus \mathrm{im}(a_i)) / (G \wr \Sigma_{n-i-1})$ consisting of those elements that satisfy condition (1) of Proposition \ref{p:two-conditions}. The result therefore follows from Proposition \ref{p:two-conditions}.
\end{proof}

By exactly the same argument, replacing $M_1$ with $M_0$ everywhere, we also have a canonical homeomorphism $q_i^{-1}(a_i) \cong Y_{n-i-1}(M \smallsetminus \mathrm{im}(a_i))$, and these fit into a commutative diagram
\begin{equation}
\centering
\begin{split}
\begin{tikzpicture}
[x=1mm,y=1mm]
\node (tl) at (0,10) {$X_{n-i-1}(M \smallsetminus \mathrm{im}(a_i))$};
\node (tr) at (80,10) {$Y_{n-i-1}(M \smallsetminus \mathrm{im}(a_i))$};
\node (bl) at (0,0) {$p_i^{-1}(a_i)$};
\node (br) at (80,0) {$q_i^{-1}(a_i)$};
\draw[->] (tl) to node[above,font=\small]{$f_{n-i-1}(M \smallsetminus \mathrm{im}(a_i))$} (tr);
\node at (0,5) {\rotatebox{-90}{$\cong$}};
\node at (80,5) {\rotatebox{-90}{$\cong$}};
\draw[->] (bl) to node[below,font=\small]{restriction of $f_i$} (br);
\end{tikzpicture}
\end{split}
\end{equation}

Thus we have shown that the restriction of $f_i$ to $p_i^{-1}(a_i) \to q_i^{-1}(a_i)$ is in $\X(n-i-1)$.

\begin{proof}[Proof of Proposition \ref{p:two-conditions}]
First assume that we have a path of configurations $\{ [\varphi_1^t] ,\ldots, [\varphi_{n-i-1}^t] \}$ in $M_1 \smallsetminus \mathrm{im}(a_i)$ indexed by $t \in [0,1]$ with $[\varphi_\alpha^0] = [\varphi_\alpha]$ and $[\varphi_\alpha^1] = [\lambda(\iota(-),u_\alpha)]$ for distinct $u_1,\ldots,u_{n-i-1} \in (1,2)$. By definition of $A_i$, there is a path $([\psi_0^t] ,\ldots, [\psi_i^t])$ in $\Emb((i+1)P,M_1)/G^{i+1}$ with $[\psi_\alpha^0] = [\psi_\alpha]$ and $[\psi_\alpha^1] = [\lambda(\iota(-),v_\alpha)]$ for distinct $v_0,\ldots,v_i \in (1,2)$. By compactness, we may choose some $\epsilon > 0$ such that $\psi_\alpha^t(P) \subseteq M_{1+\epsilon}$ for all $\alpha$ and all $t$, in particular, $v_0,\ldots,v_i > 1+\epsilon$. Using this, we construct a path in $\Emb(nP,M_1)/(G \wr \Sigma_n)$ from $\{ [\varphi_1] ,\ldots, [\varphi_{n-i-1}] , [\psi_0] ,\ldots, [\psi_i] \}$ to a standard configuration in three steps.
\begin{itemizeb}
\item[(i)] The path $\{ [\varphi_1^t] ,\ldots, [\varphi_{n-i-1}^t] , [\psi_0] ,\ldots, [\psi_i] \}$ ends at the configuration
\[
\{ [\lambda(\iota(-),u_1)] ,\ldots, [\lambda(\iota(-),u_{n-i-1})] , [\psi_0] ,\ldots, [\psi_i] \}
\]
for some $u_1,\ldots,u_{n-i-1} \in (1,2)$.
\item[(ii)] Next, keep the $[\psi_0] ,\ldots, [\psi_i]$ fixed and move the other part of the configuration by gradually decreasing the values of $u_1,\ldots,u_{n-i-1}$ until they all lie in $(1,1+\epsilon)$. Note that this will not intersect any of $\psi_0(P), \ldots, \psi_i(P)$ since the latter are contained in $M_2$ by our assumption on the point $a_i \in A_i$.
\item[(iii)] Now the path $\{ [\lambda(\iota(-),u_1)] ,\ldots, [\lambda(\iota(-),u_{n-i-1})] , [\psi_0^t] ,\ldots, [\psi_i^t] \}$ ends at a standard configuration.
\end{itemizeb}
Hence we have shown that $\text{(1)} \Rightarrow \text{(2)}$.\footnote{A slight variation of the argument for this implication is as follows. One may show, by similar reasoning, that the subset of configurations satisfying condition (1) is a path-component of $\Emb((n-i-1)P,M_1 \smallsetminus \mathrm{im}(a_i)) / (G \wr \Sigma_{n-i-1})$, and the subset of configurations satisfying condition (2) is a non-empty union of path-components. Thus, once we have proven the opposite implication $\text{(2)} \Rightarrow \text{(1)}$ (which is proven just below), the implication $\text{(1)} \Rightarrow \text{(2)}$ is automatic.}

For the proof of the opposite implication, we will use the equivalent characterisation of condition (1) that \emph{there is a path from $\{ [\varphi_1] ,\ldots, [\varphi_{n-i-1}] \}$ to the image of $\mathit{st}_{n-i-1}$} and of condition (2) that \emph{there is a path from $\{ [\varphi_1] ,\ldots, [\varphi_{n-i-1}] , [\psi_0] ,\ldots, [\psi_i] \}$ to the image of $\mathit{st}_n$}. By the assumption that the normal bundle of $\iota$ admits a non-vanishing section, we may apply Proposition \ref{p:iotabb} and extend $\iota$ to an embedding $\iotabb$ of the form \eqref{eq:iotabb}, so the maps $\mathit{st}_n$ are indeed defined.

Assume that we are given a path of configurations
\[
\gamma \colon [0,1] \longrightarrow \Emb(nP,M_1)/(G \wr \Sigma_n)
\]
with $\gamma(0) = \{ [\varphi_1] ,\ldots, [\varphi_{n-i-1}] , [\psi_0] ,\ldots, [\psi_i] \}$ and $\gamma(1)$ in the image of $\mathit{st}_n$. We will write
\[
\gamma(t) = \{ [\varphi_1^t] ,\ldots, [\varphi_{n-i-1}^t] , [\psi_0^t] ,\ldots, [\psi_i^t] \}
\]
and note that it makes sense to talk about the various components $[\varphi_\alpha^t]$ and $[\psi_\alpha^t]$ individually as well as together, since they may be distinguished using unique path-lifting for the covering space $\Emb(nP,M_1)/G^n \to \Emb(nP,M_1)/(G \wr \Sigma_n)$. Without loss of generality, we may arrange that
\begin{itemizeb}
\item[(a)] $\psi_\alpha^t(P) \subseteq M_{1.5}$ for all $\alpha$ and $t$.
\item[(b)] $\varphi_\alpha^1(P) \subseteq M_1 \smallsetminus M_{1.5} = \lambda(\partial M \times (1,1.5])$ for all $\alpha$.
\end{itemizeb}
Condition (a) is possible to arrange, using the collar neighbourhood $\lambda$, since we have assumed that $\psi_\alpha^0(P) \subseteq M_2$. For condition (b): once we have arrived in the image of $\mathit{st}_n$, we may choose an appropriate path in the configuration space $C_n((-1,1) \times (1,2))$ whose image under $\mathit{st}_n$ fixes the $[\psi_\alpha^1]$ and moves the $[\varphi_\alpha^1]$ into $M_1 \smallsetminus M_{1.5}$.

Define a path $\gamma^\prime \colon [0,1] \to \Emb((i+1)P,M_{1.5})/(G \wr \Sigma_{i+1})$ by $t \mapsto \{ [\psi_0^t], \ldots, [\psi_i^t] \}$ and define
\begin{equation}\label{e:isotopy-extension}
\mathrm{Diff}_c(M_{1.5}) \longrightarrow \Emb((i+1)P,M_{1.5})/(G \wr \Sigma_{i+1})
\end{equation}
by $\Phi \mapsto \{ [\Phi \circ \psi_0], \ldots, [\Phi \circ \psi_i] \}$. By Propositions \ref{p:G-locally-retractile-orbitspace} and \ref{p:G-locally-retractile} this is a fibre bundle, thus a Serre fibration, so we may find a lift $\gamma'' \colon [0,1] \to \mathrm{Diff}_c(M_{1.5})$ of $\gamma'$ such that $\gamma''(0)$ is the identity. We now define a path
\[
\gamma''' \colon [0,1] \longrightarrow \Emb((n-i-1)P,M_1 \smallsetminus \mathrm{im}(a_i)) / (G \wr \Sigma_{n-i-1})
\]
by $\gamma'''(t) = \{ [\gamma''(t)^{-1} \circ \varphi_1^t] ,\ldots, [\gamma''(t)^{-1} \circ \varphi_{n-i-1}^t] \}$, where we are implicitly extending compactly-supported diffeomorphisms of $M_{1.5}$ to $M_1$ by the identity on $M_1 \smallsetminus M_{1.5}$. This is now a path from $\{ [\varphi_1],\ldots,[\varphi_{n-i-1}] \}$ to the image of $\mathit{st}_{n-i-1}$. Hence we have shown that $\text{(2)} \Rightarrow \text{(1)}$.
\end{proof}

We finish this subsection by defining a slightly more convenient (and homeomorphic) model for the restriction of $f_0$ to $p_0^{-1}(a_0) \to q_0^{-1}(a_0)$.

\begin{defn}
Recall from \S\ref{s:firstr} (see the proof of Proposition \ref{p:iotabb}) that we have chosen a tubular neighbourhood for the embedding $\iota \colon P \hookrightarrow \partial M$. In other words, writing $\nu_\iota \colon N \to P$ for the normal bundle of $\iota$ and $o_\iota \colon P \to N$ for its zero section, we have chosen an embedding $\iotab \colon N \hookrightarrow \partial M$ such that $\iotab \circ o_\iota = \iota$. We have also chosen a metric on the bundle $\nu_\iota$. Let $D(\nu_\iota) \colon D(N) \to P$ denote the closed unit disc subbundle of $\nu_\iota$ with respect to this metric, and define
\[
T = \iotab(D(N)).
\]
This is a compact codimension-zero submanifold of $\partial M$ with boundary $\iotab(S(N))$, where $S(\nu_\iota) \colon S(N) \to P$ is the unit sphere subbundle of $\nu_\iota$ with respect to the chosen metric.
\end{defn}

\begin{defn}\label{d:xbar}
Let $\xbar$ be the path-component of
$\Emb((n-1)P,M_1 \smallsetminus \lambda(T \times \{2\})) / (G \wr \Sigma_{n-1})$
containing the image of $\mathit{st}_{n-1}$. Let $\ybar$ be the path-component of
$\Emb(nP,M_0 \smallsetminus \lambda(T \times \{2\})) / (G \wr \Sigma_n)$
containing the image of $\mathit{st}_n$. There is a continuous map
\[
g \colon \xbar \longrightarrow \ybar
\]
defined by $\{ [\varphi_1] ,\ldots, [\varphi_{n-1}] \} \;\longmapsto\; \{ [\iotabb(-,0,\tfrac12)] , [\varphi_1] ,\ldots, [\varphi_{n-1}] \}$.

(See Proposition \ref{p:iotabb} for the construction of the embedding $\iotabb$ extending $\iota$. Recall, in particular, that it satisfies $\iotabb(-,0,0) = \iota(-)$ and more generally $\iotabb(-,0,t) = \lambda(\iota(-),t)$.)
\end{defn}

Recall that we have fixed a point $a_0 = [\psi_0] \in A_0$ with $\psi_0(P) \subseteq M_2$. Choose a diffeomorphism
\[
\Psi \colon M \smallsetminus \psi_0(P) \longrightarrow M \smallsetminus \lambda(T \times \{2\})
\]
that restricts to the identity on $M \smallsetminus M_{1.5} = \lambda(\partial M \times [0,1.5])$. This exists because, firstly, $[\psi_0]$ has a path in $E/G$ to $[\iotabb(-,0,2)]$, by definition of $A_0$. Since the map \eqref{e:isotopy-extension} (with $i=0$) is a Serre fibration, we may lift this to a path of diffeomorphisms, evaluate at $1$, extend by the identity on $M \smallsetminus M_{1.5}$ and then restrict to obtain a diffeomorphism
\[
M \smallsetminus \psi_0(P) \longrightarrow M \smallsetminus \iotabb(P \times \{0\} \times \{2\})
\]
that restricts to the identity on $M \smallsetminus M_{1.5}$. Now, since $\lambda(T \times \{2\})$ is a tubular neighbourhood of $\iotabb(P \times \{0\} \times \{2\}) \subset \lambda(\partial M \times \{2\})$, it is easy to construct a diffeomorphism
\[
M \smallsetminus \iotabb(P \times \{0\} \times \{2\}) \longrightarrow M \smallsetminus \lambda(T \times \{2\})
\]
that restricts to the identity on $M \smallsetminus M_{1.5}$. (We note that, to construct this, we use the fact (assumed at the beginning of \S\ref{s:detailed-statements}) that the collar neighbourhood $\lambda \colon \partial M \times [0,2] \hookrightarrow M$ extends to a slightly larger collar neighbourhood $\partial M \times [0,2+\epsilon] \hookrightarrow M$.) Composing these two diffeomorphisms gives the desired diffeomorphism $\Psi$.

\begin{rmk}\label{r:diffeomorphism-Psi}
If we are careful in how we define the second diffeomorphism above, we may ensure the following useful property of $\Psi$. Let $s \colon P \to N$ be any section of the normal bundle $\nu_\iota$. Consider the half-open path $[0,2) \to \Emb(P,M_1 \smallsetminus \lambda(T \times \{2\}))$ given by $t \mapsto \lambda(-,t) \circ \iotab \circ s$. Postcomposing at each time $t$ with the diffeomorphism $\Psi^{-1}$ defines a half-open path $\gamma \colon [0,2) \to \Emb(P,M_1 \smallsetminus \psi_0(P))$. Then this path may be extended continuously to a path $[0,2] \to \Emb(P,M_1)$ by setting $\gamma(2) = \psi_0$.

A second useful property of $\Psi$ is that if we consider $\Psi^{-1}$ as an embedding $M \smallsetminus \lambda(T \times \{2\}) \hookrightarrow M$, then it is isotopic to the inclusion through embeddings that restrict to the identity on $M \smallsetminus M_{1.5}$.
\end{rmk}

\begin{lem}\label{l:Psi}
Postcomposition with $\Psi^{-1}$ defines homeomorphisms $\xbar \to p_0^{-1}(a_0)$ and $\ybar \to q_0^{-1}(a_0)$ such that
\begin{equation}\label{e:psi-inverse}
\centering
\begin{split}
\begin{tikzpicture}
[x=1mm,y=1mm]
\node (tl) at (0,10) {$\xbar$};
\node (tr) at (40,10) {$\ybar$};
\node (bl) at (0,0) {$p_0^{-1}(a_0)$};
\node (br) at (40,0) {$q_0^{-1}(a_0)$};
\draw[->] (tl) to node[above,font=\small]{$g$} (tr);
\draw[->] (tl) to (bl);
\draw[->] (tr) to (br);
\draw[->] (bl) to node[below,font=\small]{\textup{restriction of $f_0$}} (br);
\end{tikzpicture}
\end{split}
\end{equation}
commutes.
\end{lem}
\begin{proof}
This is immediate from the constructions. One easy but important observation is that if a configuration $\{ [\varphi_1] ,\ldots, [\varphi_{n-1}] \}$ has a path to the image of $\mathit{st}_{n-1}$, then so does the configuration $\{ [\Psi^{-1} \circ \varphi_1] ,\ldots, [\Psi^{-1} \circ \varphi_{n-1}] \}$, since $\Psi^{-1}$ is the identity on $\lambda(\partial M \times [0,1.5])$.
\end{proof}

\subsection{Resolution by tubes to the boundary.}\label{s:secondr}

\begin{defn}[The second resolution]\label{d:secondr}
Let $\xbar_0$ be the subspace of $\xbar \times \Emb(P \times [0,2],M)$ consisting of all elements $(\{ [\varphi_1] ,\ldots, [\varphi_{n-1}] \} , \tau)$ with the following properties.
\begin{itemizeb}
\item[(a)] There exist $h \in (\tfrac12,1)$ and $\epsilon \in (0,1)$ so that $\tau(-,t) = \iotabb(-,h,t)$ for all $t \in [0,1+\epsilon] \cup [2-\epsilon,2]$.
\item[(b)] The image $\tau(P \times (1,2))$ is contained in $M_1 \smallsetminus ( \varphi_1(P) \cup \varphi_2(P) \cup \cdots \cup \varphi_{n-1}(P) \cup \lambda(T \times \{2\}) )$.
\end{itemizeb}
There is an obvious map $\xbar_0 \to \xbar$ given by forgetting $\tau$. More generally, for $i\geq 0$, let $\xbar_i$ be the subspace of $\xbar \times \Emb(P \times [0,2],M)^{i+1}$ consisting of all elements $(\{ [\varphi_1] ,\ldots, [\varphi_{n-1}] \},(\tau_0,\ldots,\tau_i))$ with the following properties.
\begin{itemizeb}
\item[(ab)] Each $\tau_\alpha$ satisfies the properties (a) and (b) above.
\item[(c)] For $\alpha \neq \beta$, the images $\tau_\alpha(P \times [0,2])$ and $\tau_\beta(P \times [0,2])$ are disjoint.
\item[(d)] Write $h_\alpha \in (\tfrac12,1)$ for the number associated to $\tau_\alpha$ by condition (a). Then $h_0 < h_1 < \cdots < h_i$.
\end{itemizeb}
There are maps $d_j \colon \xbar_i \to \xbar_{i-1}$ defined by forgetting $\tau_j$. These obviously satisfy the simplicial identities, so they give $\xbar_\bullet = \{ \xbar_i \}_{i\geq 0} \cup \{ \xbar \}$ the structure of an augmented semi-simplicial space.

The augmented semi-simplicial space $\ybar_\bullet$ is defined similarly --- the space $\ybar_i$ is the subspace of $\ybar \times \Emb(P \times [0,2],M)$ consisting of all elements $(\{ [\varphi_1] ,\ldots, [\varphi_{n-1}] \},(\tau_0,\ldots,\tau_i))$ with the properties (c) and (d) above, as well as the following variants of (a) and (b).
\begin{itemizeb}
\item[({\=a})] There exist $h \in (\tfrac12,1)$ and $\epsilon \in (0,1)$ so that $\tau(-,t) = \iotabb(-,h,t)$ for all $t \in [0,\epsilon] \cup [2-\epsilon,2]$.
\item[({\=b})] The image $\tau(P \times (0,2))$ is contained in $M_0 \smallsetminus ( \varphi_1(P) \cup \varphi_2(P) \cup \cdots \cup \varphi_{n-1}(P) \cup \lambda(T \times \{2\}) )$.
\end{itemizeb}
There are again maps forgetting $\tau_j$ that give $\ybar_\bullet = \{ \ybar_i \}_{i\geq 0} \cup \{ \ybar \}$ the structure of an augmented semi-simplicial space. There is a map of augmented semi-simplicial spaces
\[
g_\bullet \colon \xbar_\bullet \longrightarrow \ybar_\bullet
\]
given by $(\{ [\varphi_1] ,\ldots, [\varphi_{n-1}] \} , (\tau_0,\ldots,\tau_i)) \;\longmapsto\; (\{ [\iotabb(-,0,\tfrac12)] , [\varphi_1] ,\ldots, [\varphi_{n-1}] \} , (\tau_0,\ldots,\tau_i))$ on spaces of $i$-simplices. Clearly $g_{-1} = g$, i.e.\ this extends $g \colon \xbar \to \ybar$ to augmented semi-simplicial spaces.
\end{defn}

\begin{rmk}
A vertex of $\xbar_\bullet$ is intuitively a ``tube'' with cross-section $P$ going from the boundary $\partial M$ to $\lambda(T \times \{2\})$, which is a thickened copy of $P$ in the interior of $M$. This tube must be ``straight'' in a certain sense near each end, and it must be disjoint from the configuration (and the interior of the tube must also be disjoint from $\lambda(T \times \{2\})$). An ordered collection of such tubes forms a simplex if and only if they are pairwise disjoint and the ordering coincides with the intrinsic ordering that they inherit from the boundary condition near $\partial M$.
\end{rmk}

\begin{rmk}
\label{r:iotabb}
Since the spaces under consideration from now on are only defined if we assume the existence of (and choose) an embedding $\hat{\iota}$ of the form \eqref{eq:iotabb} extending $\iota$ (see condition (a) and its variants of Definition \ref{d:secondr} above), we will not mention this assumption again for the remainder of this section. Recall that we are, in any case, assuming throughout this section the dimension hypothesis $p \leq \tfrac12(m-3)$, which implies the existence of such an embedding $\hat{\iota}$ by Proposition \ref{p:iotabb}.
\end{rmk}

Our task in this subsection is to show that the induced maps
\[
\lVert \xbar_\bullet \rVert \longrightarrow \xbar \qquad\text{and}\qquad \lVert \ybar_\bullet \rVert \longrightarrow \ybar
\]
are weak equivalences. We will do this explicitly just for $\lVert \xbar_\bullet \rVert \to \xbar$, since the other case is almost identical. We will use the following theorem due to Galatius and Randal-Williams.

\begin{thm}[Theorem 6.2 of \cite{GalatiusRandalWilliams2014Stablemodulispaces}]\label{t:grw}
If $Z_\bullet \to Z$ is an augmented semi-simplicial space, the following conditions imply that the map $\lVert Z_\bullet \rVert \to Z$ is a weak equivalence.
\begin{itemizeb}
\item[\textup{(i)}] The map $Z_i \to Z_0 \times_Z \cdots \times_Z Z_0$ taking an $i$-simplex to the ordered set of its $i+1$ vertices is a homeomorphism onto an open subspace.
\item[\textup{(ii)}] Under this identification, an $(i+1)$-tuple of vertices $(v_0,v_1,\ldots,v_i)$ lies in $Z_i$ if and only if $(v_\alpha,v_\beta)$ lies in $Z_1$ for all $\alpha < \beta$.
\item[\textup{(iii)}] Denote the augmentation map $Z_0 \to Z$ by $a$. For every point $v \in Z_0$ there is an open neighbourhood $U$ of $a(v) \in Z$ and a section $s \colon U \to Z_0$ of $a$ such that $s(a(v)) = v$.
\item[\textup{(iv)}] For any finite set $\{ v_1,\ldots,v_k \}$ in a fibre of $a$, there is another vertex $v$ in the same fibre such that $(v_\alpha,v) \in Z_1$ for all $\alpha \in \{ 1,\ldots,k \}$.
\end{itemizeb}
\end{thm}

\begin{rmk}
We note that, in \cite{GalatiusRandalWilliams2014Stablemodulispaces}, condition (iii) is slightly weaker and more complicated to state, and incorporates the $k=0$ part of condition (iv), i.e.\ surjectivity of $a \colon Z_0 \to Z$.
\end{rmk}

First note that conditions (i)--(iii) are clearly true for $Z_\bullet = \xbar_\bullet$. For (i) and (ii) this is because we defined an $i$-simplex to be an $(i+1)$-tuple of vertices satisfying conditions (c) and (d), which are open conditions that may be determined by looking at ordered sub-tuples of length $2$. For condition (iii), let $v = (\{ [\varphi_1] ,\ldots, [\varphi_{n-1}] \} , \tau) \in \xbar_0$. Define $U$ to be the open subspace of $\xbar$ consisting of all configurations $\{ [\varphi_1^{\prime}] ,\ldots, [\varphi_{n-1}^\prime] \}$ such that $\bigcup_{\alpha = 1}^{n-1} \varphi_{\alpha}^\prime (P)$ is disjoint from $\tau(P \times [0,2])$. This is an open neighbourhood of $a(v) = \{ [\varphi_1] ,\ldots, [\varphi_{n-1}] \}$ and we may define a section of $a \colon \xbar_0 \to \xbar$ on $U$ by
\[
s \colon U \longrightarrow \xbar_0 \qquad \{ [\varphi_1^{\prime}] ,\ldots, [\varphi_{n-1}^\prime] \} \;\longmapsto\; (\{ [\varphi_1^{\prime}] ,\ldots, [\varphi_{n-1}^\prime] \} , \tau),
\]
which sends $a(v)$ to $v$.

In order to verify condition (iv) for $Z_\bullet = \xbar_\bullet$ we will first take a detour to discuss transversality. One corollary of Thom's transversality theorem is the following.

\begin{thm}[Corollary II.4.12(b), page 56, \cite{GolubitskyGuillemin1973Stablemappingsand}]\label{t:gg}
Let $L$ and $N$ be smooth manifolds without boundary and $f \colon L \to N$ a smooth map. Let $W \subseteq N$ be a smooth submanifold and $A \subseteq B \subseteq L$ open subsets such that $\abar \subseteq B$. Let $\cU$ be an open neighbourhood of $f \in C^\infty(L,N)$ in the Whitney $C^\infty$-topology. Then there exists $g \in \cU$ such that
\begin{itemizeb}
\item[\textup{(1)}] $g|_A = f|_A$,
\item[\textup{(2)}] $g$ is transverse to $W$ on $L \smallsetminus B$.
\end{itemizeb}
\end{thm}

A useful corollary of this is the following.

\begin{coro}\label{c:transversality}
Let $L$ and $N$ be smooth manifolds without boundary and let $W \subseteq N$ be a smooth submanifold that is closed as a subset such that $\dim(W) + \dim(L) < \dim(N)$. Also, let $f \colon L \to N$ be a smooth map and $B \subseteq L$ an open subset such that $f(\bbar) \subseteq N \smallsetminus W$. Let $\cU$ be an open neighbourhood of $f \in C^\infty(L,N)$ in the Whitney $C^\infty$-topology. Then, for any open subset $A \subseteq B$ with $\abar \subseteq B$, there exists $g \in \cU$ such that
\begin{itemizeb}
\item[\textup{(1)}] $g|_A = f|_A$,
\item[\textup{({\^2})}] $g(L) \subseteq N \smallsetminus W$.
\end{itemizeb}
\end{coro}

This says, roughly, that if we have a smooth map $f$ whose image is disjoint from $W$ on a closed subset $\bbar$, then we may find a smooth map $g$ that is arbitrarily close to $f$ (in the sense that $g \in \cU$), agrees with $f$ on a slightly smaller subset (namely $A$), and whose entire image is disjoint from $W$.

\begin{proof}
Let $\cV = \{ g \in C^\infty(L,N) \mid g(\bbar) \subseteq N \smallsetminus W \}$. This condition is equivalent to requiring that $\Gamma_g \subseteq (L \times N) \smallsetminus (\bbar \times W)$, where $\Gamma_g$ is the graph of $g$. Therefore $\cV$ is open in the graph topology (which is the Whitney $C^0$-topology) on $C^\infty(L,N)$, and therefore it is also open in the Whitney $C^\infty$-topology on $C^\infty(L,N)$. Applying Theorem \ref{t:gg} to the open neighbourhood $\cU \cap \cV$ of $f$ we obtain $g \in \cU$ such that $g|_A = f|_A$, $g(L \smallsetminus B) \subseteq N \smallsetminus W$ and $g(\bbar) \subseteq N \smallsetminus W$. The last two properties combined imply that $g(L) \subseteq N \smallsetminus W$.
\end{proof}

We will use this to prove:

\begin{prop}\label{p:condition-iv}
If $\mathrm{dim}(P) = p \leq \tfrac12(m-3) = \tfrac12(\mathrm{dim}(M) - 3)$, the augmented semi-simplicial space $Z_\bullet = \xbar_\bullet$ satisfies condition \textup{(iv)} of Theorem \ref{t:grw}.
\end{prop}

\begin{coro}\label{c:condition-iv}
If $p \leq \tfrac12(m-3)$, the maps $\lVert \xbar_\bullet \rVert \to \xbar$ and $\lVert \ybar_\bullet \rVert \to \ybar$ are weak equivalences.
\end{coro}
\begin{proof}
By the discussion above, $Z_\bullet = \xbar_\bullet$ satisfies conditions (i)--(iii) and by Proposition \ref{p:condition-iv} it also satisfies condition (iv). Theorem \ref{t:grw} therefore implies that $\lVert \xbar_\bullet \rVert \to \xbar$ is a weak equivalence. The argument for $\lVert \ybar_\bullet \rVert \to \ybar$ is almost identical, replacing $n$ by $n+1$ and $M_1$ by $M_0$ everywhere.
\end{proof}

\begin{proof}[Proof of Proposition \ref{p:condition-iv}]
Fix a point $\varphi = \{ [\varphi_1] ,\ldots, [\varphi_{n-1}] \} \in \xbar$ and a collection of embeddings $\tau_1,\ldots,\tau_k \colon P \times [0,2] \hookrightarrow M$ such that $(\varphi,\tau_\alpha) \in \xbar_0$ for each $\alpha$. Let $h_\alpha \in (\tfrac12,1)$ be the number associated to $\tau_\alpha$ by condition (a) of Definition \ref{d:secondr}. We need to construct a new embedding
\begin{equation}\label{e:tau}
\tau \colon P \times [0,2] \lhto M
\end{equation}
such that $(\varphi,\tau) \in \xbar_0$ and $(\varphi,(\tau_\alpha,\tau)) \in \xbar_1$ for all $\alpha \in \{ 1,\ldots,k \}$. As a first step, choose $h \in (\tfrac12,1)$ such that $h > \mathrm{max}_{\alpha = 1}^k h_\alpha$ and define an embedding $\sigma \colon P \times [0,2] \hookrightarrow M$ by $\sigma = \iotabb(-,h,-)$.

It now suffices to find an embedding
\[
\sigma' \colon P \times (1,2) \lhto M_1 \smallsetminus \lambda(T \times \{2\})
\]
such that
\begin{itemizeb}
\item $\sigma = \sigma'$ on $P \times ((1,1+\epsilon) \cup (2-\epsilon,2))$ for some $\epsilon > 0$,
\item the image of $\sigma'$ is disjoint from $\varphi_1(P) \cup \cdots \cup \varphi_{n-1}(P)$ and $\tau_1(P \times (1,2)) \cup \cdots \cup \tau_k(P \times (1,2))$.
\end{itemizeb}
This is because we could then define \eqref{e:tau} to agree with $\sigma$ on $P \times ([0,1] \cup \{2\})$ and to agree with $\sigma'$ on $P \times (1,2)$, and it would satisfy all of the required conditions.

Let $\sigma_0 = \sigma|_{P \times (1,2)}$. We will construct $\sigma'$ from $\sigma_0$ by using Corollary \ref{c:transversality} to modify it to be disjoint from the manifolds $\tau_\alpha(P \times (1,2))$ and $\varphi_\alpha(P)$ one at a time.

Set $L = P \times (1,2)$, $N = M_1 \smallsetminus \lambda(T \times \{2\})$, $f=\sigma_0$ and $W = \tau_1(P \times (1,2))$. By a compactness argument, and using the fact that $h \neq h_1$ and property (a) of $\tau_1$, we may find $\delta > 0$ such that
\[
\sigma_0(P \times ((1,1+2\delta] \cup [2-2\delta,2))) \subseteq N \smallsetminus W.
\]
We may therefore set $B = P \times ((1,1+2\delta) \cup (2-2\delta,2))$. Since being an embedding is an open condition in the Whitney $C^\infty$-topology, we may take $\cU$ to be an open neighbourhood of $f=\sigma_0$ in $C^\infty(L,N)$ consisting of embeddings. Corollary \ref{c:transversality} therefore gives us an embedding
\[
\sigma_1 \colon P \times (1,2) \lhto M_1 \smallsetminus (\lambda(T \times \{2\}) \cup \tau_1(P \times (1,2)))
\]
such that $\sigma_1 = \sigma_0$ on $P \times ((1,1+\delta) \cup (2-\delta,2))$. Note that here we crucially used the fact that $\dim(P) \leq \frac12(\dim(M)-3)$ in order to satisfy the dimension condition of Corollary \ref{c:transversality}.

Iterating this, we next set $L = P \times (1,2)$, $N = M_1 \smallsetminus (\lambda(T \times \{2\}) \cup \tau_1(P \times (1,2)))$, $f = \sigma_1$ and $W = \tau_2(P \times (1,2))$ and apply Corollary \ref{c:transversality} to obtain an embedding
\[
\sigma_2 \colon P \times (1,2) \lhto M_1 \smallsetminus (\lambda(T \times \{2\}) \cup \tau_1(P \times (1,2)) \cup \tau_2(P \times (1,2)))
\]
such that $\sigma_2 = \sigma_1$ on $P \times ((1,1+\delta') \cup (2-\delta',2))$ for some $\delta' > 0$. After a finite number of further applications of Corollary \ref{c:transversality} we obtain an embedding $\sigma'$ with the required properties.
\end{proof}

\subsection{The second approximation.}\label{s:seconda}

\begin{defn}
For $j\geq 0$ define $\abar_j$ to be the subspace of $\Emb(P \times [0,2],M)^{i+1}$ consisting of tuples of embeddings $(\tau_0,\ldots,\tau_j)$ such that each $\tau_\alpha$ satisfies condition (a) of Definition \ref{d:secondr}, the tuple satisfies conditions (c) and (d) of Definition \ref{d:secondr} and each $\tau_\alpha$ also satisfies:
\begin{itemizeb}
\item[({\d b})] $\tau_\alpha(P \times (1,2)) \subseteq M_1 \smallsetminus \lambda(T \times \{2\})$.
\end{itemizeb}
Similarly, define $\bbar_j$ to be the subspace of $\Emb(P \times [0,2],M)^{i+1}$ consisting of $(\tau_0,\ldots,\tau_j)$ satisfying conditions ({\=a}), (c) and (d) of Definition \ref{d:secondr}, as well as
\begin{itemizeb}
\item[({\d{\={b}}})] $\tau_\alpha(P \times (0,2)) \subseteq M_0 \smallsetminus \lambda(T \times \{2\})$.
\end{itemizeb}
\end{defn}

\begin{lem}
The inclusion $\abar_j \hookrightarrow \bbar_j$ is a homotopy equivalence.
\end{lem}
\begin{proof}
We will define a deformation retraction for the inclusion, i.e.\ a map
\[
H \colon \bbar_j \times [0,1] \longrightarrow \bbar_j
\]
such that $H(-,0) = \mathrm{id}$ and $H((\abar_j \times [0,1]) \cup (\bbar_j \times \{1\})) \subseteq \abar_j$. For this, we choose a smooth map
\[
\Upsilon \colon S \longrightarrow [0,2],
\]
where $S = \{ (s,t) \in [0,1] \times [0,2] \mid s \leq t \}$, such that
\begin{itemizeb}
\item[(i)] $\Upsilon(0,t) = t$,
\item[(ii)] $\Upsilon(s,t) = t - s$ for $t \leq s+\tfrac14$,
\item[(iii)] $\Upsilon(s,t) = t$ for $t \geq \tfrac74$,
\item[(iv)] for each fixed $s \in [0,1]$ the map $\Upsilon(s,-)$ is a diffeomorphism $[s,2] \cong [0,2]$ with $\Upsilon(s,1) \leq 1$.
\end{itemizeb}
This is not hard (although a little fiddly) to construct. Note that the function $\Upsilon(s,t) = 2(\tfrac{t-s}{2-s})$ satisfies conditions (i) and (iv), but not (ii) or (iii). If we were working only with $C^0$ embeddings, then we would not need (ii) or (iii) and this version of $\Upsilon$ would work, but in order to ensure that the two parts of the definition below glue together to form a $C^\infty$ embedding we also need conditions (ii) and (iii).

We may now use this to ``conjugate'' each $\tau_\alpha$ by reparametrising its domain by $\Upsilon(s,-)$ and its codomain by $\Upsilon(s,-)^{-1}$. More precisely, we define the deformation retraction by
\[
H((\tau_0,\ldots,\tau_j),s) = (\tau_0^s,\ldots,\tau_j^s),
\]
where $\tau_\alpha^s \colon P \times [0,2] \hookrightarrow M$ is defined by
\[
\tau_\alpha^s(z,t) = \begin{cases}
\iotabb(z,h_\alpha,t) & \text{for } 0 \leq t \leq s \\
\bar{\Upsilon}_s \circ \tau_\alpha(z,\Upsilon(s,t)) & \text{for } s \leq t \leq 2,
\end{cases}
\]
where $\bar{\Upsilon}_s \colon M \hookrightarrow M$ is the self-embedding defined by $\lambda \circ (\mathrm{id} \times \Upsilon(s,-)^{-1}) \circ \lambda^{-1}$ on the collar neighbourhood and by the identity elsewhere.

The choice of smoothly varying reparametrisation $\Upsilon(s,-)$ ensures that the two pieces of this definition glue together to form a smooth embedding, and one may easily check that this $H$ is indeed a deformation retraction.
\end{proof}

Now define a map $\bar{p}_j \colon \xbar_j \to \abar_j$ taking $(\{ [\varphi_1] ,\ldots, [\varphi_{n-1}] \},(\tau_0,\ldots,\tau_j))$ to the tuple $(\tau_0,\ldots,\tau_j)$, and similarly $\bar{q}_j \colon \ybar_j \to \bbar_j$ taking $(\{ [\varphi_1] ,\ldots, [\varphi_{n}] \},(\tau_0,\ldots,\tau_j))$ to $(\tau_0,\ldots,\tau_j)$. These forgetful maps fit into a commutative square:
\begin{equation}
\label{eq:second-approximation}
\centering
\begin{split}
\begin{tikzpicture}
[x=1mm,y=1mm]
\node (tl) at (0,12) {$\xbar_j$};
\node (tr) at (20,12) {$\ybar_j$};
\node (bl) at (0,0) {$\abar_j$};
\node (br) at (20,0) {$\bbar_j$};
\draw[->] (tl) to node[above,font=\small]{$g_j$} (tr);
\draw[->] (tl) to node[left,font=\small]{$\bar{p}_j$} (bl);
\draw[->] (tr) to node[right,font=\small]{$\bar{q}_j$} (br);
\incl{(bl)}{(br)}
\end{tikzpicture}
\end{split}
\end{equation}

\begin{prop}\label{p:Serre-fibrations-2}
The maps $\bar{p}_j \colon \xbar_j \to \abar_j$ and $\bar{q}_j \colon \ybar_j \to \bbar_j$ are Serre fibrations.
\end{prop}
\begin{proof}
We start by constructing a manifold with boundary $M_{\mathrm{cut}}$ whose interior is $M_1 \smallsetminus \lambda(T \times \{2\})$. Recall that $T \subseteq \partial M$ is a codimension-zero submanifold with boundary, and write $\mathring{T}$ for its interior. We will attach two pieces of boundary to $M_1 \smallsetminus \lambda(T \times \{2\})$. First, we (re-)attach $\lambda(\partial M \times \{1\})$ as one piece of boundary.\footnote{Note that we do not say one ``boundary-component'', because we are not assuming that $\partial M$ is connected.} Second, we attach $\lambda(\mathring{T} \times \{2\})$, \emph{but only along one side}, as the second piece of boundary. More precisely, we define
\begin{equation}\label{e:cut-manifold}
M_{\mathrm{cut}} \; = \; M \smallsetminus \lambda((\partial M \times [0,1)) \cup (T \times \{2\})) \; \underset{\lambda(\mathring{T} \times (1,2))}{\cup} \; \lambda(\mathring{T} \times (1,2]).
\end{equation}
The name $M_{\mathrm{cut}}$ comes from thinking of the operation of removing $\lambda(T \times \{2\})$ and re-attaching $\lambda(\mathring{T} \times \{2\})$ along one side only as making a ``cut'' in the interior of $M$ in order to get a new piece of boundary. One may see directly from the definition \eqref{e:cut-manifold} that $\mathrm{int}(M_{\mathrm{cut}}) = M_1 \smallsetminus \lambda(T \times \{2\})$ and $\partial M_{\mathrm{cut}} = \lambda(\partial M \times \{1\}) \sqcup \lambda(\mathring{T} \times \{2\})$ as described more heuristically above.

Now let $L = P \times [1,2] \times \{0,\ldots,j\}$ and decompose its boundary as
\[
\partial_1 L = P \times \{1\} \times \{0,\ldots,j\} \qquad \partial_2 L = P \times \{2\} \times \{0,\ldots,j\} .
\]
Similarly, decompose the boundary of $M_{\mathrm{cut}}$ as
\[
\partial_1 M_{\mathrm{cut}} = \lambda(\partial M \times \{1\}) \qquad \partial_2 M_{\mathrm{cut}} = \lambda(\mathring{T} \times \{2\}).
\]
Define $\hat{A}_j = \mathrm{Emb}_{12}(L,M_{\mathrm{cut}})$, where the embedding space $\mathrm{Emb}_{12}(-,-)$ is defined as in \S\ref{s:fibre-bundles} just before Proposition \ref{p:G-locally-retractile-boundary}. Equivalently an element of $\hat{A}_j$ consists of a tuple $(\tau_0,\ldots,\tau_j)$ of embeddings $\tau_\alpha \colon P \times [1,2] \hookrightarrow M$ with pairwise-disjoint images, such that
\begin{itemizeb}
\item $\tau_\alpha(P \times \{1\}) \subseteq \lambda(\partial M \times \{1\})$,
\item $\tau_\alpha(P \times \{2\}) \subseteq \lambda(\mathring{T} \times \{2\})$,
\item $\tau_\alpha(P \times (1,2)) \subseteq M_1 \smallsetminus \lambda(T \times [2,2+\epsilon))$ for some $\epsilon > 0$.
\end{itemizeb}
Now define $\hat{X}_j$ to be the subspace of $\xbar \times \hat{A}_j$ consisting of all $(\{ [\varphi_1] ,\ldots, [\varphi_{n-1}] \} , (\tau_0,\ldots,\tau_j))$ such that
\[
\bigcup_{\alpha = 1}^{n-1} \varphi_\alpha(P) \qquad\text{is disjoint from}\qquad \bigcup_{\alpha = 0}^{j} \tau_\alpha(P \times [1,2]).
\]
(Recall that the space $\xbar$ was defined in Definition \ref{d:xbar}.) There is a well-defined continuous action of $\mathrm{Diff}_c(M_{\mathrm{cut}})_0$ on both $\hat{X}_j$ and $\hat{A}_j$ given by post-composition, and the projection onto the second factor $\hat{p}_j \colon \hat{X}_j \to \hat{A}_j$ is equivariant with respect to these actions. (Note that it is important for the well-definedness of the actions that we are considering the path-component $\mathrm{Diff}_c(M_{\mathrm{cut}})_0$ of the identity in the group $\mathrm{Diff}_c(M_{\mathrm{cut}})$, so in particular we are considering actions of a path-connected group.) Propositions \ref{p:G-locally-retractile-boundary} and \ref{p:G-locally-retractile} now imply that $\hat{p}_j$ is a fibre bundle, in particular a Serre fibration.

Now we may define a topological embedding $\abar_j \hookrightarrow \hat{A}_j$ by $(\tau_0,\ldots,\tau_j) \mapsto (\tau_0|_{P \times [1,2]},\ldots,\tau_j|_{P \times [1,2]})$. By construction, the pullback of $\hat{p}_j$ along this embedding is exactly the map $\bar{p}_j \colon \xbar_j \to \abar_j$. Hence $\bar{p}_j$ is a Serre fibration, as required. The argument for $\bar{q}_j$ is essentially identical, so it is omitted.
\end{proof}

Now fix $\bar{a} = (\tau_0,\ldots,\tau_j) \in \abar_j$. We will show that the restriction of $g_j$ to the fibres $\bar{p}_j^{-1}(\bar{a}) \to \bar{q}_j^{-1}(\bar{a})$ lies in $\X(n-1)$. Define
\[
\breve{M} = M \smallsetminus (\tau_0(P \times [0,2]) \cup \cdots \cup \tau_j(P \times [0,2]) \cup \lambda(T \times \{2\}))
\]
and note that $\partial \breve{M} = \partial M \smallsetminus \iotabb(P \times \{h_0,\ldots,h_j\} \times \{0\})$. Choose $\epsilon > 0$ such that:
\begin{itemizeb}
\item $\tau_\alpha(z,t) = \iotabb(z,h_\alpha,t)$ for all $t \in [0,1+\epsilon]$, $z \in P$ and $\alpha \in \{ 0,\ldots,j \}$,
\item each $\tau_\alpha(P \times (1+\epsilon,2])$ is contained in $M_{1+\epsilon}$.
\end{itemizeb}
Choose a diffeomorphism $\theta \colon [0,2] \to [0,1+\epsilon]$ such that $\theta(t) = t$ for $t \in [0,1]$ and define
\[
\mu \colon \partial\breve{M} \times [0,2] \lhto \breve{M}
\]
by $\mu(z,t) = \lambda(z,\theta(t))$. This is a collar neighbourhood for the boundary of $\breve{M}$. Write $X_{n-1}(\breve{M},\mu)$ for the space $X$ from Definition \ref{d:input-data} with $P,\iota,G$ the same as before but with $n,M,\lambda$ replaced by $n-1,\breve{M},\mu$ respectively, and similarly for $Y_{n-1}(\breve{M},\mu)$ and the map $f_{n-1}(\breve{M},\mu)$.

\begin{rmk}\label{r:connectivity-of-M-breve}
In a moment we will show that the restriction of $g_j$ to the fibres $\bar{p}_j^{-1}(\bar{a}) \to \bar{q}_j^{-1}(\bar{a})$ may be identified (up to homeomorphism) with $f_{n-1}(\breve{M},\mu)$, and from this we would like to deduce that it lies in $\X(n-1)$. We therefore have to check that $\breve{M}$ and $\mu$ (as well as $P,\iota,G$) are valid input data for Definition \ref{d:input-data}. The only non-trivial issue with this is to see that $\breve{M}$ is connected. However, this follows from the fact that $M$ is connected, together with the dimension hypothesis $m \geq 2p + 3$; in fact, we only need the weaker assumption that $m \geq p + 3$ here. To see this, note that to obtain $\breve{M}$ from $M$ we first removed $\lambda(T \times \{2\})$ and then each $\tau_\alpha(P \times [0,2])$. The latter all have codimension $m - p - 1 \geq 2$ and the former deformation retracts onto its core $\lambda(\iota(P) \times \{2\})$ (recall that $T \subseteq \partial M$ is a tubular neighbourhood for $\iota(P) \subseteq \partial M$), which has codimension $m - p \geq 3$, so neither of these operations changes $\pi_0$. Therefore $\breve{M}$ is connected since $M$ is connected.
\end{rmk}

A little thought about the definitions yields the following descriptions of $\bar{p}_j^{-1}(\bar{a})$ and $X_{n-1}(\breve{M},\mu)$. First we fix some temporary notation. Write
\[
U \; = \; \mathrm{Emb}(P \times \{1,\ldots,n-1\} , M_1 \smallsetminus \lambda(T \times \{2\})) \; / \; (G \wr \Sigma_{n-1})
\]
and set $W = \tau_0(P \times (1,2)) \cup \cdots \cup \tau_j(P \times (1,2))$, which is a properly embedded submanifold of $M_1 \smallsetminus \lambda(T \times \{2\})$. Choose a tuple $(s_1,\ldots,s_{n-1})$ of distinct points in the interval $(1,1+\epsilon)$ and let $\varphi_{\mathrm{st}}$ be the embedding
\[
(z,\alpha) \; \longmapsto \; \iotabb(z,0,s_\alpha) \colon P \times \{1,\ldots,n-1\} \lhto M_1 \smallsetminus \lambda(T \times \{2\}).
\]
Note that the image of $\varphi_{\mathrm{st}}$ is disjoint from $W$ due to how we chose $\epsilon$ above. With this notation, we have
\[
X_{n-1}(\breve{M},\mu) \; \subseteq \; \bar{p}_j^{-1}(\bar{a}) \; \subseteq \; U,
\]
and, given an element $[\varphi] \in U$,
\begin{itemizeb}
\item $[\varphi] \in \bar{p}_j^{-1}(\bar{a})$ if and only if $\mathrm{im}(\varphi)$ is disjoint from $W$ and there is a path in $U$ from $[\varphi]$ to $[\varphi_{\mathrm{st}}]$,
\item $[\varphi] \in X_{n-1}(\breve{M},\mu)$ if and only if $\mathrm{im}(\varphi)$ is disjoint from $W$ and there is a path $t \mapsto [\varphi^t]$ in $U$ from $[\varphi]$ to $[\varphi_{\mathrm{st}}]$ such that $\mathrm{im}(\varphi^t)$ is disjoint from $W$ for all $t \in [0,1]$.
\end{itemizeb}
Replacing $M_1$ with $M_0$ and $n-1$ with $n$, we obtain a similar description of $Y_{n-1}(\breve{M},\mu) \subseteq \bar{q}_j^{-1}(\bar{a})$, and a commutative square
\begin{equation}
\label{eq:restriction-of-gj}
\begin{split}
\begin{tikzpicture}
[x=1mm,y=1mm]
\node (tl) at (0,10) {$\bar{p}_j^{-1}(\bar{a})$};
\node (tr) at (50,10) {$\bar{q}_j^{-1}(\bar{a})$};
\node (bl) at (0,0) {$X_{n-1}(\breve{M},\mu)$};
\node (br) at (50,0) {$Y_{n-1}(\breve{M},\mu)$};
\node at (0,5) {\rotatebox{90}{$\subseteq$}};
\node at (50,5) {\rotatebox{90}{$\subseteq$}};
\draw[->] (tl) to node[above,font=\small]{restriction of $g_j$} (tr);
\draw[->] (bl) to node[below,font=\small]{$f_{n-1}(\breve{M},\mu)$} (br);
\end{tikzpicture}
\end{split}
\end{equation}
(for commutativity we use the fact that the collar neighbourhoods $\lambda(z,t)$ and $\mu(z,t)$ of $M$ and $\breve{M}$ agree for $t \leq 1$ due to how we chose $\theta$ above; in particular they agree on $\iota(P) \times \{\tfrac12\}$).

It remains to show that in fact $X_{n-1}(\breve{M},\mu) = \bar{p}_j^{-1}(\bar{a})$ and $Y_{n-1}(\breve{M},\mu) = \bar{q}_j^{-1}(\bar{a})$. This will follow immediately from the descriptions above together with the following general lemma, in which we take $L = P \times \{1,\ldots,n-1\}$ and $N = M_1 \smallsetminus \lambda(T \times \{2\})$.

\begin{lem}\label{l:path-of-embeddings}
Let $L$ and $N$ be manifolds without boundary such that $\mathrm{dim}(N) \geq 2\,\mathrm{dim}(L) + 3$ and assume that $L$ is compact. Let $W \subseteq N$ be a closed subset that is a finite union of properly embedded submanifolds, each of dimension at most $\mathrm{dim}(N) - \mathrm{dim}(L) - 2$. Let $G$ be an open subgroup of $\mathrm{Diff}(L)$. Suppose we have a path
\[
\gamma \colon [0,1] \longrightarrow \mathrm{Emb}(L,N)/G
\]
such that $\gamma(0)(L)$ and $\gamma(1)(L)$ are disjoint from $W$. Then there exists another path
\[
\gamma' \colon [0,1] \longrightarrow \mathrm{Emb}(L,N)/G
\]
with the same endpoints as $\gamma$, such that $\gamma'(t)(L)$ is disjoint from $W$ for all $t \in [0,1]$, in other words $\gamma'$ has image contained in $\mathrm{Emb}(L,N \smallsetminus W)/G$.
\end{lem}

\begin{coro}
\label{c:Xn1}
Assume the dimension hypothesis that $\mathrm{dim}(P) = p \leq \tfrac12(m-3) = \tfrac12(\mathrm{dim}(M) - 3)$. Then, in diagram \eqref{eq:second-approximation}, the restriction of $g_j$ to the fibres over any point $\bar{a} \in \abar_j$ lies in $\X(n-1)$.
\end{coro}
\begin{proof}
By the discussion preceding Lemma \ref{l:path-of-embeddings}, we have the commutative square \eqref{eq:restriction-of-gj}. Lemma \ref{l:path-of-embeddings} then tells us that the vertical inclusions in this diagram are equalities, so the restriction of $g_j$ to the fibres over $\bar{a}$ may be identified, up to homeomorphism, with the map $f_{n-1}(\breve{M},\mu)$. Here we have used the dimension hypothesis to ensure that Lemma \ref{l:path-of-embeddings} applies to $L = P \times \{1,\ldots,n-1\}$, $N = M_1 \smallsetminus \lambda(T \times \{2\})$ and $W = \tau_0(P \times (1,2)) \cup \cdots \cup \tau_j(P \times (1,2))$. This map lies in $\X(n-1)$, by definition (see the first paragraph of \S\ref{s:proof}), using Remark \ref{r:connectivity-of-M-breve} to verify that $\breve{M}$ is connected.
\end{proof}

We will use the following theorem of Whitney in the proof of Lemma \ref{l:path-of-embeddings}.

\begin{thm}[{\cite[Theorem 5 in \S 11]{Whitney1936Differentiablemanifolds}}]\label{t:Whitney}
Let $L$ and $N$ be smooth manifolds without boundary such that $\mathrm{dim}(N) \geq 2\,\mathrm{dim}(L) + 1$ and let $A \subseteq L$ be a closed subset. Suppose we are given a continuous map $\phi \colon L \to N$ such that its restriction to $A$ is a smooth injective immersion. Then there exists a smooth injective immersion $\psi \colon L \to N$ such that $\psi|_A = \phi|_A$.
\end{thm}

We will also use the following immediate corollary of Thom's transversality theorem.

\begin{prop}\label{p:Thom-paths}
Let $L$ and $N$ be smooth manifolds without boundary, $L$ assumed to be compact, and $X \subseteq N$ a countable union of submanifolds of $N$. Then given any embedding $\phi \in \mathrm{Emb}(L,N)$ there is a path $\gamma \colon [0,1] \to \mathrm{Emb}(L,N)$ with $\gamma(0) = \phi$ and $\gamma(1)$ transverse to $X$.
\end{prop}
\begin{proof}
By the transversality theorem of Thom (\cf \cite[Theorem II.4.9, page 54]{GolubitskyGuillemin1973Stablemappingsand}), the subset
\[
\{ \phi' \in \mathrm{Emb}(L,N) \mid \phi' \text{ is transverse to } X \} \; \subseteq \; \mathrm{Emb}(L,N)
\]
is dense in the strong $C^\infty$ topology. Since $\mathrm{Emb}(L,N)$ is locally path-connected (indeed, locally contractible since $L$ is compact, \cf Fact \ref{fact:locally-contractible}) we may choose a path-connected open neighbourhood $\cU \subseteq \mathrm{Emb}(L,N)$ of $\phi$. By density, we may find $\phi' \in \cU$ that is transverse to $X$, and then by path-connectedness of $\cU$ we may find a path from $\phi$ to $\phi'$.
\end{proof}

\begin{proof}[Proof of Lemma \ref{l:path-of-embeddings}]
We will prove this in $4 = \{ 0,1,2,3 \}$ steps.

\textbf{Step 0.}
First, we show that we may assume without loss of generality that $\gamma(0)(L)$ and $\gamma(1)(L)$ are disjoint subsets of $N$. So we assume temporarily that the lemma is true under this assumption, and let $\gamma \colon [0,1] \to \mathrm{Emb}(L,N)/G$ be a path such that $\gamma(0)(L)$ and $\gamma(1)(L)$ are disjoint from $W$ (but not necessarily from each other).

The projection map $\mathrm{Emb}(L,N) \to \mathrm{Emb}(L,N)/G$ is a surjective Serre fibration (by Corollary \ref{c:principal-H-bundle}) so we may lift $\gamma$ to $\bar{\gamma} \colon [0,1] \to \mathrm{Emb}(L,N)$. Now $\bar{\gamma}(0)$ is an embedding $L \hookrightarrow N \smallsetminus W$, and by Proposition \ref{p:Thom-paths} (replacing $N$ by $N \smallsetminus W$ and setting $X = \bar{\gamma}(1)(L)$), there is a path $\delta \colon [-1,0] \to \mathrm{Emb}(L,N \smallsetminus W)$ such that $\delta(0) = \bar{\gamma}(0)$ and $\delta(-1)$ is transverse to $\bar{\gamma}(1)(L)$, which implies (since $2\,\mathrm{dim}(L) < \mathrm{dim}(N)$) that $\delta(-1)(L)$ is disjoint from $\bar{\gamma}(1)(L)$.

Write $[\delta]$ for the composition of $\delta$ and the projection to $\mathrm{Emb}(L,N)/G$. Concatenating $[\delta]$ and $\gamma$ we obtain a path in $\mathrm{Emb}(L,N)/G$ whose endpoints $[\delta(-1)]$ and $\gamma(1)$ are (orbits of) embeddings whose images are disjoint from $W$ and from each other. Thus, by our temporary assumption, there exists a path in $\mathrm{Emb}(L,N \smallsetminus W)/G$ between $[\delta(-1)]$ and $\gamma(1)$. Concatenating this path with the reverse of $[\delta]$ we obtain a path in $\mathrm{Emb}(L,N \smallsetminus W)/G$ between $[\delta(0)] = \gamma(0)$ and $\gamma(1)$, as desired.

We may from now on assume that $\gamma(0)(L)$ and $\gamma(1)(L)$ are disjoint.

\textbf{Step 1.} (Extending $\gamma$ via tubular neighbourhoods.)

As in step 0, since the projection $\mathrm{Emb}(L,N) \to \mathrm{Emb}(L,N)/G$ is a surjective Serre fibration (by Corollary \ref{c:principal-H-bundle}), we may lift $\gamma$ to $\bar{\gamma} \colon [0,1] \to \mathrm{Emb}(L,N)$ such that $\bar{\gamma}(0)(L)$ and $\bar{\gamma}(1)(L)$ are disjoint from $W$ and each other. Denote the normal bundle of $\bar{\gamma}(0) \colon L \hookrightarrow N$ by $V_0 = \nu(\bar{\gamma}(0)) \to L$ and its zero section by $z_0 \colon L \to V_0$. Choose a tubular neighbourhood, i.e.\ an embedding $t_0 \colon V_0 \hookrightarrow N$ such that $t_0 \circ z_0 = \bar{\gamma}(0)$.

\begin{sublem}
The normal bundle $V_0 \to L$ has a trivial one-dimensional subbundle.
\end{sublem}
\begin{proof}[Proof of the sublemma]
\let\qedsymboloriginal\qedsymbol
\renewcommand{\qedsymbol}{{\small (sublemma)} \qedsymboloriginal}
It suffices to show that the unit sphere subbundle $S(V_0) \to L$ has a section. The fibres of this bundle are $k$-spheres, where $k = \mathrm{dim}(N) - \mathrm{dim}(L) - 1$, and the obstruction classes to the existence of such a section live in the cohomology groups $H^i(L;\pi_{i-1}(S^k))$. If $i > \mathrm{dim}(L)$ this clearly vanishes. If $i \leq \mathrm{dim}(L)$, then
\[
i-1 \leq \mathrm{dim}(L) - 1 \leq \mathrm{dim}(N) - \mathrm{dim}(L) - 4 = k-3,
\]
so it also vanishes in this case.
\end{proof}

The inclusion of a trivial one-dimensional subbundle is an embedding $L \times \bR \hookrightarrow V_0$. Let us denote by $\hat{\gamma}_0 \colon L \times \bR \hookrightarrow N$ the composition of this embedding and the embedding $t_0$. Note that $\hat{\gamma}_0(-,0) = \bar{\gamma}(0)$. We may now choose disjoint open subsets $U_0$ and $U_1$ of $N$ such that
\begin{align*}
U_0 \; &\supseteq \; \hat{\gamma}_0(L \times \{0\}) \; = \; \bar{\gamma}(0)(L) \\
U_1 \; &\supseteq \; W \cup \bar{\gamma}(1)(L)
\end{align*}
since these are disjoint subsets of $N$ that are compact and closed respectively, and $N$ is regular (since it is a manifold). By compactness of $L$ we may find $\epsilon > 0$ such that $\hat{\gamma}_0(L \times [-\epsilon,\epsilon]) \subseteq U_0$, so in particular $\hat{\gamma}_0(L \times [-\epsilon,\epsilon])$ is disjoint from $W$ and $\bar{\gamma}(1)(L)$.

We may similarly use a tubular neighbourhood of $\bar{\gamma}(1)$ to extend it to an embedding $\hat{\gamma}_1 \colon L \times \bR \hookrightarrow N$ with $\hat{\gamma}_1(-,0) = \bar{\gamma}(1)$. As above, choose disjoint open subsets $U_0^\prime$ and $U_1^\prime$ of $N$ such that
\begin{align*}
U_0^\prime \; &\supseteq \; \hat{\gamma}_1(L \times \{0\}) \; = \; \bar{\gamma}(1)(L) \\
U_1^\prime \; &\supseteq \; W \cup \hat{\gamma}_0(L \times [-\epsilon,\epsilon]).
\end{align*}
By compactness of $L$ we may decrease $\epsilon > 0$ if necessary so that $\hat{\gamma}_1(L \times [-\epsilon,\epsilon]) \subseteq U_0^\prime$. Hence:
\begin{itemizeb}
\item The subsets $\hat{\gamma}_0(L \times [-\epsilon,\epsilon])$ and $\hat{\gamma}_1(L \times [-\epsilon,\epsilon])$ of $N$ are disjoint from $W$ and from each other.
\end{itemizeb}

Now we use $\hat{\gamma}_0$ and $\hat{\gamma}_1$ to extend $\bar{\gamma}$ to a slightly larger interval. Define $\mathring{\gamma} \colon (-3\epsilon , 1+3\epsilon) \to \mathrm{Emb}(L,N)$ by:
\[
\mathring{\gamma}(t)(z) = \begin{cases}
\hat{\gamma}_0(z,t+2\epsilon) & t \in (-3\epsilon,-\epsilon] \\
\hat{\gamma}_0(z,-t) & t \in [-\epsilon,0] \\
\bar{\gamma}(t)(z) & t \in [0,1] \\
\hat{\gamma}_1(z,t-1) & t \in [1,1+\epsilon] \\
\hat{\gamma}_1(z,1-t+2\epsilon) & t \in [1+\epsilon,1+3\epsilon).
\end{cases}
\]
This is a continuous path in $\mathrm{Emb}(L,N)$ extending $\bar{\gamma}$ (i.e.\ $\mathring{\gamma}|_{[0,1]} = \bar{\gamma}$) such that
\begin{itemizeb}
\item $\mathring{\gamma}(-2\epsilon) = \bar{\gamma}(0)$,
\item $\mathring{\gamma}(1+2\epsilon) = \bar{\gamma}(1)$,
\item $\mathring{\gamma}(t)(L)$ is disjoint from $W$ for $t \leq 0$ and for $t \geq 1$.
\end{itemizeb}

\textbf{Step 2.} (Making the adjoint of $\mathring{\gamma}$ into an embedding.)

The adjoint of $\mathring{\gamma}$ is a continuous map
\[
f_1 \colon L \times (-3\epsilon,1+3\epsilon) \longrightarrow N
\]
defined by $f_1(z,t) = \mathring{\gamma}(t)(z)$. Let $A_0 = L \times (-3\epsilon,-\epsilon]$ and $A_1 = L \times [1+\epsilon,1+3\epsilon)$. The restriction of $f_1$ to $A_0$ is (up to an affine shift in coordinates) equal to $\hat{\gamma}_0$, so it is an embedding. Similarly, up to an affine shift in coordinates, the restriction of $f_1$ to $A_1$ is equal to $\hat{\gamma}_1$, so it is also an embedding. Moreover, the closures of the images of $A_0$ and $A_1$ under $f_1$ are disjoint, so the restriction of $f_1$ to $A = A_0 \sqcup A_1$ is an embedding. Applying Theorem \ref{t:Whitney} to $\phi = f_1$ (note that the $L$ of the theorem is what we are calling $L \times (-3\epsilon , 1+3\epsilon)$ here) we obtain a smooth injective immersion
\[
f_2 \colon L \times (-3\epsilon,1+3\epsilon) \longrightarrow N
\]
such that $f_2|_A = f_1|_A$. Let
\[
f_3 \colon L \times (-2.5\epsilon,1+2.5\epsilon) \longrightarrow N
\]
be the restriction of $f_2$ to the relatively compact subset $L \times (-2.5\epsilon,1+2.5\epsilon)$. Note that
\begin{itemizeb}
\item $f_3$ is an embedding, since it is the restriction of an injective immersion to a relatively compact subset of the domain,
\item $f_3(B_1)$ is disjoint from $W$,
\end{itemizeb}
where for $r \in (0,2)$ we write
\[
B_r = L \times ((-2.5\epsilon , -r\epsilon) \cup (1 + r\epsilon , 1 + 2.5\epsilon)).
\]

\textbf{Step 3.} (Transversality and disjointness from $W$.)

We may write $W = W_1 \cup \cdots \cup W_k$, where each $W_i$ is a properly embedded submanifold of $N$ of dimension at most $\mathrm{dim}(N) - \mathrm{dim}(L) - 2$. We will use Corollary \ref{c:transversality} to modify $f_3$ to be disjoint from the $W_i$ one at a time. First, since $f_3(B_1)$ is disjoint from $W_1$, we may apply Corollary \ref{c:transversality} to $f_3$ and $W_1$ to obtain an embedding
\[
f_4 \colon L \times (-2.5\epsilon,1+2.5\epsilon) \lhto N \smallsetminus W_1
\]
such that $f_4$ agrees with $f_3$ on $B_{1.5}$. (We note that, in order to obtain an \emph{embedding}, we chose the $\cU$ in Corollary \ref{c:transversality} to be an open neighbourhood of $f_3$ consisting of embeddings.) Now, since $f_4(B_{1.5}) = f_3(B_{1.5})$ is disjoint from $W_2$, we may apply Corollary \ref{c:transversality} to $f_4$ and $W_2$ to obtain an embedding
\[
f_5 \colon L \times (-2.5\epsilon,1+2.5\epsilon) \lhto N \smallsetminus (W_1 \cup W_2)
\]
such that $f_5$ agrees with $f_4$ (and therefore also with $f_3$) on $B_{1.55}$. Iterating this procedure another $k-2$ times we eventually obtain an embedding
\[
f_{k+3} \colon L \times (-2.5\epsilon,1+2.5\epsilon) \lhto N \smallsetminus (W_1 \cup \cdots \cup W_k) = N \smallsetminus W
\]
such that $f_{k+3}$ agrees with $f_3$ on $B_{1.55\ldots 5}$. In particular, it agrees with $f_3$, and therefore with $f_1$, on $L \times \{-2\epsilon\}$ and $L \times \{1+2\epsilon\}$. We may now define
\[
\gamma' \colon [-2\epsilon , 1+2\epsilon] \longrightarrow \mathrm{Emb}(L,N)
\]
by $\gamma'(t)(z) = f_{k+3}(z,t)$ and note that
\begin{itemizeb}
\item $\gamma'(t)(L)$ is disjoint from $W$ for all $t \in [-2\epsilon , 1+2\epsilon]$,
\item $\gamma'(-2\epsilon) = f_{k+3}(\,\cdot\, ,-2\epsilon) = f_1(\,\cdot\, ,-2\epsilon) = \mathring{\gamma}(-2\epsilon) = \bar{\gamma}(0)$,
\item $\gamma'(1+2\epsilon) = f_{k+3}(\,\cdot\, ,1+2\epsilon) = f_1(\,\cdot\, ,1+2\epsilon) = \mathring{\gamma}(1+2\epsilon) = \bar{\gamma}(1)$.
\end{itemizeb}
Rescaling the interval $[-2\epsilon , 1+2\epsilon]$ to $[0,1]$ and post-composing with the projection $\mathrm{Emb}(L,N) \to \mathrm{Emb}(L,N)/G$, we obtain a path $\gamma'$ as claimed in the lemma.
\end{proof}

\begin{rmk}[\emph{The use of the dimension hypothesis in the proof of Lemma \ref{l:path-of-embeddings}.}]
\label{r:dimension-hypothesis-lemma}
In Lemma \ref{l:path-of-embeddings} we made two assumptions about the relative dimensions of the manifolds $L$ and $N$, and the union of submanifolds $W \subseteq N$, namely:
\[
\mathrm{dim}(N) \geq 2\,\mathrm{dim}(L) + 3 \qquad\text{and}\qquad \mathrm{dim}(W) \leq \mathrm{dim}(N) - \mathrm{dim}(L) - 2.
\]
The first assumption was essential in step 2 of the proof, in order to apply Theorem \ref{t:Whitney} to $N$ and $L \times (\text{interval})$. The second assumption was essential in step 3 of the proof, in order to apply Corollary \ref{c:transversality} to $N$, $W$ and $L \times (\text{interval})$. In step 0, only the weaker dimension hypothesis $\mathrm{dim}(N) \geq 2\,\mathrm{dim}(L) + 1$ was needed. In step 1, a sufficient hypothesis would have been that the orbits of embeddings $\gamma(0)$ and $\gamma(1) \in \mathrm{Emb}(L,N)/G$ both have the property that some (equivalently, any) representative embedding $L\hookrightarrow N$ admits a non-vanishing normal section. This normal section hypothesis is also guaranteed by the weaker dimension hypothesis $\mathrm{dim}(N) \geq 2\,\mathrm{dim}(L) + 1$.
\end{rmk}

\subsection{The weak factorisation condition.}\label{s:weak-factorisation}

We first quickly recap some choices and constructions that we have made so far in the proof.

\begin{recap}
In \S\ref{s:firsta} we chose an embedding $\psi_0 \colon P \hookrightarrow M_2$ such that $a_0 = [\psi_0] \in A_0 \subseteq B_0$, and we constructed a diffeomorphism
\[
\Psi \colon M \smallsetminus \psi_0(P) \longrightarrow M \smallsetminus \lambda(T \times \{2\}).
\]
In \S\ref{s:seconda} we chose an embedding $\tau_0 \colon P \times [0,2] \hookrightarrow M$ such that $\bar{a} = \tau_0 \in \abar_0 \subseteq \bbar_0$. With respect to these choices, there is a composite map
\begin{equation}\label{e:composite-map}
\eta \colon Y_{n-1}(\breve{M},\mu) = \bar{q}_0^{-1}(\bar{a}) \lhto \ybar_0 \longrightarrow \ybar \xrightarrow{\;\;\Psi^{-1} \,\circ\, -\;\;} q_0^{-1}(a_0) \lhto Y_0 \longrightarrow Y,
\end{equation}
where $Y_0 \to Y$ and $\ybar_0 \to \ybar$ are the augmentation maps of the semi-simplicial spaces $Y_\bullet$ and $\ybar_\bullet$ respectively, the two inclusions are the inclusions of fibres for the Serre fibrations $q_0 \colon Y_0 \to B_0$ and $\bar{q}_0 \colon \ybar_0 \to \bbar_0$ (\cf Lemma \ref{l:Serre-fibrations-1} and Proposition \ref{p:Serre-fibrations-2}) and the middle map is a homeomorphism induced by post-composing all embeddings with the inverse of $\Psi$ (\cf Lemma \ref{l:Psi}). Finally, $\bar{q}_0^{-1}(\bar{a})$ was identified in \S\ref{s:seconda} with $Y_{n-1}(\breve{M},\mu)$, where $\breve{M} = M \smallsetminus (\tau_0(P \times [0,2]) \cup \lambda(T \times \{2\}))$ and $\mu$ is a collar neighbourhood for the boundary of $\breve{M}$, a restriction and rescaling of $\lambda$.

In a similar way, we have a composite map $X_{n-1}(\breve{M},\mu) \to X$ and a commutative square (ignoring for now the grey diagonal map and homotopy):
\begin{equation}
\label{eq:final-commutative-square}
\begin{split}
\begin{tikzpicture}
[x=1mm,y=1mm]
\node (tl) at (0,15) {$X_{n-1}(\breve{M},\mu)$};
\node (tr) at (50,15) {$Y_{n-1}(\breve{M},\mu)$};
\node (bl) at (0,0) {$X$};
\node (br) at (50,0) {$Y$};
\draw[->] (tl) to node[above,font=\small]{$f_{n-1}(\breve{M},\mu)$} (tr);
\draw[->] (bl) to node[below,font=\small]{$f$} (br);
\draw[->] (tl) to (bl);
\draw[->] (tr) to node[right,font=\small]{$\eta$} (br);
\draw[->,densely dashed,black!50] (tr) to node[above,font=\small]{$\zeta$} (bl);
\node (homotopy) at (40,4) [black!50] {$\Longrightarrow$};
\node at (homotopy.north) [anchor=mid,font=\small,black!50] {$\cH$};
\end{tikzpicture}
\end{split}
\end{equation}
obtained by gluing together the following six commutative squares:
\begin{itemizeb}
\item two coming from the augmentations of the semi-simplicial maps $X_\bullet \to Y_\bullet$ and $\xbar_\bullet \to \ybar_\bullet$,
\item two coming from the inclusions of the fibres of \eqref{e:firsta} and \eqref{e:seconda} with $i=j=0$,
\item \eqref{e:psi-inverse},
\item \eqref{eq:restriction-of-gj} with $j=0$.
\end{itemizeb}
\end{recap}

\textbf{Last step of the proof.}
The commutative square \eqref{eq:final-commutative-square} above is the commutative square \eqref{eSquare2} of Theorem \ref{tAxiomatic} for which we must check the weak factorisation condition, i.e., we must check that the map $\eta$ in \eqref{eq:final-commutative-square} factors up to homotopy through $f$, as the last step of the proof of Theorem~\ref{tmain}. We will therefore construct a map $\zeta \colon Y_{n-1}(\breve{M},\mu) \to X$ and a homotopy $\cH$ from $f \circ \zeta$ to $\eta$, as pictured in grey in \eqref{eq:final-commutative-square}. First we describe $f$ and $\eta$ explicitly. By definition, $f \colon X \to Y$ acts by
\[
\{ [\varphi_1] ,\ldots, [\varphi_n] \} \;\longmapsto\; \{ [\iotabb(-,0,0.5)] , [\varphi_1] ,\ldots, [\varphi_n] \} .
\]
Combining the two augmentation maps, the diffeomorphism $\Psi$ and the identification $Y_{n-1}(\breve{M},\mu) = \bar{q}_0^{-1}(\bar{a})$, we see that $\eta \colon Y_{n-1}(\breve{M},\mu) \to Y$ acts by
\begin{equation}\label{e:eta}
\{ [\varphi_1] ,\ldots, [\varphi_n] \} \;\longmapsto\; \{ [\psi_0] , [\Psi^{-1} \circ \varphi_1] ,\ldots, [\Psi^{-1} \circ \varphi_n] \} .
\end{equation}

\textbf{Two paths of embeddings.}
To construct $\zeta$ and $\cH$ we will use two paths of embeddings $\Gamma$ and $\Delta$. First, recall from Remark \ref{r:diffeomorphism-Psi} that there is a path
\[
\Gamma \colon [0,1] \longrightarrow \mathrm{Emb}(M \smallsetminus \lambda(T \times \{2\}),M)
\]
such that $\Gamma(0) = \Psi^{-1}$, $\Gamma(1)$ is the inclusion and the restriction of $\Gamma(s)$ to $M \smallsetminus M_{1.5}$ is the identity for all $s \in [0,1]$. Second, recall that in \S\ref{s:seconda} we chose an $\epsilon > 0$ such that
\begin{itemizeb}
\item $\tau_0(z,t) = \iotabb(z,h_0,t)$ for all $t \in [0,1+\epsilon]$ and $z \in P$,
\item $\tau_0(P \times (1+\epsilon,2])$ is contained in $M_{1+\epsilon}$.
\end{itemizeb}
Now choose a path $\delta \colon [0,1] \to \mathrm{Emb}([0,2],[0,2])$ such that $\delta(0)$ is the identity, $\delta(1)$ takes $[0,2]$ to $[1,2]$ and $\delta(s)|_{[1+\epsilon,2]} = \mathrm{id}$ for all $s \in [0,1]$. This determines a path
\[
\Delta \colon [0,1] \longrightarrow \mathrm{Emb}(M,M)
\]
by defining $\Delta(s)|_{M \smallsetminus \lambda(\partial M \times [0,2])} = \mathrm{id}$ and $\Delta(s)(\lambda(z,t)) = \lambda(z,\delta(s)(t))$ for $(z,t) \in \partial M \times [0,2]$. Note that $\Delta(s)(\breve{M}) \subseteq \breve{M}$ for all $s \in [0,1]$, due to how we chose $\epsilon$ and $\delta$ above, so this may be regarded as a path
\[
\Delta \colon [0,1] \longrightarrow \mathrm{Emb}(\breve{M},\breve{M})
\]
with $\Delta(0) = \mathrm{id}$.

\textbf{The map $\zeta$.}
We now define $\zeta \colon Y_{n-1}(\breve{M},\mu) \to X$ to act by
\[
\{ [\varphi_1] ,\ldots, [\varphi_n] \} \;\longmapsto\; \{ [\Delta(1) \circ \varphi_1] ,\ldots, [\Delta(1) \circ \varphi_n] \} .
\]
This is easily seen to be well-defined and continuous. The map $f \circ \zeta \colon Y_{n-1}(\breve{M},\mu) \to Y$ therefore acts by
\begin{equation}\label{e:fzeta}
\{ [\varphi_1] ,\ldots, [\varphi_n] \} \;\longmapsto\; \{ [\iotabb(-,0,0.5)] , [\Delta(1) \circ \varphi_1] ,\ldots, [\Delta(1) \circ \varphi_n] \} .
\end{equation}

\textbf{The homotopy $\cH$.}
We will now construct a homotopy
\[
\cH \colon Y_{n-1}(\breve{M},\mu) \times [0,4] \longrightarrow Y
\]
between $\cH(-,0) = \eqref{e:fzeta}$ and $\cH(-,4) = \eqref{e:eta}$. We first define it explicitly and then explain what we are doing more intuitively. Explicitly, the definition is
\begin{equation}\label{e:homotopy}
\cH (\{ [\varphi_{1,\ldots,n}] \},t) = \begin{cases}
\{ [\iotabb(-,th_0,0.5)] , [\Delta(1) \circ \varphi_{1,\ldots,n}] \} & t \in [0,1] \\
\{ [\tau_0(-,0.5)] , [\Delta(2-t) \circ \varphi_{1,\ldots,n}] \} & t \in [1,2] \\
\{ [\tau_0(-,0.5)] , [\Gamma(3-t) \circ \varphi_{1,\ldots,n}] \} & t \in [2,3] \\
\{ [\Psi^{-1} \circ \tau_0(- , 1.5 t - 4)] , [\Psi^{-1} \circ \varphi_{1,\ldots,n}] \} & t \in [3,4) \\
\{ [\psi_0] , [\Psi^{-1} \circ \varphi_{1,\ldots,n}] \} & t = 4,
\end{cases}
\end{equation}
where the notation $[\text{\small (wxyz)}_{1,\ldots,n}]$ is an abbreviation of $[\text{\small (wxyz)}_1] ,\ldots, [\text{\small (wxyz)}_n]$. A more explanatory description is as follows:
\begin{itemizeb}
\item[(1)] At time $t=0$ the homotopy $\cH(-,0)$ pushes the configuration $[\varphi_{1,\ldots,n}]$ away from the boundary of $\breve{M}$ using the self-embedding $\Delta(1)$ (this is $\zeta$) and then adjoins a new embedded copy of $P$ via the embedding $[\iotabb(-,0,0.5)]$ (this is $f$).
\item[(2)] In the interval $t \in [0,1]$ the location of the new copy of $P$ is modified (by a straight line) from $[\iotabb(-,0,0.5)]$ to $[\iotabb(-,h_0,0.5)] = [\tau_0(-,0.5)]$. Since the rest of the configuration $[\Delta(1) \circ \varphi_{1,\ldots,n}]$ is contained in $M_1$, this does not result in any collisions.
\item[(3)] In the interval $t \in [1,2]$ the new copy of $P$ stays fixed at $[\tau_0(-,0.5)]$ and the rest of the configuration moves from $[\Delta(1) \circ \varphi_{1,\ldots,n}]$ to $[\Delta(0) \circ \varphi_{1,\ldots,n}] = [\varphi_{1,\ldots,n}]$. This does not result in any collisions with $\tau_0(P \times \{0.5\})$ since the original configuration $[\varphi_{1,\ldots,n}]$ is contained in $\breve{M}$, so in particular it is disjoint from
\begin{equation}\label{e:p-times-1e}
\tau_0(P \times [0,1+\epsilon]) = \iotabb(P \times \{h_0\} \times [0,1+\epsilon]) = \lambda(P_{h_0} \times [0,1+\epsilon]),
\end{equation}
where for $r \in (-1,1)$ we write $P_r = \iotabb(P \times \{r\} \times \{0\}) \subseteq \partial M$, and the self-embeddings $\Delta(s)$ act on the subset $\eqref{e:p-times-1e} \subseteq M$ by ``compressing'' the second coordinate.
\item[(4)] In the interval $t \in [2,3]$ the new copy of $P$ remains fixed at $[\tau_0(-,0.5)]$ and the rest of the configuration moves from $[\varphi_{1,\ldots,n}]$ to $[\Psi^{-1} \circ \varphi_{1,\ldots,n}]$ via the isotopy of embeddings $\Gamma$.
\item[(5)] In the half-open interval $t \in [3,4)$ the configuration $[\Psi^{-1} \circ \varphi_{1,\ldots,n}]$ remains fixed. Note that, since $[\varphi_{1,\ldots,n}]$ is disjoint from $\tau_0(P \times [0,2])$, the configuration $[\Psi^{-1} \circ \varphi_{1,\ldots,n}]$ is disjoint from $\Psi^{-1} \circ \tau_0(P \times [0,2])$. We may therefore move the new copy of $P$ along this embedded copy of $P \times [0,2]$ from $[\tau_0(-,0.5)] = [\Psi^{-1} \circ \tau_0(-,0.5)]$ \emph{almost} to $[\Psi^{-1} \circ \tau_0(-,2)]$ (but not quite, since $\Psi^{-1}$ is not defined exactly at $\tau_0(P \times \{2\}) \subseteq \lambda(T \times \{2\})$). As $t$ approaches $4$ (so $1.5t-4$ approaches $2$), the new copy of $P$ that we have adjoined approaches $\psi_0(P)$. More precisely:
\item[(6)] For $t \in [4-\epsilon',4)$, for some $\epsilon' > 0$, the path of embeddings $t \mapsto \Psi^{-1} \circ \tau_0(-,1.5 t - 4)$ is of the form considered in Remark \ref{r:diffeomorphism-Psi}, and it may therefore be continuously extended to $t=4$ with the embedding $\psi_0$. Hence we may define the homotopy $\cH(-,4)$ to send the configuration $[\varphi_{1,\ldots,n}]$ to the configuration $[\Psi^{-1} \circ \varphi_{1,\ldots,n}]$ with a new copy of $P$ adjoined via the embedding $[\psi_0]$. This is precisely $\eta$.
\end{itemizeb}
For a picture worth $\tfrac{10}{3}$ times as much as these approx.\ $300$ words, see Figure \ref{fhomotopy}. This homotopy
\[
\cH \;\colon\; f \circ \zeta \;\simeq\; \eta
\]
verifies the weak factorisation condition, and therefore completes the proof of Theorem \ref{tmain}.

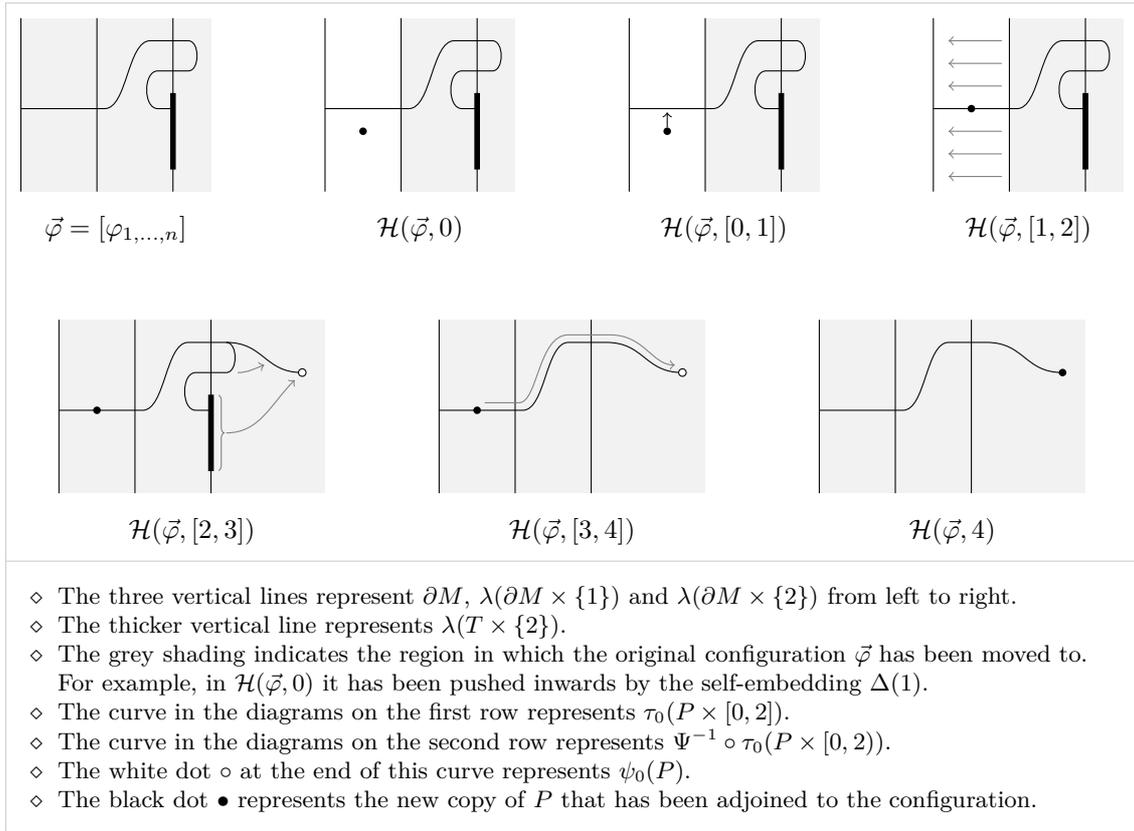
\begin{figure}[ht]
\centering
\begin{tikzpicture}
[x=1mm,y=1mm]

\begin{scope}[xshift=0mm,yshift=0mm]
\fill[black!5] (0,-8) rectangle (25,15);
\draw (0,-8)--(0,15);
\draw (10,-8)--(10,15);
\draw (20,-8)--(20,15);
\draw[fill] (19.7,-5) rectangle (20.3,5);
\draw (0,3) -- (11,3) .. controls (14,3) and (14,12) .. (17,12) -- (22,12) to[out=0,in=0] (22,8) -- (18,8) to[out=180,in=180] (18,3) -- (20,3);
\node at (12.5,-13) {$\vec{\varphi} = [\varphi_{1,\ldots,n}]$};
\end{scope}

\begin{scope}[xshift=40mm,yshift=0mm]
\fill[black!5] (10,-8) rectangle (25,15);
\draw (0,-8)--(0,15);
\draw (10,-8)--(10,15);
\draw (20,-8)--(20,15);
\draw[fill] (19.7,-5) rectangle (20.3,5);
\draw (0,3) -- (11,3) .. controls (14,3) and (14,12) .. (17,12) -- (22,12) to[out=0,in=0] (22,8) -- (18,8) to[out=180,in=180] (18,3) -- (20,3);
\node[circle,fill,inner sep=1pt] at (5,0) {};
\node at (12.5,-13) {$\cH(\vec{\varphi},0)$};
\end{scope}

\begin{scope}[xshift=80mm,yshift=0mm]
\fill[black!5] (10,-8) rectangle (25,15);
\draw (0,-8)--(0,15);
\draw (10,-8)--(10,15);
\draw (20,-8)--(20,15);
\draw[fill] (19.7,-5) rectangle (20.3,5);
\draw (0,3) -- (11,3) .. controls (14,3) and (14,12) .. (17,12) -- (22,12) to[out=0,in=0] (22,8) -- (18,8) to[out=180,in=180] (18,3) -- (20,3);
\draw[->] (5,0.5) -- (5,2.5);
\node[circle,fill,inner sep=1pt] at (5,0) {};
\node at (12.5,-13) {$\cH(\vec{\varphi},[0,1])$};
\end{scope}

\begin{scope}[xshift=120mm,yshift=0mm]
\fill[black!5] (10,-8) rectangle (25,15);
\draw (0,-8)--(0,15);
\draw (10,-8)--(10,15);
\draw (20,-8)--(20,15);
\draw[fill] (19.7,-5) rectangle (20.3,5);
\draw (0,3) -- (11,3) .. controls (14,3) and (14,12) .. (17,12) -- (22,12) to[out=0,in=0] (22,8) -- (18,8) to[out=180,in=180] (18,3) -- (20,3);
\node[circle,fill,inner sep=1pt] at (5,3) {};
\draw[->,black!50] (9,12) -- (2,12);
\draw[->,black!50] (9,9) -- (2,9);
\draw[->,black!50] (9,6) -- (2,6);
\draw[->,black!50] (9,0) -- (2,0);
\draw[->,black!50] (9,-3) -- (2,-3);
\draw[->,black!50] (9,-6) -- (2,-6);
\node at (12.5,-13) {$\cH(\vec{\varphi},[1,2])$};
\end{scope}

\begin{scope}[xshift=5mm,yshift=-40mm]
\fill[black!5] (0,-8) rectangle (35,15);
\draw (0,-8)--(0,15);
\draw (10,-8)--(10,15);
\draw (20,-8)--(20,15);
\draw[fill] (19.7,-5) rectangle (20.3,5);
\draw (0,3) -- (11,3) .. controls (14,3) and (14,12) .. (17,12) -- (22,12) to[out=0,in=0] (22,8) -- (18,8) to[out=180,in=180] (18,3) -- (20,3);
\node[circle,fill,inner sep=1pt] at (5,3) {};
% New arc and dot at the end:
\draw (22,12) .. controls (27,12) and (27,8) .. (32,8);
\node[circle,fill,inner sep=1pt,white] at (32,8) {};
\node[circle,inner sep=1pt,draw] at (32,8) {};
\draw[->,black!50] (23.5,8) to[out=0,in=200] (27,9);
\draw[black!50,decorate,decoration={brace,amplitude=2pt,mirror}] (21,-5) -- (21,5);
\draw[->,black!50] (22,0) to[out=0,in=225] (31,7);
\node at (17.5,-13) {$\cH(\vec{\varphi},[2,3])$};
\end{scope}

\begin{scope}[xshift=55mm,yshift=-40mm]
\fill[black!5] (0,-8) rectangle (35,15);
\draw (0,-8)--(0,15);
\draw (10,-8)--(10,15);
\draw (20,-8)--(20,15);
\node[circle,fill,inner sep=1pt] at (5,3) {};
\draw (0,3) -- (11,3) .. controls (14,3) and (14,12) .. (17,12) -- (22,12) .. controls (27,12) and (27,8) .. (32,8);
\node[circle,fill,inner sep=1pt,white] at (32,8) {};
\node[circle,inner sep=1pt,draw] at (32,8) {};
\node at (17.5,-13) {$\cH(\vec{\varphi},[3,4])$};
\draw[->,black!50] (6,4) -- (10.5,4) .. controls (13,4) and (13,13) .. (17,13) -- (22.5,13) .. controls (27,13) and (28,9) .. (31,9);
\end{scope}

\begin{scope}[xshift=105mm,yshift=-40mm]
\fill[black!5] (0,-8) rectangle (35,15);
\draw (0,-8)--(0,15);
\draw (10,-8)--(10,15);
\draw (20,-8)--(20,15);
\draw (0,3) -- (11,3) .. controls (14,3) and (14,12) .. (17,12) -- (22,12) .. controls (27,12) and (27,8) .. (32,8);
\node[circle,fill,inner sep=1pt] at (32,8) {};
\node at (17.5,-13) {$\cH(\vec{\varphi},4)$};
\end{scope}

% Bounding box:
\draw[black!20] (-2,-93) rectangle (147,17);
\draw[black!20] (-2,-57) -- (147,-57);

\node[text width=14.5cm,anchor=north west,font=\small] at (0,-59) {%
$\diamond\; $ The three vertical lines represent $\partial M$, $\lambda(\partial M \times \{1\})$ and $\lambda(\partial M \times \{2\})$ from left to right. \\%
$\diamond\; $ The thicker vertical line represents $\lambda(T \times \{2\})$. \\%
$\diamond\; $ The grey shading indicates the region in which the original configuration $\vec{\varphi}$ has been moved to. \\%
\textcolor{white}{$\diamond\; $} For example, in $\cH(\vec{\varphi},0)$ it has been pushed inwards by the self-embedding $\Delta(1)$. \\%
$\diamond\; $ The curve in the diagrams on the first row represents $\tau_0(P \times [0,2])$. \\%
$\diamond\; $ The curve in the diagrams on the second row represents $\Psi^{-1} \circ \tau_0(P \times [0,2))$. \\%
$\diamond\; $ The white dot $\circ$ at the end of this curve represents $\psi_0(P)$. \\%
$\diamond\; $ The black dot $\bullet$ represents the new copy of $P$ that has been adjoined to the configuration. %
};

\end{tikzpicture}
\caption{The homotopy $\cH$ verifying the weak factorisation condition.}\label{fhomotopy}
\end{figure}

%%%%%%%%%%%%%%%%%%%%%%%%%%%%%%%%%%%%%%%%%%%%%%%%%%%%%%%%%
%%%%%%%%%%%%%%%%%%%%%%%%%%%%%%%%%%%%%%%%%%%%%%%%%%%%%%%%%

\section{Appendix. Homotopy fibres of augmented semi-simplicial spaces}\label{s:appendix}

In this appendix, we give a proof of Lemma \ref{l:homotopy-fibres}, which is restated as Lemma \ref{l:rwappendix} below.

\subsection{Preliminaries}

We begin by mentioning an elementary lemma on colimits in topological spaces.

\begin{lem}\label{lem:closed-inclusion-colim}
Let $F,G\colon I\to \mathsf{Top}$ be diagrams of topological spaces and $\alpha \colon F\Rightarrow G$ be a natural transformation with the property that $\alpha_i \colon F(i) \to G(i)$ is a closed inclusion for each $i\in I$. Suppose also that for each $i\in I$ the square
\begin{equation}\label{eq:closed-inclusion-colim}
\centering
\begin{split}
\begin{tikzpicture}
[x=1mm,y=1mm]
\node (tl) at (0,12) {$F(i)$};
\node (tr) at (30,12) {$\colim(F)$};
\node (bl) at (0,0) {$G(i)$};
\node (br) at (30,0) {$\colim(G)$};
\draw[->] (tl) to node[above,font=\small]{$\mu_i$} (tr);
\draw[->] (bl) to node[below,font=\small]{$\nu_i$} (br);
\draw[->] (tl) to node[left,font=\small]{$\alpha_i$} (bl);
\draw[->] (tr) to node[right,font=\small]{$\alpha_*$} (br);
\end{tikzpicture}
\end{split}
\end{equation}
is \emph{Cartesian} in the category of sets, meaning that the canonical function from $F(i)$ to the pullback \textup{(}in sets\textup{)} of the rest of the diagram is a bijection. Then the map
\[
\alpha_* \colon \colim(F) \longrightarrow \colim(G)
\]
is also a closed inclusion. The same holds with ``closed inclusion'' replaced by ``open inclusion''.
\end{lem}

\begin{rmk}
It is sufficient to check that each square \eqref{eq:closed-inclusion-colim} is Cartesian in the category of sets for each $i$ in a given \emph{cofinal} subcategory $J$ of $I$, since restricting $F$ and $G$ to $J$ does not change their colimits (see \cite[\href{http://stacks.math.columbia.edu/tag/09WN}{09WN}]{stacks-project}).
\end{rmk}

\begin{proof}
It is immediate that $\alpha_*$ is injective, so it remains to show that it is a closed (resp.\ open) map. We write the proof in the closed case; the open case is identical. Let $A$ be a closed subset of $\colim(F)$. By the definition of the colimit topology, it suffices to show that $\nu_i^{-1}(\alpha_*(A))$ is closed in $G(i)$ for each $i$. But since the square above is Cartesian in the category of sets, we have $\nu_i^{-1}(\alpha_*(A)) = \alpha_i(\mu_i^{-1}(A))$, which is closed since $\alpha_i$ is a closed map.
\end{proof}

\begin{defn}
Let $Z_\bullet$ be a semi-simplicial space, and let $Y_\bullet$ be a semi-simplicial subspace. We call $Y_\bullet$ \emph{full} relative to $Z_\bullet$ if each square
\begin{equation}\label{eq:def-of-full-sssubspace}
\centering
\begin{split}
\begin{tikzpicture}
[x=1mm,y=1mm]
\node (tl) at (0,12) {$\Delta^n \times Y_n$};
\node (tr) at (30,12) {$\Delta^n \times Z_n$};
\node (bl) at (0,0) {$\lVert Y_\bullet \rVert^{(n)}$};
\node (br) at (30,0) {$\lVert Z_\bullet \rVert^{(n)}$};
\draw[->] (tl) to (tr);
\draw[->] (bl) to (br);
\draw[->] (tl) to (bl);
\draw[->] (tr) to (br);
\end{tikzpicture}
\end{split}
\end{equation}
is \emph{Cartesian} in the category of sets, meaning that the canonical function from $\Delta^n \times Y_n$ to the pullback (in sets) of the rest of the diagram is a bijection. Another way of stating this is that a simplex $\sigma \in Z_n$ is contained in $Y_n$ whenever at least one of its vertices is contained in $Y_0$. Another equivalent characterisation is that the subspace $\lVert Y^{\delta}_\bullet \rVert$ of $\lVert Z^{\delta}_\bullet \rVert$ is a union of path-components, where $Y_n^\delta$ denotes $Y_n$ given the discrete topology and similarly for $Z_n^\delta$.
\end{defn}

\begin{eg}\label{eg:augmented-ssspace-fibres}
Let $Z_\bullet$ be an \emph{augmented} semi-simplicial space, choose a subset $A \subseteq Z_{-1}$ and define $Y_n = f_n^{-1}(A)$ where $f_n \colon Z_n \to Z_{-1}$ is the unique composition of face maps. Then $Y_\bullet$ is a full semi-simplicial subspace of $Z_\bullet$, which can be seen as follows. Suppose $\sigma$ is a simplex in $Z_n$ and one of its vertices is in $Y_0$, in other words $d^n(\sigma)\in Y_0$, where $d^n$ is one of the possible compositions of face maps $Z_n \to Z_0$. Then we have $f_n(\sigma) = f_0(d^n(\sigma))\in A$ and so $\sigma \in Y_n$.
\end{eg}

The relevance of this definition is the following lemma.

\begin{lem}\label{lem:closed-inclusion}
Let $Z_\bullet$ be a semi-simplicial space and $Y_\bullet$ be a full semi-simplicial subspace. If each inclusion $Y_n \hookrightarrow Z_n$ is a closed inclusion then the map of geometric realisations
\begin{equation}\label{eq:map-of-geometric-realisations}
\lVert Y_\bullet \rVert \longrightarrow \lVert Z_\bullet \rVert
\end{equation}
is also a closed inclusion. Similarly when ``closed inclusion'' is replaced by ``open inclusion''.
\end{lem}

\begin{proof}
The proof is identical for the ``open'' statement and the ``closed'' statement, so we will just prove the ``closed'' statement. We will first prove by induction that the map of skeleta
\begin{equation}\label{eq:map-of-skeleta}
\lVert Y_\bullet \rVert^{(n)} \longrightarrow \lVert Z_\bullet \rVert^{(n)}
\end{equation}
is a closed inclusion for all $n$. The case $n=0$ is part of the assumption, so we let $n\geq 1$ and assume the result for $n-1$ by inductive hypothesis. The map \eqref{eq:map-of-skeleta} is the map of pushouts induced by the following diagram:
\begin{equation}\label{eq:closed-inclusion-inductive-step}
\centering
\begin{split}
\begin{tikzpicture}
[x=1mm,y=1mm]
\node (t1) at (0,15) {$\Delta^n \times Y_n$};
\node (t2) at (30,15) {$\partial\Delta^n \times Y_n$};
\node (t3) at (60,15) {$\lVert Y_\bullet \rVert^{(n-1)}$};
\node (b1) at (0,0) {$\Delta^n \times Z_n$};
\node (b2) at (30,0) {$\partial\Delta^n \times Z_n$};
\node (b3) at (60,0) {$\lVert Z_\bullet \rVert^{(n-1)}$};
\draw[->] (t2) to (t1);
\draw[->] (b2) to (b1);
\draw[->] (t1) to (b1);
\draw[->] (t2) to (b2);
\draw[->] (t3) to (b3);
\draw[->] (t2) to (t3);
\draw[->] (b2) to (b3);
\end{tikzpicture}
\end{split}
\end{equation}
in which the vertical maps are closed inclusions by assumption and inductive hypothesis. To apply Lemma \ref{lem:closed-inclusion-colim} we need to check that three squares of the form \eqref{eq:closed-inclusion-colim} are Cartesian in the category of sets. The square corresponding to the left-hand side of \eqref{eq:closed-inclusion-inductive-step} is precisely \eqref{eq:def-of-full-sssubspace}, and therefore Cartesian in the category of sets by assumption. The square corresponding to the middle of \eqref{eq:closed-inclusion-inductive-step} therefore also has this property, since the left-hand square of \eqref{eq:closed-inclusion-inductive-step} is Cartesian (in spaces, not just in sets, although this is not relevant here). The square corresponding to the right-hand side of \eqref{eq:closed-inclusion-inductive-step} is
\begin{equation}\label{eq:closed-inclusion-inductive-step-2}
\centering
\begin{split}
\begin{tikzpicture}
[x=1mm,y=1mm]
\node (tl) at (0,12) {$\lVert Y_\bullet \rVert^{(n-1)}$};
\node (tr) at (30,12) {$\lVert Z_\bullet \rVert^{(n-1)}$};
\node (bl) at (0,0) {$\lVert Y_\bullet \rVert^{(n)}$};
\node (br) at (30,0) {$\lVert Z_\bullet \rVert^{(n)}$};
\draw[->] (tl) to (tr);
\draw[->] (bl) to (br);
\draw[->] (tl) to (bl);
\draw[->] (tr) to (br);
\end{tikzpicture}
\end{split}
\end{equation}
which is also Cartesian in the category of sets.\footnote{To see this, consider everything as subsets of the set $\lVert Z_\bullet \rVert$. There is a surjection
\[
k\colon \textstyle{\bigsqcup}_{m} (\Delta^m \times Z_m) \longrightarrow \lVert Z_\bullet \rVert
\]
and we need to show that the intersection of the subsets $k(\Delta^n \times Y_n)$ and $k(\Delta^{n-1} \times Z_{n-1})$ is precisely $k(\Delta^{n-1} \times Y_{n-1})$. Clearly this is contained in the intersection. To show the converse, consider $(t,y)\in \Delta^n \times Y_n$ and $(s,z)\in \Delta^{n-1} \times Z_{n-1}$ such that $k(t,y)=k(s,z)$. The equivalence relation defining $k$ glues each face of an $n$-simplex to exactly one $(n-1)$-simplex, so we must have that $z=d_i(y)$ for some face map $d_i$. Hence $z\in Y_{n-1}$ and so $k(s,z)\in k(\Delta^{n-1} \times Y_{n-1})$.} Lemma \ref{lem:closed-inclusion-colim} then implies that \eqref{eq:map-of-skeleta} is a closed inclusion.

Finally, we show that \eqref{eq:map-of-geometric-realisations} is a closed inclusion. By the same reasoning as above, the square
\begin{equation}\label{eq:closed-inclusion-final-step}
\centering
\begin{split}
\begin{tikzpicture}
[x=1mm,y=1mm]
\node (tl) at (0,12) {$\lVert Y_\bullet \rVert^{(n)}$};
\node (tr) at (30,12) {$\lVert Z_\bullet \rVert^{(n)}$};
\node (bl) at (0,0) {$\lVert Y_\bullet \rVert$};
\node (br) at (30,0) {$\lVert Z_\bullet \rVert$};
\draw[->] (tl) to (tr);
\draw[->] (bl) to (br);
\draw[->] (tl) to (bl);
\draw[->] (tr) to (br);
\end{tikzpicture}
\end{split}
\end{equation}
is Cartesian in the category of sets for all $n$. We already know that \eqref{eq:map-of-skeleta} is a closed inclusion for each $n$, so Lemma \ref{eq:closed-inclusion-colim} implies that \eqref{eq:map-of-geometric-realisations} is a closed inclusion.
\end{proof}

\subsection{Computing homotopy fibres levelwise}\label{ss:homotopy-fibres}

\begin{convention}
For a map $g\colon Y\to Z$ and point $z\in Z$, we take the following explicit model for the homotopy fibre $\hofib_z(g)$: denote the mapping space $\Map([0,1],\{0\};Z,\{z\})$ by $P_z Z$ and let $\hofib_z(g)$ be the pullback of $g$ and the evaluation map $\mathrm{ev}_1\colon P_z Z \to Z$.
\end{convention}

Let $X_\bullet$ be an augmented semi-simplicial space and write $f_n \colon X_n \to X \coloneqq X_{-1}$ for the unique composition of face maps. Note that, for any point $x\in X$, the homotopy fibres $\{\hofib_x(f_n)\}_{n\geq 0}$ naturally inherit face maps from $X_\bullet$ and so form a semi-simplicial space $\hofib_x(f_\bullet)$.

The following lemma appeared as Lemma 2.1 in the third arXiv version of \cite{Randal-Williams2016Resolutionsmodulispaces} (but not in the final, published version) and is also very similar to Lemma 2.14 of \cite{EbertRandal-Williams2019Semi-simplicialspaces} (they assume that the maps $f_n$ are all quasifibrations). We give a complete proof in this appendix, which is self-contained apart from an appeal to one theorem of Dold and Thom (Theorem \ref{thm:doldthom} below).

\begin{lem}\label{l:rwappendix}
The sequence $\lVert \hofib_x(f_\bullet) \rVert \to \lVert X_\bullet \rVert \to X$ is a homotopy fibre sequence over $x$.
\end{lem}

A sequence $A\to B\to C$ is a \emph{homotopy fibre sequence over $c\in C$} if the square
\begin{center}
\begin{tikzpicture}
[x=1mm,y=1mm]
\node (tl) at (0,10) {$A$};
\node (tr) at (20,10) {$B$};
\node (bl) at (0,0) {$\{c\}$};
\node (br) at (20,0) {$C$};
\draw[->] (tl) to (tr);
\draw[->] (bl) to (br);
\draw[->] (tl) to (bl);
\draw[->] (tr) to (br);
\end{tikzpicture}
\end{center}
is homotopy Cartesian, meaning that there exists a homotopy filling the square, and the canonical map (induced by this homotopy) from $A$ to the homotopy pullback of the rest of the diagram is a weak equivalence. This property is invariant under objectwise weak equivalence of diagrams (which will be used in the proof of the lemma, \cf diagram \eqref{eq:reduction-to-special-case} below). The lemma therefore equivalently states that the canonical map
\begin{equation}\label{eq:canonical-map}
\lVert \hofib_x(f_\bullet) \rVert \longrightarrow \hofib_x(\lVert X_\bullet \rVert \to X)
\end{equation}
is a weak equivalence.

Before proving this, we recall that a \emph{quasifibration}, introduced in \cite{DoldThom1958Quasifaserungenundunendliche}, is a map $f\colon X\to Y$ with the property that for each $y\in Y$ the natural map $f^{-1}(y) \to \hofib_y(f)$ is a weak equivalence. Quasifibrations are not closed under taking pullbacks,\footnote{See Bemerkung 2.3 of \cite{DoldThom1958Quasifaserungenundunendliche}.} but they do satisfy a useful local-to-global result, proved in \cite{DoldThom1958Quasifaserungenundunendliche} (see also \cite[Theorem 15.84]{Strom2011Modernclassicalhomotopy}).

\begin{thm}[{\cite[Satz 2.2]{DoldThom1958Quasifaserungenundunendliche}}]\label{thm:doldthom}
Let $f\colon X\to Y$ be a surjective map and $\cU$ a basis for the topology of $Y$. If the restriction $f^{-1}(U)\to U$ is a quasifibration for all $U\in \cU$, then $f$ is a quasifibration.
\end{thm}

\begin{proof}[Proof of Lemma \ref{l:rwappendix}]
For each $n\geq 0$ define $Y_n \coloneqq X^{[0,1]} \times_X X_n$, so the map $f_n$ factors as a weak equivalence followed by a Serre fibration $X_n \to Y_n \to X$. Denote the second map by $g_n$. The spaces $\{Y_n\}_{n\geq 0}$ inherit face maps from $X_\bullet$ and form an augmented semi-simplicial space $Y_\bullet \to X$ whose composition of face maps $Y_n \to X$ is $g_n$. Now choose a CW-approximation $w\colon Z \to X$ for $X$, and write $Z_\bullet \to Z$ for the levelwise pullback of $Y_\bullet$ along $w$. Note that, since $w$ is a weak equivalence and each $g_n \colon Y_n \to X$ is a Serre fibration, the map of augmented semi-simplicial spaces $Z_\bullet \to Y_\bullet$ over $w$ is a levelwise weak equivalence and the compositions of face maps $h_n \colon Z_n \to Z$ are Serre fibrations. Choose a point $z \in Z$ and a path in $X$ from $w(z)$ to $x$. This path induces homotopy equivalences $\hofib_{w(z)}(f_n) \simeq \hofib_x(f_n)$. Summarising, we have the following commutative diagram, where $\sim_\ell$ denotes a levelwise weak equivalence:
\begin{equation}\label{eq:reduction-to-special-case}
\centering
\begin{split}
\begin{tikzpicture}
[x=1.5mm,y=1.2mm]
\node (t1) at (0,20) {$\hofib_x(f_\bullet)$};
\node (t2) at (20,20) {$\hofib_{w(z)}(f_\bullet)$};
\node (t3) at (40,20) {$\hofib_{w(z)}(g_\bullet)$};
\node (t4) at (60,20) {$\hofib_z(h_\bullet)$};
\node (m2) at (20,10) {$X_\bullet$};
\node (m3) at (40,10) {$Y_\bullet$};
\node (m4) at (60,10) {$Z_\bullet$};
\node (b3) at (40,0) {$X$};
\node (b4) at (60,0) {$Z$};
\draw[->] (t1) to node[above,font=\small]{$\sim_\ell$} (t2);
\draw[->] (t2) to node[above,font=\small]{$\sim_\ell$} (t3);
\draw[->] (t4) to node[above,font=\small]{$\sim_\ell$} (t3);
\draw[->] (m2) to node[above,font=\small]{$\sim_\ell$} (m3);
\draw[->] (m4) to node[above,font=\small]{$\sim_\ell$} (m3);
\draw[->] (b4) to node[above,font=\small]{$\sim$} node[below,font=\small]{$w$} (b3);
\draw[->] (t1) to (m2);
\draw[->] (t2) to (m2);
\draw[->] (t3) to (m3);
\draw[->] (t4) to (m4);
\draw[->] (m2) to (b3);
\draw[->] (m3) to (b3);
\draw[->] (m4) to (b4);
\end{tikzpicture}
\end{split}
\end{equation}
In the induced diagram after taking (thick) geometric realisations, all horizontal maps are weak equivalences. Hence it suffices to prove that the sequence $\lVert \hofib_z(h_\bullet) \rVert \to \lVert Z_\bullet \rVert \to Z$ is a homotopy fibre sequence for any $z\in Z$. In other words, we would like to show that the map $t$ in the square
\begin{equation}\label{eq:tblr}
\centering
\begin{split}
\begin{tikzpicture}
[x=1mm,y=1mm]
\node (tl) at (0,15) {$\lVert \hofib_z(h_\bullet) \rVert$};
\node (tr) at (50,15) {$\hofib_z(h)$};
\node (bl) at (0,0) {$\lVert \fib_z(h_\bullet) \rVert$};
\node (br) at (50,0) {$\fib_z(h)$};
\draw[->] (tl) to node[above,font=\small]{$t$} (tr);
\draw[->] (bl) to node[below,font=\small]{$b$} (br);
\draw[->] (bl) to node[left,font=\small]{$\ell$} (tl);
\draw[->] (br) to node[right,font=\small]{$r$} (tr);
\end{tikzpicture}
\end{split}
\end{equation}
is a weak equivalence, where $h$ denotes the induced map $\lVert Z_\bullet \rVert \to Z$. Since each $h_n$ is a Serre fibration, the map $\ell$ in \eqref{eq:tblr} is a weak equivalence. It is easy to see that the map $b$ is a continuous bijection (in fact the same is true for $t$, although we will not use this). To see that it is in fact a homeomorphism, consider the triangle
\begin{equation}\label{eq:triangle}
\centering
\begin{split}
\begin{tikzpicture}
[x=1mm,y=1mm]
\node (l) at (0,0) {$\lVert \fib_z(h_\bullet) \rVert$};
\node (m) at (30,0) {$\fib_z(h)$};
\node (r) at (60,0) {$\lVert Z_\bullet \rVert$};
\draw[->] (l) to node[above,font=\small]{$b$} (m);
\draw[->] (m) to node[above,font=\small]{$i$} (r);
\draw[->] (l.south east) to [out=-15,in=195] node[below,font=\small]{$\lVert i_\bullet \rVert$} (r.south west);
\end{tikzpicture}
\end{split}
\end{equation}
where $i_n \colon \fib_z(h_n) \hookrightarrow Z_n$ and $i\colon \fib_z(h) \hookrightarrow \lVert Z_\bullet \rVert$ are the inclusions. The map $\lVert i_\bullet \rVert$ is injective, and $b$ will be a homeomorphism if and only if $\lVert i_\bullet \rVert$ is an inclusion, i.e.\ a topological embedding. By Lemma \ref{lem:closed-inclusion} (and Example \ref{eg:augmented-ssspace-fibres}), the map $\lVert i_\bullet \rVert$ will be a closed inclusion as long as each $i_n$ is a closed inclusion. So we need to show that $\fib_z(h_n) = h_n^{-1}(z)$ is a closed subset of $Z_n$. But $Z$ is a CW-complex, so in particular its points are closed, so $\{z\}$ is closed in $Z$ and so $h_n^{-1}(z)$ is closed in $Z_n$ by continuity of $h_n$. Hence the map $b$ in \eqref{eq:tblr} is a homeomorphism.

It remains to show that the map $r$ in \eqref{eq:tblr} is a weak equivalence for each $z\in Z$; in other words, we need to show that the map $h\colon \lVert Z_\bullet \rVert \to Z$ is a quasifibration. We will do this using Theorem \ref{thm:doldthom}. Since $Z$ is a CW-complex, it is locally contractible, so we may take $\cU$ to be a basis for its topology consisting of contractible subsets.

To apply Theorem \ref{thm:doldthom}, we need to know that $h$ is surjective, which will be true if and only if $h_0\colon Z_0 \to Z$ is surjective. But we may assume without loss of generality that $f_0\colon X_0\to X$ is $\pi_0$-surjective.\footnote{Suppose that Lemma \ref{l:rwappendix} holds with this assumption and let $X_\bullet$ be any augmented semi-simplicial space with basepoint $x\in X=X_{-1}$. Define $X^\prime$ to be the smallest union of path-components of $X$ containing the image of $\lVert X_\bullet \rVert$ (which is the same as the image of $X_0$) and define $X_n^\prime = X_n$ for $n\geq 0$. If $x\notin X^\prime$ then the lemma is vacuously true (since both $\hofib_x(\lVert X_\bullet \rVert \to X)$ and $\lVert \hofib_x(f_\bullet) \rVert$ are empty). Otherwise, applying the lemma to $X_{\bullet}^\prime$ we see that $\lVert \hofib_x(f_\bullet) \rVert \to \lVert X_\bullet \rVert \to X^\prime$ is a homotopy fibre sequence over $x$. Adding disjoint path-components to the third space in a homotopy fibre sequence does not affect the property of being a homotopy fibre sequence, so $\lVert \hofib_x(f_\bullet) \rVert \to \lVert X_\bullet \rVert \to X$ is also a homotopy fibre sequence over $x$.} By construction, this is equivalent to surjectivity of the map $g_0\colon Y_0\to X$. Now $h_0$ is the pullback of $g_0$ along $w$, and is therefore also surjective.

Let $U\in \cU$. We need to show that the restriction $h|_U\colon h^{-1}(U)\to U$ is a quasifibration. Let $z\in U$. Since each $h_n\colon Z_n \to Z$ is a Serre fibration, so is its restriction $h_n|_U\colon h_n^{-1}(U)\to U$. Hence we have weak equivalences
\[
h_n^{-1}(z) \xrightarrow{\sim} \hofib_z(h_n|_U) \xrightarrow{\sim} h_n^{-1}(U),
\]
where the second weak equivalence is due to the fact that $U$ is contractible. We therefore have a levelwise weak equivalence $h_{\bullet}^{-1}(z) \to h_{\bullet}^{-1}(U)$, and hence a weak equivalence on thick geometric realisations $\lVert h_{\bullet}^{-1}(z) \rVert \to \lVert h_{\bullet}^{-1}(U) \rVert$. But this is the inclusion $h^{-1}(z)\to h^{-1}(U)$, so the composition of the maps
\[
h^{-1}(z) \longrightarrow \hofib_z(h|_U) \longrightarrow h^{-1}(U)
\]
is a weak equivalence. The right-hand map is a weak equivalence since $U$ is contractible, so by 2-out-of-3, the left-hand map is also a weak equivalence. Since $z\in U$ was arbitrary, we have shown that $h|_U$ is a quasifibration.

Theorem \ref{thm:doldthom} now tells us that $h$ is a quasifibration, and so the map $r$ in \eqref{eq:tblr} is a weak equivalence. We showed above that $\ell$ is a weak equivalence and $b$ is a homeomorphism, so the map $t$ is also a weak equivalence. Thus the sequence $\lVert \hofib_z(h_\bullet) \rVert \to \lVert Z_\bullet \rVert \to Z$ is a homotopy fibre sequence for any $z\in Z$. By diagram \eqref{eq:reduction-to-special-case} this implies that $\lVert \hofib_x(f_\bullet) \rVert \to \lVert X_\bullet \rVert \to X$ is a homotopy fibre sequence for any $x\in X$, as required.
\end{proof}

%%%%%%%%%%%%%%%%%%%%%%%%%%%%%%%%%%%%%%%%%%%%%%%%%%%%%%%%%%%%%%%%%%%%%%%%%%%%%%%%%%
%%%%%%%%%%%%%%%%%%%%%%%%%%%%%%%%%%%%%%%%%%%%%%%%%%%%%%%%%%%%%%%%%%%%%%%%%%%%%%%%%%

\phantomsection
\addcontentsline{toc}{section}{References}
\renewcommand{\bibfont}{\normalfont\small}
\setlength{\bibitemsep}{0pt}
\printbibliography

\vfill

\noindent {\itshape Institutul de Matematică Simion Stoilow al Academiei Române,
21 Calea Griviței, 010702 București, România}

\noindent {\tt mpanghel@imar.ro}

\end{document}